\documentclass[11pt,reqno]{amsart}    

\usepackage[margin=1in]{geometry}
\usepackage{amsfonts,amsmath,amssymb,enumerate}    
\usepackage{amsthm, bm}    
\usepackage{mathrsfs}
\usepackage{setspace}
\usepackage{comment}

\onehalfspace
\usepackage{tikz}    
\usetikzlibrary{arrows,matrix,snakes}
\usetikzlibrary{decorations.markings,shapes.geometric,shapes.misc}    

\numberwithin{equation}{section}    

\theoremstyle{plain}
\newtheorem{thm}{Theorem}[section]
\newtheorem*{thm*}{Theorem}
\newtheorem{lem}[thm]{Lemma}
\newtheorem{prop}[thm]{Proposition}
\newtheorem{cor}[thm]{Corollary}
\theoremstyle{definition}

\newtheorem{exmp}[thm]{Example}
\theoremstyle{remark}
\newtheorem{rem}[thm]{Remark}

\newtheorem*{rem*}{Remark}

\newcommand{\be}{\begin{equation}}    
\newcommand{\ee}{\end{equation}}    
\newcommand{\beu}{\begin{equation*}}    
\newcommand{\eeu}{\end{equation*}}    
\newcommand{\bea}{\begin{eqnarray}}    
\newcommand{\eea}{\end{eqnarray}}    
\newcommand{\beaa}{\begin{eqnarray*}}    
\newcommand{\eeaa}{\end{eqnarray*}}    
\newcommand{\bmx}{\begin{pmatrix}}    
\newcommand{\emx}{\end{pmatrix}}

\newcommand{\del}{\partial}    
\newcommand{\g}{{\mathfrak g}}    
\newcommand{\n}{{\mathfrak n}}    
\newcommand{\h}{{\mathfrak h}}

\newcommand{\mf}{\mathfrak}
\newcommand{\mc}{\mathcal}    

\newcommand{\gh}{{\widehat \g}}

\newcommand{\D}{{\mathcal D}}

\newcommand{\nn}{\nonumber}

\newcommand{\8}{{\infty}}

\newcommand{\Z}{{\mathbb Z}}

\newcommand{\C}{{\mathbb C}}

\newcommand{\id}{{\mathrm{id}}}

\newcommand{\MM}{\mathbb{M}}

\newcommand{\M}{M}

\newcommand{\im}{\textup{Im}\,}

\newcommand{\goi}[2]{=}    
\newcommand{\Hom}{\mathrm{Hom}}

\newcommand{\on}{.}    

\newcommand{\Cx}{\mathbb C^\times}

\newcommand{\btp}{\begin{tikzpicture}[baseline=0pt,scale=0.9,line width=0.25pt]}    
\newcommand{\etp}{\end{tikzpicture}}    

\newcommand{\Roff}{\color{black}}

\newcommand{\path}{\longrightarrow}

\newcommand{\at}[2]{\underset{\substack{$ $ \\ \tikz{\draw (0,0) -- (0,.13);\draw[black,fill=white] (0,.13) circle (1pt); } \\[-.2mm] {#1}}}{\smash{#2}}}
\newcommand{\atp}[2]{\underset{\substack{$ $ \\ \tikz{
\draw[->] (0,0) -- (0,.13);} \\[-.2mm] {#1}}}{\smash{#2}}}

\DeclareMathOperator{\res}{res}

\DeclareMathOperator{\Span}{span}

\newcommand{\nord}[1]{:\!\! #1 \!\!:}

\newcommand*{\longhookrightarrow}{\ensuremath{\lhook\joinrel\relbar\joinrel\rightarrow}}
\newcommand*{\longtwoheadrightarrow}{\relbar\joinrel\twoheadrightarrow}

\newcommand{\va}{\mathscr V}
\newcommand{\vla}{\mathscr{L}}
\newcommand{\vma}{\mathscr M}
\newcommand{\ueva}{\VV(\vla)}
\newcommand{\ueval}{\VV(\vma)}
\newcommand{\uevaG}{\VV^\Gamma(\vla)}

\newcommand{\genvla}{L}

\newcommand{\algf}[1]{\mathsf{#1}}

\newcommand{\F}{\algf F}

\newcommand{\VV}{{\mathbb V}}
\newcommand{\WW}{{\mathbb W}}
\newcommand{\vac}{|0\rangle}

\newcommand{\cent}{\mathsf c} 
\newcommand{\ie}{\textit{i.e. }}

\DeclareMathOperator{\Ind}{Ind}

\DeclareMathOperator{\End}{End}

\DeclareMathOperator{\Lie}{Lie}
\DeclareMathOperator{\Aut}{Aut}

\newcommand{\U}{\mathsf L}

\begin{document}

\title{Vertex Lie algebras and cyclotomic coinvariants}

\author{Beno\^{\i}t Vicedo}
\author{Charles Young}
\address{School of Physics, Astronomy and Mathematics, University of Hertfordshire, College Lane, Hatfield AL10 9AB, UK.} \email{benoit.vicedo@gmail.com} \email{charlesyoung@cantab.net}

\begin{abstract}
Given a vertex Lie algebra $\vla$ equipped with an action by automorphisms of a cyclic group $\Gamma$, we define spaces of \emph{cyclotomic coinvariants} over the Riemann sphere. These are quotients of tensor products of smooth modules over `local' Lie algebras $\U(\vla)_{z_i}$ assigned to marked points $z_i$, by the action of a `global' Lie algebra $\U^{\Gamma}_{\{ z_i \}}(\vla)$ of $\Gamma$-equivariant functions.

On the other hand, the universal enveloping vertex algebra $\ueva$ of $\vla$ is itself a vertex Lie algebra with an induced action of $\Gamma$. This gives `big' analogs of the Lie algebras above. From these we construct the space of `big' cyclotomic coinvariants, \emph{i.e.} coinvariants with respect to $\U^{\Gamma}_{\{ z_i \}}(\ueva)$. We prove that these two definitions of cyclotomic coinvariants in fact coincide, provided the origin is included as a marked point. As a corollary we prove a result on the functoriality of cyclotomic coinvariants which we require for the solution of cyclotomic Gaudin models in \cite{VY}.

At the origin, which is fixed by $\Gamma$, one must assign a module over the stable subalgebra $\U(\vla)^{\Gamma}$ of $\U(\vla)$. This module becomes a $\ueva$-quasi-module in the sense of Li. As a bi-product we obtain an iterate formula for such quasi-modules.
\end{abstract}

\maketitle
\setcounter{tocdepth}{1}
\tableofcontents

\section{Introduction and overview}

The theory of vertex Lie algebras was introduced in \cite{KacVertex} under the name `conformal algebra' and further developed in \cite{Primc, Dong02vertexlie} (see also \cite[\S 16.1]{FB}). It provides a unifying framework for describing large families of infinite-dimensional Lie algebras with central extensions, including affine Kac-Moody algebras, Heisenberg Lie algebras, and the Virasoro algebra. 

Given a vertex Lie algebra $\vla$, there is a procedure which associates a (genuine) Lie algebra to every choice of commutative algebra $\mathcal{A}$ over $\C$.
For instance, if one takes $\mc A$ to be the algebra $\C((t))$ of formal Laurent series in a formal variable $t$ then one obtains a Lie algebra $\U(\vla)$ which is topologically generated by modes $a(n):= a\otimes t^n$, $a \in \vla$, $n \in \Z$. Its Lie brackets can be specified in terms of generating series of the form $a(x) := \sum_{n \in \Z} a(n) x^{-n-1}\in \U(\vla)[[x,x^{-1}]]$ as
\begin{equation*}
[a(x), b(y)] = \sum_{k \geq 0} \frac{1}{k!} (a_{(k)} b)(y) \partial_y^k \delta(x, y),
\end{equation*}
for any $a, b \in \vla$. 
Here $a_{(k)}b$ is the \emph{$k$-th product} of the elements $a,b\in \vla$: by definition,  a vertex Lie algebra is a vector space equipped with a collection of such products, $\cdot_{(n)}\cdot: \vla\otimes \vla \to \vla$, labelled by the non-negative integers $n \in \Z_{\geq 0}$ and obeying certain axioms -- see \S\ref{sec: VLAs} -- which ensure, in particular, that the sum on the right hand side is finite and that the resulting Lie brackets are skew-symmetric and satisfy the Jacobi identity.

This Lie algebra $\U(\vla)$ can be regarded as a `local' Lie algebra attached to the origin of the complex plane. More generally, we can construct a copy $\U(\vla)_z$ of $\U(\vla)$ attached to any point $z \in \C$ by letting $\mathcal{A}$ be the algebra $\C((t - z))$ of formal Laurent series in the formal local coordinate $t-z$ at $z$. The $\U(\vla)_z$ are all examples of local Lie algebras.
By contrast, consider taking $\mc A$ to be the commutative algebra $\C_{\bm z}^{\infty}(t)$ of formal rational functions in $t$ having a zero at infinity and poles at most in a given finite set of pairwise distinct \emph{marked points}, $\bm z = \{ z_i \}_{i=1}^N$. The resulting Lie algebra, denoted $\U_{\bm z}(\vla)$, is in a sense `global': it is associated to the entire Riemann sphere with marked points, whereas $\U(\vla)_z$ was attached to the formal punctured disc at the point $z$. 
There is a natural embedding of this global Lie algebra $\U_{\bm z}(\vla)$ into the direct sum $\bigoplus_{i=1}^N\U(\vla)_{z_i}$ of the local Lie algebras\footnote{up to a subtlety concerning the identification of central charges, which we ignore in this introduction; see \S\ref{sec: gla}.}, which comes essentially from taking Laurent expansions at the marked points. That means that a tensor product $\bigotimes_{i=1}^N\M_{z_i}$ of modules $\M_{z_i}$ over $\U(\vla)_{z_i}$, $i=1,\dots,N$, pulls back to become a module over the global Lie algebra $\U_{\bm z}(\vla)$. Quotienting by this action, we obtain a space of \emph{coinvariants},
\be \left.\bigotimes_{i=1}^N \M_{z_i} \right/ \U_{\bm z}(\vla).\nn\ee
Spaces of coinvariants (or equivalently their duals, \emph{conformal blocks}) are one of the main objects of study in conformal field theory \cite{BPZ, TUY}, as well as the study of quantum integrable models including the Gaudin model \cite{FFR} and the KZ equations \cite{KZ, EFKbook}. 

\bigskip

In the present paper we generalize $\U_{\bm z}(\vla)$ to a class of global algebras $\U^{\Gamma}_{\bm z}(\vla)$ that are \emph{cyclotomic} -- or, more precisely, \emph{$\Gamma$-equivariant}, where $\Gamma\simeq \Z/T\Z$ is  a copy of the cyclic group of finite order $T\in \Z_{\geq 1}$. Indeed, $\Gamma$ acts on $\C$ by multiplication by roots of unity, and we can make it act on $\vla$ by powers of some fixed automorphism $\sigma:\vla\to\vla$ whose order divides $T$. So we pick our set of marked points $\bm z = \{z_i\}_{i=1}^N$ in $\C$ as before, but now insist that they have disjoint $\Gamma$-orbits \ie $\Gamma z_i \cap \Gamma z_j = \emptyset$ whenever $i\neq j$. Then $\U^{\Gamma}_{\bm z}(\vla)$ will be the subalgebra of $\Gamma$-equivariant elements of the global Lie algebra $\U_{\Gamma \bm z}(\vla)$. It turns out that there is again an embedding of Lie algebras, global into local; the local Lie algebras remain as before, with the exception that if the fixed point $0$ is a marked point, the local Lie algebra attached to it is a copy of the subalgebra  $\U(\vla)^\Gamma$ of $\Gamma$-equivariant elements of $\U(\vla)$. So we obtain spaces of \emph{cyclotomic coinvariants}
\begin{equation}
\left.\bigotimes_{i=1}^N \M_{z_i} \right/ \U^\Gamma_{\bm z}(\vla), \qquad\text{and}\qquad
\left.\bigotimes_{i=1}^N \M_{z_i} \otimes \M_0 \right/ \U^\Gamma_{\bm z,0}(\vla), \label{introcycco}
\end{equation}
where $\M_0$ is a module over $\U(\vla)^\Gamma$. (It will prove helpful to distinguish the case where $0$ is a marked point.) Our specific motivation arises from a companion paper on cyclotomic generalizations of the quantum Gaudin model \cite{VY}.

\bigskip

To state the main result of the present paper, recall that to any vertex Lie algebra $\vla$ is associated its  \emph{universal enveloping vertex algebra},  $\ueva$. By definition $\ueva$ is, first of all, a module over $\U(\vla)$: namely, it is the \emph{vacuum Verma module} induced from a vacuum state $\vac$ which is annihilated by all non-negative modes. See \S\ref{sec: ueva}. But, in addition, it turns out  to have the structure of a \emph{vertex algebra}. The axiomatic definition of vertex algebras -- recalled in \S\ref{sec: VAs} below --  is similar to that of vertex Lie algebras, except that one has an $n$-th product for every integer $n\in \Z$, rather than merely every non-negative integer. A vertex algebra is thus, in particular, a vertex Lie algebra, simply by forgetting about the negative products.\footnote{Note that of course this is not to say that every vertex algebra is the universal envelope of a vertex Lie algebra, which is certainly not the case. For example, lattice vertex algebras are not of this type \cite{FB}.} Hence $\ueva$ is a vertex Lie algebra. 
That means one has analogs of all the Lie algebras discussed above, but with $\vla$ replaced by $\ueva$. For instance, in place of $\U(\vla)$, one has a Lie algebra $\U(\ueva)$ consisting of all modes of all states in $\ueva$. What is more, there is a natural embedding of vertex Lie algebras $\vla \hookrightarrow \ueva$, which in turn gives rise to embeddings of the Lie algebras: both the local Lie algebras, e.g. $\U(\vla)\hookrightarrow \U(\ueva)$,  and the global ones, e.g. $\U_{\bm z}^\Gamma(\vla) \hookrightarrow \U_{\bm z}^\Gamma(\ueva)$. For that reason, we refer to the Lie algebras associated to the original vertex Lie algebra $\vla$ as `little', and those associated to its envelope $\ueva$ as `big'. In this language, then, there are embeddings of Lie algebras given schematically by the following commutative diagram (see Proposition \ref{prop: gz Uz}):
\begin{equation} \label{se}
\begin{aligned}
\begin{tikzpicture}    
\matrix (m) [matrix of nodes, row sep=2em, column sep=3em]    
{
 {local, little} &  {local, big}  \\
{(cyclotomic) global, little} & {(cyclotomic) global, big.} \\
};
\path[right hook->]
(m-1-1) edge (m-1-2)
(m-2-1) edge (m-2-2)
(m-2-1) edge (m-1-1)
(m-2-2) edge (m-1-2);
\end{tikzpicture}
\end{aligned}
\end{equation}
Next, modules over the little local Lie algebras $\U(\vla)_{z_i}$, provided they are \emph{smooth} (see \S\ref{sec: SM}), automatically have the structure of modules over the corresponding big local Lie algebras $\U(\ueva)_{z_i}$. Likewise, smooth modules over $\U(\vla)^\Gamma$ have the structure of modules over $\U(\ueva)^\Gamma$. The reasons for this are rather subtle -- we sketch them below, and the details are in \S\ref{sec: bigmod} --  but the upshot is that $\bigotimes_{i=1}^N \M_{z_i}\otimes \M_0$ becomes a module over the big global Lie algebra $\U^\Gamma_{\bm z,0}(\ueva)$, by pulling back by the right-hand vertical embedding in \eqref{se}. We are thus free to take coinvariants, to form the space $$\left.\bigotimes_{i=1}^N \M_{z_i}\otimes \M_0 \right/ \U^\Gamma_{\bm z,0}(\ueva).$$ One should then ask how the spaces of coinvariants with respect to the big and little global Lie algebras are related. Since $\U^\Gamma_{\bm z,0}(\vla)\hookrightarrow \U^\Gamma_{\bm z,0}(\ueva)$, as in \eqref{se}, there is certainly a surjective linear map  $\left.\bigotimes_{i=1}^N \M_{z_i}\otimes \M_0 \right/ \U^\Gamma_{\bm z,0}(\vla) \twoheadrightarrow\left.\bigotimes_{i=1}^N \M_{z_i} \otimes \M_0 \right/ \U^\Gamma_{\bm z,0}(\ueva)$. On the face of it, one might expect this map to have a large kernel, given that $\U^\Gamma_{\bm z,0}(\ueva)$ is `much bigger' than $\U^\Gamma_{\bm z,0}(\vla)$. The remarkable fact is that, on the contrary, it is \emph{injective}. Indeed, we shall establish the following in Theorem \ref{thm: big=little}.
\begin{thm*}
There is a linear isomorphism
\begin{align}
          \left.\bigotimes_{i=1}^N\M_{z_i}  \otimes \M_0\right/ \U^{\Gamma}_{\bm z, 0}(\ueva) 
&\quad\cong_{\C}\quad \left.\bigotimes_{i=1}^N\M_{z_i}  \otimes \M_0\right/ \U^{\Gamma}_{\bm z, 0}(\vla). \nn
\end{align}
\end{thm*}
It should be stressed that, when $\Gamma\neq \{1\}$, the presence of the module at the origin here is crucial. As we show by way of an explicit example (Example \ref{exmp: need0}), if $\Gamma\neq \{1\}$ then generically
\be \left.\bigotimes_{i=1}^N\M_{z_i}  \right/ \U^{\Gamma}_{\bm z}(\ueva) 
\quad\not\cong_{\C}\quad \left.\bigotimes_{i=1}^N\M_{z_i}  \right/ \U^{\Gamma}_{\bm z}(\vla).\nn\ee 
(This isomorphism does hold when $\Gamma=\{1\}$.)

\bigskip

Now let us describe in more detail our perspective on the vertex algebra structure of $\ueva$, and on ($\Gamma$-coherent quasi-) modules over this vertex algebra. The subject of vertex algebras is of course a large one and there exist many different approaches. The reader is referred to \cite{Borcherds,FLM,FHL,KacVertex,LLbook,FB}. 
In this paper we are greatly influenced by \cite{FB}, in that we stress the global/geometrical meaning of the vertex algebra structure on $\ueva$. In \cite{FB} the authors work over an algebraic curve of arbitrary genus, whereas we restrict to genus zero, \ie to the Riemann sphere.
Furthermore, we work in a fixed global coordinate on the complex plane in contrast to the coordinate independent approach of \cite{FB}.
By paying this price in generality, we are able work in a more elementary language (essentially just that of rational functions and the residue theorem, with no mention of for example vertex algebra bundles, connections, and sheaves). 
Since for many integrable systems of interest the underlying curve is, in fact, just a copy of the complex plane, we hope that the more explicit treatment here may be useful to others, independently of the specifics of the $\Gamma$-equivariant case.

For us, then, spaces of coinvariants as in \eqref{introcycco} are the central objects. 
Starting with the space of coinvariants \eqref{introcycco} -- for definiteness, with a module assigned to the origin -- consider adding an additional non-zero marked point $u$ and assigning to it the $\U(\vla)$-module $\ueva$. This module has the special property that there is a linear isomorphism 
\be \left. \ueva \otimes \bigotimes_{i=1}^N \M_{z_i} \otimes \M_0 \right/ \U_{u,\bm z,0}^\Gamma(\vla)
\cong_\C \left.  \bigotimes_{i=1}^N \M_{z_i} \otimes \M_0 \right/ \U_{\bm z,0}^\Gamma(\vla)\label{VLremoveintro}\ee
-- see Proposition \ref{prop: VL remove}. That means  it is possible to add one, or in fact arbitrarily many, additional non-zero marked points with copies of the module $\ueva$ assigned to them, without actually altering the space of coinvariants. We use the notation 
\be \big[ \atp u A \otimes \atp {z_1}{m_1} \otimes \dots \atp {z_N}{m_N} \otimes \atp {0}{m_0}  \big] \label{introrat}\ee
for classes in the space \eqref{VLremoveintro}. One can regard such a class as a function of $u$, \ie one can consider varying the point $u$ while holding fixed the other marked points and the tensor factors $A$, $m_1,\dots,m_N$, $m_0$. The result is a rational function of $u$, valued in $\left.  \bigotimes_{i=1}^N \M_{z_i} \otimes \M_0 \right/ \U_{\bm z,0}^\Gamma(\vla)$, with poles at the points $\Gamma \bm z \cup \{0\}$. See Proposition \ref{prop: VL rat}. By Laurent expanding this rational function when $u$ is close to one of the points $z_i$, one obtains a Laurent series in $u-z_i$ with coefficients in $\left.  \bigotimes_{i=1}^N \M_{z_i} \otimes \M_0 \right/ \U_{\bm z,0}^\Gamma(\vla)$. It turns out --  Proposition \ref{prop: YMYW} -- that this Laurent series can always be expressed in the form
\be \big[ \atp{z_1}{m_1}\otimes \dots \otimes \atp{z_{i-1}}{m_{i-1}} \otimes Y_M(A,u-z_i)  \on \atp{z_i}{m_i} \otimes \atp{z_{i+1}}{m_{i+1}} \otimes 
\dots \atp{z_N}{m_N} \otimes \atp{0}{m_0} \big] \nn
\ee
for a certain linear map $Y_M(\cdot,u-z_i):\ueva \to \Hom(\M_{z_i}, \M_{z_i}((u-z_i)))$, which depends only on the choice of module $\M_{z_i}$ and not on the other modules or their positions. It is also independent of $\Gamma$ and of the automorphism $\sigma:\vla\to\vla$.  As a special case, one can take $\M_{z_i}$ to be a copy of $\ueva$ itself, to obtain a linear map $Y(\cdot,u-z_i):\ueva \to \Hom(\ueva, \ueva((u-z_i)))$. This map $Y$ is precisely the usual \emph{state-field correspondence} or \emph{vertex operator} map which encodes all the $n$-th products that define the vertex algebra structure of $\ueva$:
\be Y(A,x) B = \sum_{n\in \Z} A_{(n)} B x^{-n-1}. \nn\ee
Meanwhile the linear maps $Y_M$ are the \emph{module maps} that endow each smooth module $\M_{z_i}$ with the structure of a module over the vertex algebra $\ueva$ -- and thence with the structure of a module over the big Lie algebra $\U(\ueva)$; see Proposition \ref{prop: com2}. 

These global definitions of $Y$ and $Y_M$ should be contrasted with the standard approach to defining the vertex algebra structure on $\ueva$ and its module maps. Recall that one usually defines the state-field map or module map first on elementary states $a(-1)\vac$, $a\in \vla$, and then extends it to the whole of $\ueva$ by requiring consistency with the vertex algebra axioms, principally the Borcherds identity and its corollaries (which include the usual iterate formula involving normal ordering of fields). It is a well-known piece of intuition that the three terms in the Borcherds identity -- see e.g. \cite[\S2.3.1]{Fre07} or \cite[Remark 3.1.16]{LLbook} -- should be thought of as the expansions in different asymptotic regimes of some single `global object'. From the present perspective this is manifest: the `global object' in question is the coinvariant
\begin{equation} \label{Bor-intro}
f(u,v) \big[ \atp u A \otimes \atp v B \otimes \atp {z_1}{m_1} \otimes \dots \atp {z_N}{m_N} \otimes \atp {0}{m_0}  \big],
\end{equation}
where $f(u,v)$ is a rational function in $u$ and $v$ with poles only at $u=z_i$, $v=z_i$ and $u=v$. As above, \eqref{Bor-intro} can be regarded as a rational function of $u$ valued in $\left. \ueva \otimes \bigotimes_{i=1}^N \M_{z_i} \otimes \M_0 \right/ \U_{v,\bm z,0}^\Gamma(\vla)$. Then the Borcherds identity is nothing but the Laurent expansion, in $v-z_i$, of the statement of the residue theorem for this rational function. See Proposition \ref{prop: borcherds} and its proof in \S\ref{sec: borcherdsproof}.

Now, in \eqref{introrat} we should also consider taking the point $u$ near the origin. The result is a Laurent series in $u$ of the form 
\be \big[ \atp{z_1}{m_1}\otimes \dots \otimes 
\dots \atp{z_N}{m_N} \otimes Y_W(A,u)\atp{0}{m_0} \big]\nn\ee
for a certain map $Y_W(\cdot,u) : \ueva \to \Hom(\M_0,\M_0((u)))$, which depends on the choice of the $\U(\vla)^\Gamma$-module $\M_0$. This map $Y_W$ turns out to obey the axioms of a \emph{$\Gamma$-coherent quasi-module map}. These were introduced by Li in a series of papers \cite{Li4,Li3,Li5} under the name $(G, \phi)$-coherent quasi-modules.
We will find, as a by-product of this definition of $Y_W$, an \emph{iterate formula} for such quasi-modules: see Theorem \ref{thm: quasi rec} and Corollary \ref{cor: quasi rec}.

\bigskip

Our Theorem \ref{thm: big=little} here has some overlap with results in \cite{FS}, where coinvariants/conformal blocks for \emph{twisted modules} were introduced. Twisted modules and $\Gamma$-coherent quasi-modules are closely related \cite{Li5} and so in a sense our Theorem \ref{thm: big=little} is the special case of Theorem 7.1/8.1 from \cite{FS} in which one takes the algebraic curve to be (a quotient of) $\C$. (See also \cite{FB} Theorem 9.3.3 and Remark 9.3.10.) On the other hand, in \cite{FS} the theorem is proved only for the special case of the Heisenberg vertex algebra and an automorphism of order 2, whereas here we work with the universal envelope of an arbitrary vertex Lie algebra $\vla$ ($\vla$ need not be finitely generated) and with an arbitrary automorphism of finite order. 

\bigskip

This paper is structured as follows. In \S\ref{sec: vlasandllas} we recall the definition of vertex Lie algebras. Coinvariants are defined in \S\ref{sec: mtc}, which leads up to the definitions of the maps $Y$, $Y_M$ and $Y_W$. After the axioms of vertex algebras are reviewed in \S\ref{sec: VAs}, the `big' Lie algebras are introduced in \S\ref{sec: uevaandbigla}. The main results are stated in \S\ref{sec: mr}. Some of the longer proofs are postponed until \S\ref{sec: proofs}.


\section{Vertex Lie algebras and `little' Lie algebras} \label{sec: vlasandllas}

In this section we begin by recalling the definition and basic properties of a vertex Lie algebra following \cite{KacVertex, Primc}.

\subsection{Vertex Lie algebras}\label{sec: VLAs}

Suppose that $\genvla$ is a complex vector space and $\D$ is a partially defined linear map $\genvla\to \genvla$, \ie suppose that $\D$ is a linear map $\genvla'\to \genvla$ for some subspace $\genvla'\subseteq \genvla$. Let $\vla$ denote the quotient of the free left $\C[\D]$-module $\C[\D]\otimes \genvla$ by the ideal generated by
\be \D \otimes a - 1 \otimes \D a, \qquad a\in \genvla' .\nn\ee
Henceforth we shall often abbreviate $\D^k \otimes a$ as $\D^k a$, for all $a\in \genvla$.
  
We assume that the kernel of the map $\D$ is one-dimensional and not contained in its image. So we can pick a vector $\cent\in \ker \D$ and a subspace $\genvla^o\subset \genvla$ such that $\genvla = \genvla^o \oplus \C\cent \oplus \im \D$. Then 
\be \vla = \vla^o \oplus \C\cent \oplus \im \D ,\label{Lodef}\ee
where $\vla^o:=1\otimes \genvla^o$ (and where $\im \D$ is now the image of $\D$ in $\vla$, rather than $\genvla$).

A \emph{vertex Lie algebra} structure on the $\C[\D]$-module $\vla$ is a collection of \emph{$n^{\rm th}$-products} $\vla \otimes \vla \to \vla$, $(a, b) \mapsto a_{(n)} b$ labelled by $n \in \mathbb{Z}_{\geq 0}$ with the property that for any $a, b \in \vla$ we have $a_{(n)} b = 0$ for $n \gg 0$.\footnote{meaning that for any given $a,b$ there is an $n$ such that $a_{(m)}b=0$ for all $m\geq n$.} 
Moreover, these $n^{\rm th}$-products must satisfy the following set of axioms
\begin{enumerate}[$(i)$]
  \item \emph{Translation} \--- For any $a, b \in \vla$ and $n \in \Z_{\geq 0}$, $(\D a)_{(n)} b = - n a_{(n-1)} b$,
  \item \emph{Skew-symmetry} \--- For any $a, b \in \vla$ and $n \in \Z_{\geq 0}$, we have $\displaystyle a_{(n)} b = - \sum_{k \geq 0} \frac{(-1)^{n + k}}{k!} \D^k \big( b_{(n + k)} a \big)$,
  \item \emph{Commutator} \--- For any $a, b, c \in \vla$ and $m, n \in \Z_{\geq 0}$,
\begin{equation*}
\displaystyle a_{(m)} \big( b_{(n)} c \big) - b_{(n)} \big( a_{(m)} c \big) = \sum_{k \geq 0} \bigg( \!\!\! \begin{array}{c} m\\ k \end{array} \!\!\! \bigg) \big( a_{(k)} b \big)_{(m + n - k)} c.\end{equation*}
\end{enumerate}
It follows from axioms $(i)$ and $(ii)$ that $\cent$ is central with respect to all of the $n^{\rm th}$-products. That is, $\cent_{(n)} b = b_{(n)} \cent = 0$ for any $b \in \vla$ and $n \in \Z_{\geq 0}$. Moreover, it also follows that each $n^{\rm th}$-product $\vla \otimes \vla \to \vla$, $(a, b) \mapsto a_{(n)} b$ is completely determined by its restriction $\vla^o \otimes \vla^o \to \vla$, $(a, b) \mapsto a_{(n)} b$ to the subspace $\vla^o$.

We say a vertex Lie algebra $\vla$ is \emph{$\Z_{\geq 0}$-graded} if $\vla = \bigoplus_{n \in \Z_{\geq 0}} \vla_{(n)}$ as a $\C[\D]$-module with the operator $\D$ of degree $1$, \ie $\D a \in \vla_{(n+1)}$ for any $a \in \vla_{(n)}$, and if, moreover,
\begin{subequations}\label{deg VLA}
\begin{equation} 
\deg \big( a_{(n)} b \big) = \deg a + \deg b - n - 1
\end{equation}
for all homogeneous elements $a,b\in \vla$, and
\be \deg \cent = 0.\ee \end{subequations}
Here $\deg a := m$ for any $a \in \vla_{(m)}$. We shall henceforth always assume that an element $a \in \vla$ is homogeneous when writing $\deg a$. 
Let $L(0)$ denote the degree operator on $\vla$, namely the linear map defined by $L(0) a = \deg (a) \, a$.

Unless otherwise specified we shall restrict attention to $\Z_{\geq 0}$-graded vertex Lie algebras.
Note that for $\vla$ to be $\Z_{\geq 0}$-graded it is sufficient that the underlying vector space $\genvla$ be $\Z_{\geq 0}$-graded, $\D$ be of degree 1 with respect to this grading, and \eqref{deg VLA} hold for any homogeneous $a, b \in \genvla$.

The collection of $n^{\rm th}$-products of a vertex Lie algebra may be conveniently combined into a single linear map
\begin{equation} \label{Y map VLA}
\begin{split}
Y_-(\cdot, x) : \vla &\to \Hom\big( \vla, x^{-1} \vla[x^{-1}] \big),\\
a &\mapsto Y_-(a, x) = \sum_{n \geq 0} a_{(n)} x^{-n-1}, \qquad a_{(n)} \in \End \vla.
\end{split}
\end{equation}
Individual $n^{\rm th}$-products are extracted as
$a_{(n)} b = \res_x x^n Y_-(a, x) b$.
In terms of this map, the above axioms may be rewritten more succinctly, for any $a, b, c \in \vla$, as
\begin{enumerate}[$(i)$]
  \item $Y_-(\D a, x) b = \partial_x Y_-(a, x) b$,
  \item $Y_-(a, x) b = \big( e^{x \D} Y_-(b, -x) a \big)_-$,
  \item $\big[ Y_-(a, x_1), Y_-(b, x_2) \big] c = \big( Y_-( Y_-(a, x_1 - x_2) b, x_2 ) \big)_- c$.
\end{enumerate}
Here we use the notation $F(x)_- = \sum_{n \geq 0} F_n x^{-n-1}$ for the pole part of any $F(x) = \sum_{n \in \Z} F_n x^{-n-1}$. Furthermore, in axiom $(iii)$ we used the standard convention that $(x_1 - x_2)^{-n}$ for $n > 0$ is understood to mean its expansion in small $x_2$.  Thus we have $(x_1-x_2)^{-n}\in \C[x_1^{-1}][[x_2]] \subset \C((x_1))((x_2))$, whereas $(x_2-x_1)^{-n}\in \C[x_2^{-1}][[x_1]] \subset \C((x_2))((x_1))$.
Written in terms of the linear map \eqref{Y map VLA}, the centrality of $\cent$ reads
\begin{equation*}
Y_-(\cent, x) b = Y_-(b, x) \cent = 0,
\end{equation*}
for any $b \in \vla$.

\begin{rem}\label{kerrem}
The reader should note that our definitions of vertex Lie algebras and, below, of vertex algebras are  slightly more restrictive than is standard in that we insist that the kernel of the translation operator $\D$ must be one-dimensional. 
\end{rem}

\subsection{From vertex Lie algebras to Lie algebras}\label{sec: LtoLie}

To any vertex Lie algebra $\vla$ we can associate a genuine Lie algebra as follows \cite{KacVertex}. By the translation and skew-symmetry axioms of a vertex Lie algebra we have $(\D a)_{(0)} b = 0$ and $b_{(0)} (\D a) \in \im \D$ for any $a, b \in \vla$. Thus $\im \D$ is a 2-sided ideal of $\vla$ for the $0^{\rm th}$-product. Moreover, for any $a, b, c \in \vla$ we have $a_{(0)} b + b_{(0)} a \in \im \D$ by the skew-symmetry axiom and $a_{(0)} \big( b_{(0)} c \big) - b_{(0)} \big( a_{(0)} c \big) = \big( a_{(0)} b \big)_{(0)} c$ by the commutator axiom. It follows that the quotient module $\vla / \im \D$ is a Lie algebra with Lie bracket given by
\begin{equation*}
\big[ a + \im \D, b + \im \D \big] := a_{(0)} b + \im \D.
\end{equation*}
As a vector space we have $\vla / \im \D \cong_{\C} \vla^o \oplus \C\cent$. 
More generally, we can associate to a given vertex Lie algebra $\vla$ various infinite-dimensional Lie algebras using the following lemma.

\begin{lem}[{\cite[remark 2.7d]{KacVertex}}] \label{lem: VLA}
Let $\vla$ be a vertex Lie algebra and $\mathcal{A}$ a commutative associative algebra with derivation $\delta$. Then
\begin{equation*}
\Lie_{\mathcal{A}} \vla := \big( \vla \otimes \mathcal{A} \big) \big/ \im \partial
\end{equation*}
where $\partial = \D \otimes 1 + 1 \otimes \delta$, is a Lie algebra. Letting $\rho$ denote the quotient map from $\vla \otimes \mathcal{A}$ to $\Lie_{\mathcal{A}} \vla$, the Lie bracket is given explicitly by
\begin{equation} \label{VLA bracket}
\big[ \rho(a \otimes f), \rho(b \otimes g) \big] = \sum_{n \geq 0} \frac{1}{n!} \rho\big(a_{(n)} b \otimes (\delta^n f) g \big).
\end{equation}
There is a derivation $\D : \Lie_{\mathcal{A}} \vla \to \Lie_{\mathcal{A}} \vla$ defined by
\begin{equation} \label{D lie alg}
\D \rho(a \otimes f) := \rho(\D a \otimes f) = - \rho(a \otimes \delta f).
\end{equation}
The element $\rho(\cent \otimes f)$ is central in $\Lie_{\mathcal{A}} \vla$ for any $f \in \mathcal{A}$.
\qed\end{lem}

An important special case of Lemma \ref{lem: VLA} is when $\mathcal{A} = \C((t))$, the field of formal Laurent series in $t$, with derivation $\delta = \partial_t$. Let
\begin{equation} \label{loop algebra}
\U(\vla) := \Lie_{\C((t))} \vla.
\end{equation}
We use the notation $a(n) := \rho(a \otimes t^n)$ and call this the \emph{$n^{\rm th}$-mode} of $a \in \vla$. When $a$ is homogeneous we will sometimes also use the notation $a[n] := \rho(a \otimes t^{n + \deg a - 1})$. As a vector space $\U(\vla)$ is spanned by formal sums $\sum_{n \geq N} f_n a(n)$ with $a \in \vla$, $f_n \in \C$ and $N \in \mathbb{Z}$ modulo the relation
\begin{equation} \label{D rel VLA}
(\D a)(n) = - n \, a(n - 1).
\end{equation}
If $a \in \vla$ is homogeneous then this relation can also be written as $(\D a)[n] = - (n + \deg a) a[n]$.
Moreover, the Lie bracket \eqref{VLA bracket} on $\U(\vla)$ may be written explicitly as
\begin{equation} \label{VLA bracket loop}
\big[ a(m), b(n) \big] = \sum_{k \geq 0} \bigg( \!\!\! \begin{array}{c} m\\ k \end{array} \!\!\! \bigg) \big(a_{(k)} b \big)(m+n-k).
\end{equation}
In the present case, the linear map $\D : \U(\vla) \to \U(\vla)$ defined by \eqref{D lie alg} reads
\begin{equation} \label{D Lvla}
\D \big( a(n) \big) = (\D a)(n).
\end{equation}

Define a $\Z$-grading on the vector space $\U(\vla)$ by letting
\begin{equation} \label{deg Lvla}
\deg \big( a(n) \big) := \deg a - n - 1,
\end{equation}
for any homogeneous $a \in \vla$ and $n \in \Z$. This does indeed define a grading on $\U(\vla)$ since \eqref{deg Lvla} is compatible with the relations \eqref{D rel VLA}. Moreover, it follows from definition \eqref{deg Lvla} and the explicit form of the Lie bracket \eqref{VLA bracket loop} that $\U(\vla)$ is in fact $\Z$-graded as a Lie algebra. Definition \eqref{deg Lvla} may be equally written as $\deg \big( a[n] \big) = - n$. Note that $\U(\vla)$ is $\Z$-graded even though $\vla$ is $\Z_{\geq 0}$-graded by assumption. Explicitly we have
\begin{equation} \label{Lvla Zgrad}
\U(\vla) = \bigoplus_{n \in \mathbb{Z}} \U_{(n)}(\vla), \qquad
\U_{(n)}(\vla) := \text{span}\, \{ a[- n] \,|\, a \in \vla \; \text{homogeneous} \; \}.
\end{equation}

The following lemma gives a more precise description of $\Lie_{\mathcal{A}} \vla$ as a vector space.
\begin{lem} \label{lem: VLA central}
Let $\mathcal{A}^o$ be a complementary subspace to $\im \delta$ in $\mathcal{A}$, that is $\mathcal{A} = \mathcal{A}^o \oplus \im\, \delta$. Then the quotient map $\rho$ provides a linear isomorphism
\begin{equation*}
\Lie_{\mathcal{A}} \vla \cong_{\C} \vla^o \otimes \mathcal{A}\,\, \oplus\,\, \cent \otimes \mathcal{A}^o,
\end{equation*}
where $\vla^o$ is as in \eqref{Lodef}.
\end{lem}
\begin{proof}
A proof is given in \S\ref{sec: proofVLAcentral}.
\end{proof}

In the particular example $\mathcal{A} = \C((t))$ with $\delta = \partial_t$ considered above, the subspace $\im \partial_t$ may be characterised as the set of $f \in \C((t))$ such that $\res_t f = 0$.
We can therefore take $\mathcal{A}^o = \C t^{-1}$ as a complementary subspace. Applying Lemma \ref{lem: VLA central} to the case \eqref{loop algebra} we have
\begin{equation} \label{L(vla) decomp}
\U(\vla) \cong_{\C} \vla^o \otimes \C((t)) \,\,\oplus\,\, \C \cent \otimes t^{-1}.
\end{equation}
Note in particular that $\cent(n) = 0$ unless $n = -1$. Moreover, $\cent(-1)$ is central in $\U(\vla)$. 

The Lie algebra $\U(\vla)$ admits a canonical polar decomposition
\begin{equation} \label{polar decomp}
\U(\vla) = \U^-(\vla) \dotplus \U^+(\vla),
\end{equation}
where $\dotplus$ means sum of complementary subspaces and where the subalgebras $\U^-(\vla)$ and $\U^+(\vla)$ are defined as
\begin{equation*}
\U^-(\vla) := \rho\big( \vla^o \otimes t^{-1} \C[t^{-1}] \big), \qquad
\U^+(\vla) := \rho\big( \vla^o \otimes \C[[t]] \,\, \oplus \,\, \C \cent \otimes t^{-1} \big).
\end{equation*}
It is clear from the form of the Lie bracket \eqref{VLA bracket loop} that $\U^-(\vla)$ and $\U^+(\vla)$ are both Lie subalgebras.
The subalgebra $\U^-(\vla)$ will play an important role later in characterising the vertex algebra which is canonically associated to the vertex Lie algebra $\vla$.

It follows from the definition \eqref{deg Lvla} of the $\Z$-grading on $\U(\vla)$ that the subspace $\U^-(\vla)$ is in fact $\Z_{\geq 0}$-graded since $\vla$ is. Moreover, we have the following result.
\begin{lem}[{\cite[Theorem 4.6]{Primc}}]      \label{vla Lm iso}
The linear map
\begin{equation*}
i_{\vla} : \vla \longrightarrow \U^-(\vla) \oplus \C \cent(-1), \qquad
a \longmapsto a(-1)
\end{equation*}
is an isomorphism of $\Z_{\geq 0}$-graded vector spaces such that $i_{\vla} \D = \D \, i_{\vla}$.
\begin{proof}
Using \eqref{D rel VLA} we may write any element of $\U^-(\vla) \oplus \C \cent(-1)$ as $a(- n - 1) = \frac{1}{n!} (\D^n a)(-1)$ for $a \in \vla$ and $n \geq 0$. Thus $i_{\vla}$ is surjective. To show injectivity, suppose $a(-1) = 0$ for some $a \in \vla$. That is, $a \otimes t^{-1} = \sum_i \D a_i \otimes f_i + a_i \otimes \partial_t f_i$ for some $f_i \in \C((t))$, $a_i \in \vla$ and where the sum is finite. Writing $f_i = f_i^- + f_i^+$ where $f_i^- \in t^{-1} \C[t^{-1}]$ and $f_i^+ \in \C[[t]]$ we clearly have $\sum_i \D a_i \otimes f_i^+ + a_i \otimes \partial_t f_i^+ = 0$.
By rearranging the finite sum $a \otimes t^{-1} = \sum_i \D a_i \otimes f_i^- + a_i \otimes \partial_t f_i^-$ it can be rewritten as
\begin{equation*}
a \otimes t^{-1} = \sum_{k = 1}^n \D b_k \otimes t^{- k} - k \, b_k \otimes t^{- k - 1},
\end{equation*}
for some $b_k \in \vla$. Comparing lowest powers of $t$ on both sides it follows recursively that $b_k = 0$ for each $1 \leq k \leq n$. Therefore $a = 0$ and so $i_{\vla}$ is injective.

From \eqref{deg Lvla} we get $\deg \big( a(-1) \big) = \deg a$, so that $i_{\vla}$ is degree-preserving. Finally, it follows from \eqref{D Lvla} that $i_{\vla}$ commutes with the respective actions of $\D$ on $\vla$ and $\U^-(\vla) \oplus \C \cent(-1)$.
\end{proof}
\end{lem}

\begin{exmp}[Centrally extended loop algebras] \label{ex: ghat VLA}
Let $\g$ be a simple Lie algebra over $\C$ and $\langle \cdot, \cdot \rangle : \g \times \g \to \C$  a non-degenerate symmetric invariant bilinear form. Let $\genvla = \g \oplus \C K$, $\genvla' = \C K$ and $\D K = 0$. Set $\genvla^o = \g$. Then we have $\vla = \C[\D] \otimes \g \oplus \C K$.
The non-trivial $n^{\rm th}$-products between elements of $\genvla$ are
\begin{equation*}
a_{(0)} b = [a, b], \qquad
a_{(1)} b = \langle a, b \rangle K.
\end{equation*}
for $a, b \in \g$. There is a $\Z_{\geq 0}$-grading given by $\genvla = \genvla_{(0)} \oplus \genvla_{(1)}$ with $\genvla_{(0)} = \C K$ and $\genvla_{(1)} = \g$.
Here the Lie algebra $\vla / \im \D \cong_{\C} \genvla$ is the direct sum of $\g$ and the one-dimensional abelian Lie algebra $\C K$.
The Lie algebra $\U(\vla)$, on the other hand, is a copy of the untwisted affine Lie algebra $\gh$. Indeed, from \eqref{L(vla) decomp} we have $\U(\vla) \cong_{\C} \g \otimes \C((t)) \oplus \C K$ and the Lie bracket relations \eqref{VLA bracket} read
\begin{equation*}
\big[ a[m], b[n] \big] = [a, b][m + n] + m \, \delta_{m + n, 0} \langle a, b \rangle K,
\end{equation*}
where in this equation $K$ is understood to mean $\rho(K \otimes t^{-1}) = K(-1) = K[0]$.
\end{exmp}

\begin{exmp}[Heisenberg Lie algebras] \label{ex: Heis VLA}
Let $\mathfrak{b} = \h \oplus \n$ be a Borel subalgebra of a simple Lie algebra $\g$ over $\C$. 
Consider the vector space $\genvla = \n \oplus \n^{\ast} \oplus \h \oplus \C {\bf 1}$ and the subspace $\genvla' = \C {\bf 1}$, and define $\D {\bf 1} = 0$. Let $\genvla^o = \n \oplus \n^{\ast} \oplus \h$. Then $\vla = \C[\D] \otimes (\n \oplus \n^{\ast} \oplus \h) \oplus \C {\bf 1}$. The only non-trivial $n^{\rm th}$-products between elements of $\genvla$ are given by
\begin{equation*}
a_{(0)} b = (a, b) {\bf 1},
\end{equation*}
for any $a, b \in \n \oplus \n^{\ast}$ where $(\cdot, \cdot)$ is the standard skew-symmetric form on $\n \oplus \n^{\ast}$ defined as follows. If $a$ and $b$ both belong to $\n$ or $\n^{\ast}$ then $(a, b) = 0$ and if $a \in \n$, $b \in \n^{\ast}$ we set $(a, b) = - (b, a) = \langle a, b \rangle$ where $\langle \cdot, \cdot \rangle : \n \otimes \n^{\ast} \to \C$ is the natural pairing.
We define a $\Z_{\geq 0}$-grading by letting $\genvla = \genvla_{(0)} \oplus \genvla_{(1)}$ where $\genvla_{(0)} = \n^{\ast} \oplus \C {\bf 1}$ and $\genvla_{(1)} = \n \oplus \h$.
The Lie algebra $\vla / \im \D \cong_{\C} \genvla$ is the direct sum of the abelian Lie algebra $\h$ and the Heisenberg Lie algebra $\n \oplus \n^{\ast} \oplus \C {\bf 1}$ with relations $[a, b] = (a, b) {\bf 1}$ for any $a, b \in \n \oplus \n^{\ast}$.
On the other hand, the Lie algebra $\U(\vla)$ is the direct sum of the abelian Lie algebra $\h((t))$ and an infinite-dimensional Heisenberg Lie algebra with relations
\begin{equation*}
\big[ a[m], b[n] \big] = \delta_{m + n, 0} (a, b) {\bf 1},
\end{equation*}
for $a, b \in \n \oplus \n^{\ast}$, where by convention here ${\bf 1}$ stands for $\rho\big( {\bf 1} \otimes t^{-1} \big) = {\bf 1}(-1) = {\bf 1}[0]$.
\end{exmp}

\begin{exmp}[Virasoro algebra]
Consider the vector space $\genvla = \C \omega \oplus \C c$ with subspace $\genvla' = \C c$ and define $\D c = 0$. We take $\genvla^o = \C \omega$. Then $\vla = \C[\D] \otimes \omega \oplus \C c$. The non-zero $n^{\rm th}$-products are given by
\begin{equation*}
\omega_{(0)} \omega = \D \omega, \qquad
\omega_{(1)} \omega = 2 \omega, \qquad
\omega_{(3)} \omega = \frac{c}{2}.
\end{equation*}
A $\Z_{\geq 0}$-grading is given by $\genvla = \genvla_{(0)} \oplus \genvla_{(2)}$ where $\genvla_{(0)} = \C c$ and $\genvla_{(2)} = \C \omega$. The two dimensional Lie algebra $\vla / \im \D \cong_{\C} \genvla$ is abelian and the Lie algebra $\U(\vla)$ coincides with the Virasoro algebra whose relations read
\begin{equation*}
\big[ \omega[m], \omega[n] \big] = (m - n) \, \omega[m+n] + \frac{c}{12} (m^3 - m) \delta_{m+n, 0}.
\end{equation*}
\end{exmp}

\subsection{Local Lie algebras $\U(\vla)_{x}$}
Pick a $z\in \C$. By applying Lemma \ref{lem: VLA} to the commutative associative algebra $\mathcal{A} = \C(( t - z ))$ with derivation $\partial_t$, we obtain a Lie algebra
\begin{equation*}
\U(\vla)_{z} := \Lie_{\C((t - z))} \vla = \big( \vla \otimes \C((t - z)) \big) \big/ \im \partial,
\end{equation*}
where $\partial = \D \otimes 1 + 1 \otimes \partial_t$. Given an element $a \in \genvla$ we shall denote the class of $a \otimes (t-z)^n$ in $\U(\vla)_{z}$ as $a(n)_{z}$, or simply $a(n)$ when there is no risk of ambiguity. If $a \in \vla$ is homogeneous we also use the notation $a[n]_{z}$, or simply $a[n]$, to denote the class of $a \otimes (t-z)^{n + \deg a - 1}$ in $\U(\vla)_{z}$. 

Since $t - z$ is a formal coordinate at $z \in \C$, we may regard the Lie algebra $\U(\vla)_{z}$ as a local copy of $\U(\vla)$ attached to the point $z$. In particular, it admits a polar decomposition 
\begin{equation} \label{polar decomp i}
\U(\vla)_{z} = \U^-(\vla)_{z} \dotplus \U^+(\vla)_{z}
\end{equation}
as in \eqref{polar decomp}, where the subalgebras $\U^-(\vla)_{z}$ and $\U^+(\vla)_{z}$ are defined as
\begin{align*}
\U^-(\vla)_{z} := \; &\rho\big( \vla^o \otimes (t - z)^{-1} \C[(t - z)^{-1}] \big),\\
\U^+(\vla)_{z} := \; &\rho\big( \vla^o \otimes \C[[t - z]] \,\, \oplus \,\, \C \cent \otimes (t - z)^{-1} \big).
\end{align*}

\subsection{Global Lie algebra $\U_{\bm z}(\vla)$}\label{sec: gla}

Suppose that $\bm z=\{z_1,\dots,z_N\}$ is a collection of $N\in \Z_{\geq 1}$ pairwise distinct points in $\C$. They are called \emph{marked points}; they will be the points to which we assign modules in \S\ref{sec: mtc} below.
 
In Lemma \ref{lem: VLA} we can take $\mathcal{A}$ to be the algebra $\C^{\infty}_{\bm z}(t)$ of rational functions of $t$ that vanish at infinity and that have poles at most at the points $z_i$.
We use the notation $a f$ for the class of $a \otimes f \in \vla \otimes \C^{\infty}_{\bm z}(t)$ in $\Lie_{\C^{\infty}_{\bm z}(t)} \vla$. There is a natural embedding of Lie algebras
\begin{equation} \label{Lpz embed}
\Lie_{\C^{\infty}_{\bm z}(t)} \vla \longhookrightarrow \bigoplus_{i = 1}^N \U(\vla)_{z_i}  ; \quad a f \longmapsto \big( \rho(a \otimes \iota_{t - z_i} f) \big)_{i = 1}^N.
\end{equation}
where $\iota_{t-z_i} :\C^{\infty}_{\bm z}(t) \to \C((t-z_i)) $ denotes the formal Laurent expansion at $z_i$.

Let $I_{\bm z}$ be the ideal in $\Lie_{\C^{\infty}_{\bm z}(t)} \vla$ defined by 
\begin{equation} \label{Lz ideal I}
I_{\bm z} := \textup{span}_{\C} \left\{ \frac{\cent}{t - z_i} - \frac{\cent}{t - z_j}  
\; \bigg| \;  1\leq i<j\leq N\right\}.
\end{equation}
In Lemma \ref{lem: VLA central} we may choose $\mathcal{A}^o$ to be the span of the rational functions $(t - z_i)^{-1}$ for $i = 1, \ldots, N$, and then consider the subspace $\rho \big( \vla^o \otimes \C^{\infty}_{\bm z}(t) \big)$ of $\Lie_{\C^{\infty}_{\bm z}(t)} \vla$. Let $\U_{\bm z}(\vla)$ denote the image of this subspace under the quotient map 
\be \Lie_{\C^{\infty}_{\bm z}(t)} \vla \to \big( \Lie_{\C^{\infty}_{\bm z}(t)} \vla \big) \big/ I_{\bm z}.\nn\ee
That is,
\begin{equation} \label{vlaz def}
\U_{\bm z}(\vla) := \rho \big( \vla^o \otimes \C^{\infty}_{\bm z}(t) \big) \big/ I_{\bm z} \subset \big( \Lie_{\C^{\infty}_{\bm z}(t)} \vla \big) \big/ I_{\bm z}.
\end{equation}

In the Lie algebra direct sum $\bigoplus_{i = 1}^N \U(\vla)_{z_i}$ there is an ideal $I_N$ spanned by
\begin{equation} \label{LN ideal I}
\cent(-1)_{z_i} - \cent(-1)_{z_j} ,\quad 1\leq i< j\leq N.
\end{equation}
Let $\U(\vla)_{\bm z}$ be the quotient by this ideal:
\begin{equation} \label{LN def}
\U(\vla)_{\bm z} := \left.\bigoplus_{i = 1}^N \U(\vla)_{z_i} \right/ I_P.
\end{equation}
\begin{prop} 
$\U_{\bm z}(\vla)$ is a Lie subalgebra of $\big( \Lie_{\C^{\infty}_{\bm z}(t)} \vla \big) \big/ I_{\bm z}$. Moreover, there is an embedding of Lie algebras 
\begin{equation*}
\U_{\bm z}(\vla) \longhookrightarrow \U(\vla)_{\bm z}.
\end{equation*}
\end{prop}
\begin{proof}
For the first part, we must show that $\U_{\bm z}(\vla)$ closes under the Lie bracket. Suppose $a,b\in \vla^o$ and $f,g\in \mc A=\C_{\bm z}^\8(t)$. In view of \eqref{VLA bracket} we should consider $(a_{(n)} b) (\partial_t^n f) g$.  Note that $(\del_t^nf)g$ has at least a double zero at infinity. Now for any $h \in \C^{\infty}_{\bm z}(t)$ with a double zero at infinity, we have in $\big( \Lie_{\C^{\infty}_{\bm z}(t)} \vla \big) \big/ I_{\bm z}$ that
\begin{equation*}
\cent h = \sum_{i = 1}^N \big( \! \res_{t - z_i} \iota_{t - z_i} h \big) \, \frac{\cent}{t - z_i} = \left( \sum_{i = 1}^N \res_{t - z_i} \iota_{t - z_i} h \right) \, \frac{\cent}{t - z_1} = 0,
\end{equation*}
where in the last equality we used the residue theorem.
In particular this shows, \emph{cf.} Lemma \ref{lem: VLA central}, that the class in $\big( \Lie_{\C^{\infty}_{\bm z}(t)} \vla \big) \big/ I_{\bm z}$ of $\left[ af,bg\right]$ has vanishing component in $\rho(\cent\otimes \mc A^o)/I_{\bm z}$, \ie that it is actually in $\rho \big( \vla^o \otimes \C^{\infty}_{\bm z}(t) \big) \big/ I_{\bm z}  =\U_{\bm z}(\vla)$, as required.

For the `moreover' part it is enough to show that there is an embedding of Lie algebras 
\begin{equation*}
\big( \Lie_{\C^{\infty}_{\bm z}(t)} \vla \big) \big/ I_{\bm z} \longhookrightarrow \U(\vla)_{\bm z}.
\end{equation*}
Now, the ideal $I_N$ is precisely the image of the ideal $I_{\bm z}$, \eqref{Lz ideal I}. Therefore the result follows from the following elementary lemma.
\end{proof}

\begin{lem} \label{simple lem}
Let $U$ and $V$ be Lie algebras and $j : U \rightarrow V$ a Lie algebra homomorphism. Furthermore, suppose $I$ is an ideal in $U$ and $J$ an ideal in $V$ such that $j(U) \cap J = j(I)$. Then there is a natural homomorphism of the quotient Lie algebras, $\bar{\jmath} : U/I \rightarrow V/J$.

Moreover, if $j$ is an embedding, then so is $\bar{\jmath}$.
\begin{proof}
By assumption we have in particular $j(I) \subset J$. Therefore the homomorphism $j : U \hookrightarrow V$ induces a well defined map
of the quotient Lie algebras $\bar{\jmath} : U/I \to V/J$ given by $u + I \mapsto j(u) + J$ for $u \in U$.
This is a homomorphism using the fact that $j$ is,
\begin{equation*}
\big[ \bar{\jmath}(u + I), \bar{\jmath}(v + I) \big] = \big[ j(u) + J, j(v) + J \big] = [j(u), j(v)] + J = j([u, v]) + J = \bar{\jmath}([u, v] + I) = \bar{\jmath}([u + I, v + I]).
\end{equation*}
It remains to show that if $j$ is injective then $\bar{\jmath}$ is injective. So suppose $j(u) + J = J$, which means $j(u) \in J$, then we deduce that $j(u) \in j(I)$ using $j(U) \cap J = j(I)$. If $j$ is injective this implies $u \in I$ so that $u + I = I$, as required.
\end{proof}
\end{lem}

\begin{exmp}
With $\genvla = \g \oplus \C K$ as in example \ref{ex: ghat VLA}, $\vla^o = \g$
so that $\U_{\bm z}(\vla) = \g \otimes \C^{\infty}_{\bm z}(t)$, the Lie algebra of $\g$-valued formal rational functions with poles at the $z_i$ and vanishing at infinity.
\end{exmp}

\subsection{Compatible $\Gamma$-action}

We are interested in vertex Lie algebras equipped with an action of the group $\Gamma$ by automorphisms. 

Recall from \cite{Primc} that, given two vertex Lie algebras $\vla$ and $\vma$, a \emph{homomorphism of vertex Lie algebras} $\varphi : \vla \to \vma$ is a linear map
such that
\begin{equation*}
\varphi \D = \D \varphi, \qquad
\varphi \big( a_{(n)} b \big) = \big( \varphi(a) \big)_{(n)} \big( \varphi(b) \big),
\end{equation*}
for any $n \geq 0$ and $a, b \in \vla$.
In terms of the linear map \eqref{Y map VLA}, the latter property reads
\begin{equation} \label{hom VLA}
\varphi \big( Y_-(a, x) b \big) = Y_- \big( \varphi (a), x \big) \varphi (b),
\end{equation}
for any $a, b \in \vla$. On the left hand side of \eqref{hom VLA}, the map $\varphi$ has been extended component-wise to a linear map $x^{-1} \vla [x^{-1}] \to x^{-1} \vma [x^{-1}]$. If $\vla$ and $\vma$ are graded then $\varphi$ is required to be degree preserving, \ie $\text{deg}\, \varphi(a) = \text{deg}\, a$.

Since $\varphi$ commutes with $\D$, it follows that $\varphi(\cent)\in \C\widetilde\cent$, where $\C\widetilde\cent$ is the kernel of the translation map $\widetilde\D$ of $\vma$. It also follows that $\varphi$ is uniquely specified by its restriction $\varphi : \vla^o \oplus \C\cent \to \vma$ to $\vla^o \oplus \C\cent$. 

For us, an \emph{automorphism of a vertex Lie algebra} $\vla$ is an isomorphism 
\be \varphi : \vla \to \vla\nn\ee 
such that, in addition, 
\be \varphi(\cent) = \cent.\nn\ee 
We denote by $\Aut \vla$ the group of automorphisms of $\vla$.

Let $\vla$ be a $\mathbb{Z}_{\geq 0}$-graded vertex Lie algebra and $\sigma : \vla \to \vla$ an automorphism whose order divides $T$, that is $\sigma^T = \text{id}$. We can define an action of $\Gamma$ on $\vla$ through automorphisms by letting $\omega$ act as $\sigma$. In other words, we have a group homomorphism
\begin{equation*}
\check{R} : \Gamma \longrightarrow \Aut \vla; \quad \alpha \longmapsto \check{R}_{\alpha},
\end{equation*}
given by $\check{R}_{\omega} = \sigma$. It is convenient to introduce a slightly modified action of $\Gamma$ which acts differently on elements of different degrees in $\vla$. Specifically, we define the group homomorphism
\begin{equation} \label{Gamma act VLA}
R : \Gamma \longrightarrow GL(\vla); \quad \alpha \longmapsto R_{\alpha} := \alpha^{L(0)} \check{R}_{\alpha}.
\end{equation}
To show this defines a group homomorphism one uses the fact that $\check{R}_{\alpha} \in \Aut \vla$ is degree preserving.
Unlike $\check{R}_{\alpha} \in \Aut \vla$, however, the map $R_{\alpha}$ is not an automorphism of $\vla$. Instead, it satisfies
\begin{equation} \label{Gamma eq VLA}
R_{\alpha} \big( Y_-(a, x) b \big) = Y_- \big( R_{\alpha} (a), \alpha x \big) R_{\alpha} (b),
\end{equation}
for any $a, b \in \vla$. This follows from the fact that $\check{R}_{\alpha}$ is an automorphism and \eqref{deg VLA}.

Following the terminology introduced in \cite{Li4, Li3} in the context of vertex algebras (which we recall below in \S\ref{sec: Gamma VA}), we shall sometimes refer to a vertex Lie algebra equipped with an action of $\Gamma$ satisfying \eqref{Gamma eq VLA} as a \emph{$\Gamma$-vertex Lie algebra}. \Roff

\begin{exmp}
Let $\vla$ be the vertex Lie algebra of Example \ref{ex: ghat VLA}. Let $\sigma : \g \to \g$ be an automorphism of the simple Lie algebra $\g$ whose order divides $T$ and with respect to which the non-degenerate symmetric invariant bilinear form $\langle \cdot, \cdot \rangle : \g \times \g \to \C$ is invariant, namely $\langle \sigma x, \sigma y \rangle = \langle x, y \rangle$ for any $x, y \in \g$. Extend this to a linear map $\sigma : \genvla \to \genvla$ by letting $\sigma(K) = K$. By construction this map satisfies $\sigma \big( a_{(n)} b \big) = (\sigma a)_{(n)} (\sigma b)$ for any $a, b \in \genvla$ and $n \geq 0$ and extends uniquely to a vertex Lie algebra automorphism $\sigma \in \Aut \vla$.
\end{exmp}

\begin{exmp}
Let $\g$ be a simple Lie algebra over $\C$ and $\sigma : \g \to \g$ an automorphism of $\g$ whose order divides $T$. Then there exists a Cartan decomposition $\g = \n_- \oplus \h \oplus \n$ with the property that $\sigma(\h) = \h$, $\sigma(\n) = \n$ and $\sigma(\n_-) = \n_-$.
Now consider the corresponding vertex Lie algebra $\vla$ constructed as in Example \ref{ex: Heis VLA}. Extend the definition of $\sigma$ to a linear map $\sigma : \genvla \to \genvla$ by letting $\sigma({\bf 1}) = {\bf 1}$ and $\langle \sigma b, a \rangle = \langle b, \sigma^{-1} a \rangle$ for any $a \in \n$ and $b \in \n^{\ast}$. We then have $\sigma \big( a_{(n)} b \big) = (\sigma a)_{(n)} (\sigma b)$ for any $a, b \in \genvla$ and $n \geq 0$ and thus $\sigma$ extends uniquely to a vertex Lie algebra automorphism $\sigma \in \Aut \vla$.
\end{exmp}

\subsection{Cyclotomic global Lie algebra $\U_{\bm z}^\Gamma(\vla)$}\label{sec: tgla}

Let $\sigma \in \Aut \vla$ be an automorphism of the vertex Lie algebra $\vla$ such that $\sigma^T = \text{id}$. Let $R : \Gamma \to GL(\vla)$ be the corresponding group homomorphism defined in \eqref{Gamma act VLA}.

Now, and for the rest of the paper, we require that the points $\bm z$ from \S\ref{sec: gla} are nonzero and have disjoint orbits under the action of $\Gamma$. That is, we insist $z_i\neq 0$ and $\Gamma z_i\cap \Gamma z_j = \emptyset$ for all $1\leq i<j\leq N$.

Consider the algebra $\C^{\infty}_{\Gamma {\bm z}}(t)$ of formal rational functions with poles at most at the points in $\Gamma {\bm z} = \{ \alpha z_i \,|\, \alpha \in \Gamma, i = 1, \ldots, N \}$ and vanishing at infinity.
Applying Lemma \ref{lem: VLA} in the case $\mathcal{A} = \C^{\infty}_{\Gamma {\bm z}}(t)$ with derivation $\partial_t$ we obtain the Lie algebra $\Lie_{\C^{\infty}_{\Gamma {\bm z}}(t)} \vla$. We denote by $a f := \rho(a \otimes f)$ the class in $\Lie_{\C^{\infty}_{\Gamma {\bm z}}(t)} \vla$ of an element $a \otimes f \in \vla \otimes \C^{\infty}_{\Gamma {\bm z}}(t)$.

There is an action of the group $\Gamma$ on $\vla \otimes \C^{\infty}_{\Gamma {\bm z}}(t)$ given for any $\alpha \in \Gamma$ by
\begin{equation} \label{Gamma act LCz}
\alpha . (a \otimes f(t)) := \alpha^{-1} R_{\alpha} a \otimes f(\alpha^{-1} t).
\end{equation}

\begin{lem} \label{lem: twist action}
The action \eqref{Gamma act LCz} descends to an action of $\Gamma$ by automorphisms on $\Lie_{\C^{\infty}_{\Gamma {\bm z}}(t)} \vla$.
\begin{proof}
Recall that $R_{\alpha} = \alpha^{L(0)} \check{R}_{\alpha}$ where $\check{R}_{\alpha} \in \Aut \genvla$. By definition of an automorphism of $\vla$ we have $\D \, \check{R}_{\alpha} = \check{R}_{\alpha} \D$. Moreover, since $\D$ is an operator of degree 1 we have $\D L(0) = (L(0) - 1) \D$, from which it follows that $\D \, R_{\alpha} = \alpha^{-1} \, R_{\alpha} \D$. From the definition $\partial = \D \otimes 1 + 1 \otimes \partial_t$ we find
\begin{align*}
\alpha . \partial (a \otimes f(t)) &= \alpha . (\D a \otimes f(t) + a \otimes \partial_t f(t)) = \alpha^{-1} R_{\alpha} \D a \otimes f(\alpha^{-1} t) + \alpha^{-1} R_{\alpha} a \otimes f'(\alpha^{-1} t)\\
&= \D \, R_{\alpha} a \otimes f(\alpha^{-1} t) + R_{\alpha} a \otimes \partial_t f(\alpha^{-1} t) = \alpha \, \partial \big( \alpha . (a \otimes f(t)) \big).
\end{align*}
In particular it follows that $\Gamma . \im \partial = \im \partial$ and hence the action of $\Gamma$ on $\vla \otimes \C^{\infty}_{\Gamma {\bm z}}(t)$ descends to the quotient $\Lie_{\C^{\infty}_{\Gamma {\bm z}}(t)} \vla$, given explicitly by
$\alpha . \rho(a \otimes f) = \rho\big( \alpha . (a \otimes f) \big)$, \ie
\begin{equation*}
\alpha . (a f(t)) = (\alpha^{-1} R_{\alpha} a) f(\alpha^{-1} t).
\end{equation*}

Finally, consider the action of $\alpha \in \Gamma$ on the Lie bracket of two elements in $\Lie_{\C^{\infty}_{\Gamma {\bm z}}(t)} \vla$ given by \eqref{VLA bracket}. 
A straightforward calculation shows that
\begin{align*}
\alpha . \big[ a f, b g \big] = \big[ \alpha . (a f), \alpha . (b g) \big].
\end{align*}
In other words, $\Gamma$ acts on $\Lie_{\C^{\infty}_{\Gamma {\bm z}}(t)} \vla$ by automorphisms, as required.
\end{proof}
\end{lem}

Let $I_{\Gamma\bm z}$ be the ideal in $\Lie_{\C^{\infty}_{\Gamma {\bm z}}(t)} \vla$ defined by
\begin{equation} \label{LGz ideal I}
I_{\Gamma\bm z} := \textup{span}_{\C} \left\{ \sum_{\alpha \in \Gamma} \frac{\cent}{ t - \alpha z_i} 
   - \sum_{\alpha \in \Gamma} \frac{\cent}{t - \alpha z_j} \; \bigg| \;  1\leq i < j \leq N\right\}.
\end{equation}
Consider the restriction of the quotient map $\Lie_{\C^{\infty}_{\Gamma {\bm z}}(t)} \vla \to \big( \Lie_{\C^{\infty}_{\Gamma {\bm z}}(t)} \vla \big) / I_{\Gamma\bm z}$ to the subspace $\rho\big( \vla^o \otimes \C^{\infty}_{\Gamma {\bm z}}(t) \big)$. We denote its image by $\U_{\Gamma\bm z}(\vla)$, \ie
\begin{equation} \label{quotient VLA}
\U_{\Gamma\bm z}(\vla) := \rho\big( \vla^o \otimes \C^{\infty}_{\Gamma {\bm z}}(t) \big) \big/ I_{\Gamma\bm z}.
\end{equation}
The ideal $I_{\Gamma\bm z}$ is clearly invariant under the action of $\Gamma$, \ie $\Gamma . I_{\Gamma\bm z} = I_{\Gamma\bm z}$, and hence there is a well-defined induced action of $\Gamma$ on the quotient \eqref{quotient VLA}.
We can therefore introduce the subspace of $\Gamma$-invariants in $\U_{\Gamma\bm z}(\vla)$, namely
\begin{equation} \label{gsdef}
\U_{\bm z}^\Gamma(\vla) := \big( \U_{\Gamma\bm z}(\vla) \big)^{\Gamma}.
\end{equation}
We have the following useful linear isomorphism
\begin{equation} \label{Lz LGz}
\U_{\bm z}(\vla) \longrightarrow \U_{\bm z}^\Gamma(\vla) ; \quad a f \longmapsto \sum_{\alpha \in \Gamma} \alpha . (a f),
\end{equation}
which is defined by regarding $a f \in \U_{\bm z}(\vla)$ with $a \in \vla^o$ and $f \in \C^{\infty}_{\bm z}(t) \subset \C^{\infty}_{\Gamma {\bm z}}(t)$ as an element in $\U_{\Gamma\bm z}(\vla)$ and constructing the $\Gamma$-invariant element $\sum_{\alpha \in \Gamma} \alpha . (a f)$.

\begin{prop}\label{prop: LG embed} 
$\U_{\bm z}^\Gamma(\vla)$ is a Lie subalgebra of $\big( \Lie_{\C^{\infty}_{\Gamma\bm z}(t)} \vla \big) \big/ I_{\Gamma\bm z}$. Moreover, there is an embedding of Lie algebras 
\be \U_{\bm z}^\Gamma(\vla) \longhookrightarrow \U(\vla)_{\bm z}. \label{LG embed}\ee
\end{prop}

\begin{proof}
Using the linear isomorphism \eqref{Lz LGz} we write elements of $\U_{\bm z}^\Gamma(\vla)$ in the form $\sum_{\alpha \in \Gamma} \alpha . (a f)$ for some $a \in \vla^o$ and $f \in \C^{\infty}_{\bm z}(t)$. The bracket of two such elements takes the form
\begin{equation} \label{bracket LGz}
\bigg[ \sum_{\alpha \in \Gamma} \alpha . (a f), \sum_{\beta \in \Gamma} \beta . (b g) \bigg] = \sum_{\alpha \in \Gamma} \alpha . \bigg[ a f, \sum_{\beta \in \Gamma} (\alpha^{-1} \beta) . (b g) \bigg].
\end{equation}
Furthermore, for any $h \in \C^{\infty}_{\Gamma \bm z}(t)$ with second order zero at infinity, we have
\begin{align}\label{resuse}
\sum_{\alpha \in \Gamma} &\alpha . (\cent h) = \sum_{\alpha \in \Gamma} \alpha . \left( \sum_{i = 1}^N \sum_{\beta \in \Gamma} \big( \! \res_{t - \beta z_i} \iota_{t - \beta z_i} h \big) \, \frac{\cent}{t - \beta z_i} \right)
= \sum_{i = 1}^N \sum_{\beta \in \Gamma} \big( \! \res_{t - \beta z_i} \iota_{t - \beta z_i} h \big) \, \sum_{\alpha \in \Gamma} \frac{\cent}{t - \alpha \beta z_i} \notag\\
&= \sum_{i = 1}^N \sum_{\beta \in \Gamma} \big( \! \res_{t - \beta z_i} \iota_{t - \beta z_i} h \big) \, \sum_{\alpha \in \Gamma} \frac{\cent}{t - \alpha z_i} = \left( \sum_{i = 1}^N \sum_{\beta \in \Gamma} \res_{t - \beta z_i} \iota_{t - \beta z_i} h \right) \, \sum_{\alpha \in \Gamma} \frac{\cent}{t - \alpha z_1} = 0.
\end{align}
In the first equality we used the fact that $\cent (t - \beta z_i)^n = 0$ unless $n = -1$. In the second last equality we used the definition \eqref{LGz ideal I} of the ideal $I_{\Gamma\bm z}$ and in the last equality the residue theorem for the function $h \in \C^{\infty}_{\Gamma \bm z}(t)$. It follows that the right hand side of \eqref{bracket LGz} lives in $\U_{\bm z}^\Gamma(\vla)$, as required.

As in the usual case ($\Gamma=\{1\}$) there is an embedding of Lie algebras
\begin{equation*}
\Lie_{\C^{\infty}_{\Gamma {\bm z}}(t)} \vla \longhookrightarrow \bigoplus_{i = 1}^N \U(\vla)_{z_i}; \quad
a f \longmapsto \big( \rho(a \otimes \iota_{t - z_i} f) \big)_{i = 1}^N.
\end{equation*}
Note that we take the Laurent expansions only at the points in $\bm z$, not at their images in $\alpha \bm z$, $\alpha\in \Gamma\setminus \{1\}$. Now the image under this embedding of the ideal $I_{\Gamma\bm z}\subset\Lie_{\C^{\infty}_{\Gamma {\bm z}}(t)} \vla$ defined in \eqref{LGz ideal I} coincides with the ideal $I_N \subset \bigoplus_{i = 1}^N \U(\vla)_{z_i}$ spanned by \eqref{LN ideal I}.
(Recall that $\cent(n) = 0$ for all $n\geq 0$.)
Hence, applying Lemma \ref{simple lem} we obtain an embedding of Lie algebras 
\be \big( \Lie_{\C^{\infty}_{\Gamma {\bm z}}(t)} \vla \big) \big/ I_{\Gamma\bm z} \longhookrightarrow \U(\vla)_{\bm z}.\nn\ee 
Restricting the latter first to $\U_{\Gamma\bm z}(\vla)$ and subsequently to the subalgebra of $\Gamma$-invariants \eqref{gsdef} we have an embedding of Lie algebras
$\U_{\bm z}^\Gamma(\vla) \longhookrightarrow \U(\vla)_{\bm z}$, as required.
\end{proof}

The Lie algebra $\U_{\bm z}^{\Gamma}(\vla)$ consists of  `cyclotomic' (or, more precisely, $\Gamma$-equivariant) $\vla^o$-valued formal rational functions $\vla^o \otimes \C^{\infty}_{\Gamma \bm z}(t)$ with poles at $\alpha z_i$ for $i = 1, \ldots, N$, $\alpha \in \Gamma$ and vanishing at infinity. The condition of $\Gamma$-equivariance on a function $F(t) = a \otimes f(t) \in \vla^o \otimes \C^{\infty}_{\Gamma \bm z}(t)$ means that if $a \in \vla^o$ is homogeneous of degree $k + 1$, then $\sigma(F(t)) = \omega^{-k} F(\omega t)$.

\begin{exmp}
In the case of the vertex Lie algebra of example \ref{ex: ghat VLA} generated by $L = \g \oplus \C K$, all elements in $\vla^o = \g$ are of degree one. Therefore $L_{\bm z}^{\Gamma}(\vla)$ is the Lie algebra of $\g$-valued rational functions $F(t) \in \g \otimes \C^{\infty}_{\Gamma \bm z}(t)$ such that $\sigma(F(t)) = F(\omega t)$.
\end{exmp}

\begin{exmp}
For the vertex Lie algebra of example \ref{ex: Heis VLA} generated by $L = \n \oplus \n^{\ast} \oplus \h \oplus \C {\bf 1}$ we have $(\vla^o)_{(0)} = \n^{\ast}$ and $(\vla^o)_{(1)} = \n \oplus \h$. In this case $L_{\bm z}^{\Gamma}(\vla)$ is the commutative Lie algebra consisting or rational functions in $(\n \otimes \n^{\ast} \otimes \h) \otimes \C^{\infty}_{\Gamma \bm z}(t)$ such that $\sigma(F(t)) = F(\omega t)$ for $F(t) \in (\n \otimes \h) \otimes \C^{\infty}_{\Gamma \bm z}(t)$ and $\sigma(F(t)) = \omega F(\omega t)$ for $F(t) \in \n^{\ast} \otimes \C^{\infty}_{\Gamma \bm z}(t)$.
\end{exmp}

\subsection{Poles at the origin, and the Lie algebra $\U_{\bm z,0}^\Gamma(\vla)$}\label{sec: tgla}
In the preceding subsection we assumed that the marked points $\bm z=\{z_1,\dots,z_P\}$ were nonzero. But we shall also need the case in which the origin is a marked point. 

The group $\Gamma$ acts on $\vla\otimes \C^\8_{\Gamma\bm z\cup\{0\}}(t)$ as in \eqref{Gamma act LCz}, and one checks as in Lemma \ref{lem: twist action} that this action descends to an action by automorphisms on $\Lie_{\C^\8_{\Gamma\bm z\cup\{0\}}(t)}\vla$. We then define $\U_{\Gamma\bm z, 0}(\vla)$ to be
\be \U_{\Gamma\bm z, 0}(\vla) := \rho(\vla^o\otimes \C^\8_{\Gamma\bm z\cup\{0\}}(t))\big/ I_{\Gamma\bm z,0},\nn\ee
where $I_{\Gamma\bm z,0}$ is the ideal in $\Lie_{\C^\8_{\Gamma\bm z\cup\{0\}}(t)}\vla$ spanned by
\be
\sum_{\alpha \in \Gamma} \frac{\cent}{ t - \alpha z_i} 
                     - \sum_{\alpha\in \Gamma} \frac{\cent}{t}
=  \sum_{\alpha \in \Gamma} \frac{\cent}{ t - \alpha z_i} 
                     -  \frac{T\cent}{t}, \qquad   1\leq i \leq N,\label{LGz ideal Ip}\ee
and finally define 
\begin{equation} \label{gs0def}
\U_{\bm z,0}^\Gamma(\vla) := \big( \U_{\Gamma\bm z,0}(\vla) \big)^{\Gamma}.
\end{equation}

We now want the analog, for $\U_{\bm z,0}^\Gamma(\vla)$, of the embedding of Lie algebras \eqref{LG embed}. 
Because the origin is distinguished by being the fixed point of the action of $\Gamma$ on $\C$, the appropriate local Lie algebra to introduce is not another copy of $\U(\vla)$ but rather the subalgebra of $\Gamma$-invariant elements,
\begin{equation} \label{Gamma local}
\U(\vla)^\Gamma
\end{equation}
where the action of $\alpha\in \Gamma$ on an element  $a\otimes f(t)$ is given exactly as in \eqref{Gamma act LCz} but with $f(t)$ in $\C((t))$ rather than $\C_{\Gamma\bm z}(t)$. (Again, one checks as in Lemma \ref{lem: twist action} that this action of $\Gamma$ on $\vla\otimes \C((t))$ descends to a well-defined action on $\U(\vla) = \Lie_{\C((t))} \vla$ by automorphism.)
Define
\begin{equation} \label{Lvla z0}
\U(\vla)_{\bm z,0} := \left(\bigoplus_{i = 1}^N \U(\vla)_{z_i} \oplus \U(\vla)^{\Gamma} \right)\Big/ I_{N,0}
\end{equation}
where $I_{N,0}$ is the ideal spanned by the elements $\cent(-1)_{z_i} - T\cent(-1)_0$, $i=1,\dots, N$.

\begin{prop}\label{prop: LG embed prime}
$\U_{\bm z,0}^\Gamma(\vla)$ is a Lie subalgebra of $\big( \Lie_{\C^{\infty}_{\Gamma\bm z\cup\{0\}}(t)} \vla \big) \big/ I_{\Gamma\bm z,0}$. Moreover there is an embedding of Lie algebras 
\be \U_{\bm z,0}^\Gamma(\vla) \longhookrightarrow \U(\vla)_{\bm z,0}. \label{LG embed prime}\ee
\end{prop}
\begin{proof}
The proof is essentially the same as that of Proposition \ref{prop: LG embed}. We still have a surjective linear map
\be \U_{\bm z,0}(\vla) \to \U_{\bm z,0}^\Gamma(\vla);\qquad a f \mapsto \sum_{\alpha\in \Gamma} \alpha \on (a f) ; \nn\ee
though, \emph{cf.} \eqref{Lz LGz}, this map is no longer injective in general. So we may continue to write elements of $\U_{\bm z,0}^\Gamma(\vla)$ in the form $\sum_{\alpha\in \Gamma}\alpha\on(af)$, for some $a f \in \U_{\bm z, 0}(\vla)$. The step \eqref{resuse} becomes
\begin{align}\label{resuseprime}
\sum_{\alpha \in \Gamma} \alpha . (\cent h) &= \sum_{\alpha \in \Gamma} \alpha . \left( \sum_{i = 1}^N \sum_{\beta \in \Gamma} \big( \! \res_{t - \beta z_i} \iota_{t - \beta z_i} h \big) \, \frac{\cent}{t - \beta z_i} +  \big(\res_t h\big) \frac\cent t \right) \notag\\
&= \sum_{i = 1}^N \sum_{\beta \in \Gamma} \big( \! \res_{t - \beta z_i} \iota_{t - \beta z_i} h \big) \, \sum_{\alpha \in \Gamma} \frac{\cent}{t - \alpha z_i} + T\big(\res_t h \big)\frac\cent t \notag\\ 
&= \left( \sum_{i = 1}^N \sum_{\beta \in \Gamma} \res_{t - \beta z_i} \iota_{t - \beta z_i} h + \res_t h \right) \, T\frac\cent t = 0,
\end{align}
where in the final equality we used the residue theorem for $h \in \C^\8_{\Gamma \bm z\cup\{0\}}(t)$.

There is an embedding of Lie algebras 
$\Lie_{\C^{\infty}_{\Gamma \bm z\cup\{0\} }(t)} \vla \hookrightarrow \bigoplus_{i = 1}^N \U(\vla)_{z_i} \oplus \U(\vla)$
by taking Laurent expansions at the points $z_1,\dots,z_N$ and $0$. As before one checks that the conditions of Lemma \ref{simple lem} are satisfied, so that we obtain an embedding of Lie algebras $\big( \Lie_{\C^{\infty}_{\Gamma {\bm z\cup\{0\}}}(t)} \vla \big) \big/ I_{\Gamma\bm z,0} \hookrightarrow 
\bigoplus_{i = 1}^N \U(\vla)_{z_i}  \oplus \U(\vla)\big/ I_{N,0}$.
Restricting the latter first to $\U_{\Gamma\bm z,0}(\vla)$ and subsequently to the subalgebra \eqref{gs0def} of $\Gamma$-invariants we have an embedding of Lie algebras
$\U_{\bm z,0}^\Gamma(\vla) \longhookrightarrow \U(\vla)_{\bm z,0}$, as required.
\end{proof}

To describe the subalgebra \eqref{Gamma local} of $\Gamma$-invariant elements in $\U(\vla)$ more explicitly, recall that the Lie algebra $\U(\vla)$ is spanned by formal sums $\sum_{n \geq N} f_n a(n)$ with $a \in \vla$, $f_n \in \C$ and $N \in \Z$. Note that we have a natural surjective map
\begin{equation*}
\U(\vla) \longtwoheadrightarrow \U(\vla)^{\Gamma}
\end{equation*}
given by $\sum_{n \geq N} f_n \, a(n) \mapsto \sum_{n \geq N} f_n \, a^{\Gamma}(n)$, where
\begin{equation*}
a^{\Gamma}(n) := \sum_{\alpha \in \Gamma} \alpha . (a(n)) = \sum_{\alpha \in \Gamma} \alpha^{-n-1} (R_{\alpha} a)(n).
\end{equation*}
As a vector space, $\U(\vla)^{\Gamma}$ is therefore spanned by formal sums $\sum_{n \geq N} f_n \, a^{\Gamma}(n)$ with $a \in \vla$, $f_n \in \C$ and $N \in \Z$.

\begin{exmp}
Consider the vertex Lie algebra of example \ref{ex: ghat VLA}. Then $\U(\vla)^{\Gamma}$ is the twisted loop algebra, spanned by formal sums $\sum_{n \geq N} f_n \, a^{\sigma}(n)$ of elements of the form $a^{\sigma}(n) = \sum_{k \in \mathbb{Z}_T} \omega^{- k n} (\sigma^k a)(n)$.
\end{exmp}

\section{Modules and cyclotomic coinvariants}\label{sec: mtc}

A module $\M$ over $\U(\vla)$ has \emph{level} $k\in \C$ if $\cent(-1) \on v = k v$ for all $v\in \M$. \emph{In what follows, we assume without further comment that all modules over $\U(\vla)$ are of level 1}. That is we assume that 
\begin{equation} \label{l1}
\cent(-1) \on v = v \quad\text{for all} \quad v\in \M.
\end{equation}
On the other hand, \emph{we assume that all modules over $\U(\vla)^{\Gamma}$ are of level $\frac{1}{T}$} so that a module over the direct sum $\bigoplus_{i=1}^N \U(\vla)_{z_i} \oplus \U(\vla)^{\Gamma}$ pulls back to a module over $\U(\vla)_{\bm z, 0}$ defined in \eqref{Lvla z0}. That is, for any module $M_0$ over $\U(\vla)^{\Gamma}$ we assume that
\begin{equation} \label{l0}
\cent(-1) \on v = \frac{1}{T} v \quad\text{for all} \quad v\in \M_0.
\end{equation}

\subsection{Spaces of cyclotomic coinvariants}\label{sec: indg}
Let $\M_{z_i}$ be a module over $\U(\vla)_{z_i}$ for each $i = 1, \ldots, N$.
Then the tensor product
$\bigotimes_{i=1}^N\M_{z_i}$
is naturally a module over the Lie algebra $\U(\vla)_{\bm z}$ defined in \eqref{LN def}, because condition \eqref{l1} ensures that the action of $\bigoplus_{i = 1}^N \U(\vla)_{z_i}$ descends to an action of the quotient $\U(\vla)_{\bm z} := \bigoplus_{i = 1}^N \U(\vla)_{z_i}\big/ I_N.$ After pull-back by the embedding of Proposition \ref{prop: LG embed}, it becomes a module over $\U^\Gamma_{\bm z}(\vla)$.
The space of \emph{coinvariants with respect to $\U_{\bm z}^\Gamma(\vla)$} (or \emph{cyclotomic coinvariants}) is by definition the quotient
\be\label{tc1} 
\left.\bigotimes_{i=1}^N\M_{z_i}  \right/ \U_{\bm z}^\Gamma(\vla) :=\left.\bigotimes_{i=1}^N\M_{z_i}  \right/ \U_{\bm z}^\Gamma(\vla) \on \left(\bigotimes_{i=1}^N\M_{z_i} \right),
\ee
where $\U_{\bm z}^\Gamma(\vla) \on \M := \Span_\C \big\{ f \on m \, \big| \, f \in \U_{\bm z}^\Gamma(\vla), m \in \M \big\}$ for any $\U_{\bm z}^\Gamma(\vla)$-module $\M$.
Similarly, if $\M_{0}$ is a module over $\U(\vla)^\Gamma$ then in view of Proposition \ref{prop: LG embed prime} we have a space of coinvariants with respect to $\U^\Gamma_{\bm z,0}(\vla)$, 
\be\label{tc2}
\left.\bigotimes_{i=1}^N\M_{z_i} \otimes \M_{0}\right/ \U^\Gamma_{\bm z,0}(\vla).
\ee

The space \eqref{tc2}, for example, is spanned by classes of the form 
\be [  \atp{z_1}{m_1} \otimes \atp{z_2}{m_2} \otimes \dots \otimes \atp{z_N}{m_N} 
 \otimes \atp 0 {m_0}]\nn
\ee
with $m_i\in \M_{z_i}$, $i=1,\dots,N$ and $m_0\in \M_0$. 
Here, and in what follows, we decorate such classes with arrows as a reminder of the points to which the tensor factors are assigned.

\subsection{Vacuum Verma module $\ueva$}\label{sec: ueva}
Recall the polar decomposition $\U(\vla) = \U^-(\vla) \dotplus \U^+(\vla)$
given in \eqref{polar decomp}. Let $\C\vac$ be the one-dimensional $\U^+(\vla)$-module in which 
\begin{equation} \label{vacdef}
\cent(-1) \vac = \vac \quad\text{and}\quad
a(n) \vac = 0 \quad\text{for all $a\in \vla$ and all $n\geq 0$}.
\end{equation}
Then the \emph{vacuum Verma module} over $\U(\vla)$, denoted $\ueva$, is defined to be the $\U(\vla)$-module induced from $\C\vac$:
\begin{equation} \label{Vvla def}
\ueva := \Ind^{\U(\vla)}_{\U^+(\vla)} \C \vac = U(\U(\vla)) \otimes_{U(\U^+(\vla))} \C \vac.
\end{equation}
Given the polar decomposition \eqref{polar decomp} and using the Poincar\'e-Birkhoff-Witt theorem we have $U(\U(\vla)) \cong_{\C} U(\U^-(\vla)) \otimes U(\U^+(\vla))$, so that as a vector space
\begin{equation} \label{Vvla vect}
\ueva \cong_{\C} U(\U^-(\vla)) \vac \cong_{\C} U\big( \rho(\vla^o \otimes t^{-1} \C[t^{-1}]) \big).
\end{equation}
By this identification, $\ueva$ as a vector space inherits the natural ascending filtration on $U(\U^-(\vla))$. We use the word \emph{depth} to refer to this filtration. So the following vector has depth $j$:
\begin{equation*}
 a_1(- n_1) \ldots a_j(- n_j) \vac ,
\end{equation*}
with $a_i \in \vla^o$ and $n_i \in \Z_{> 0}$ for $i = 1, \ldots, j$. Moreover $\ueva$ inherits the $\Z_{\geq 0}$-grading of $\U^-(\vla)$ if we let $\deg \vac := 0$. The state above has grade $\deg a_1 + \ldots + \deg a_j + n_1 + \ldots + n_j - j$, for example.

There is an injection
\begin{equation} \label{vla in VV}
\vla \longhookrightarrow \ueva,
\end{equation}
sending $a \mapsto a(-1) \vac$ for all $a\in \vla$ (so in particular $\cent \mapsto \vac$, using \eqref{vacdef}).

In the above construction of cyclotomic coinvariants we may assign this module to a point $u \in \C^{\times} \setminus \Gamma \bm z$. That is, we make $\ueva$ into a module over the local copy $\U(\vla)_u$ of $\U(\vla)$ at $u$ by means of the  isomorphism 
\be \U(\vla)_u\xrightarrow\sim \U(\vla); \quad \rho(a\otimes (t-u)^n) \mapsto \rho(a\otimes t^n).\label{lgi}\ee
We may then take the space of coinvariants with respect to $\U^{\Gamma}_{u, \bm z}(\vla)$ (meaning, strictly, $\U^\Gamma_{\{u\}\cup \bm z}(\vla)$). 
The following important proposition says that adding the module $\ueva$ in this way does not modify the space of coinvariants.
\begin{prop}\label{prop: VL remove}
For each $u \in \Cx \setminus  \Gamma \bm z $ we have a linear isomorphism
\begin{equation*}
\left.\ueva \otimes \bigotimes_{i = 1}^N \M_{z_i} \right/ \U^\Gamma_{u,\bm z}(\vla) \xrightarrow{\sim_\C} \bigotimes_{i = 1}^N \M_{z_i} \bigg/ \U^\Gamma_{\bm z}(\vla)
\end{equation*}
(together with the obvious analogs for the space of coinvariants \eqref{tc2}). 
\end{prop}
\begin{proof} 
We first observe that
\begin{equation} \label{L LG decomp}
\U(\vla)_u = \U_u^{\Gamma}(\vla) \dotplus \U^+(\vla)_u,
\end{equation}
where we identify $\U_u^{\Gamma}(\vla)$ with its image $\iota_{t - u} \U_u^{\Gamma}(\vla)$ in $\U(\vla)_u$ under the embedding $\U_u^{\Gamma}(\vla) \hookrightarrow \U(\vla)_u$ of Proposition \ref{prop: LG embed}.
To see this, given any $\rho(a \otimes f(t-u)) \in \U(\vla)_u$, with $a \in \vla$ and $f(t-u) \in \C((t- u))$, let $\rho(a^o \otimes f_-(t-u)) \in \U^-(\vla)_u$, with $a^o \in \vla^o$ and $f_-(t-u) \in (t-u)^{-1} \C[(t-u)^{-1}]$, denote its component along $\U^-(\vla)_{u}$ with respect to the polar decomposition $\U(\vla)_u = \U^-(\vla)_u \dotplus \U^+(\vla)_u$ defined in \S\ref{sec: LtoLie}. Then $\rho(a \otimes f(t-u)) \in \U(\vla)_u$ can be uniquely written as a sum of $\iota_{t- u} F(t)$ where $F(t) = \sum_{\alpha \in \Gamma}(\alpha^{-1} R_{\alpha} a^o) f_-(\alpha^{-1} t-u) \in \U^{\Gamma}_u(\vla)$ and $\rho(a \otimes f(t-u)) - \iota_{t - u} F(t) \in \U^+(\vla)_u$.

By \eqref{lgi} and \eqref{Vvla def} we have $\ueva =  U(\U(\vla)_u) \otimes_{U(\U^+(\vla)_u)} \C \vac$, so \eqref{L LG decomp} implies that as a vector space
\begin{equation} \label{vli}
\ueva \cong_{\C} U(\U_u^{\Gamma}(\vla)) \vac.
\end{equation}

Now, let $V$ be a module over a Lie algebra $\mf c$, and suppose that $\mf c$ decomposes as a vector space into the direct sum $\mf c = \mf a \dotplus \mf b$ of two Lie subalgebras $\mf a$ and $\mf b$.
We have $(\mf a \dotplus \mf b)\on V = \mf a\on V + \mf b\on V$ and the natural inclusions of vector spaces $\mf a \on V \subseteq \mf a \on V + \mf b\on V \subseteq V$. 
Hence, using the third isomorphism theorem, 
\be V \big/ \mf c \on V = V\big/ (\mf a \on V + \mf b\on V) \cong_{\C} \big( V / \mf a\on V \big) \big/ \big( (\mf a \on V + \mf b\on V)/ \mf a\on V \big).\label{visoms}\ee 
In the case at hand  $\mf a = \U^\Gamma_u(\vla)$, $\mf b = \U^\Gamma_{\bm z}(\vla)$ and the decomposition $\mf c = \U^\Gamma_{u,\bm z}(\vla)$, and $\mf c = \mf a \dotplus \mf b$ is given by the partial fraction expansion. 
By \eqref{vli} we have
\begin{equation*} \ueva \otimes \bigotimes_{i=1}^N\M_{z_i} 
      \cong_\C U(\mf a)\vac \otimes M =: V \end{equation*}
where for brevity we introduce $M :=  \bigotimes_{i=1}^N\M_{z_i}$. This defines a $\mf c$-module structure on $V$, and with respect to this structure
it is enough to show that 
\begin{align} 
V  &= \mf a \on V \dotplus \vac \otimes M \label{c1}\\
\mf a \on V + \mf b \on V &= \mf a \on V \dotplus \vac\otimes \mf b\on M \label{c2}
\end{align}
for then \eqref{visoms} gives $V/\mf c \on V \cong_{\C} M / \mf b \on M$, which is the required result. 

Note that as a module over $U(\mf a)$, $U(\mf a)\vac\otimes M$ is free with basis $\{\vac\otimes m_i\}$, where $\{m_i\}$ is any $\C$-basis of $M$. It follows that 
\be \label{decVLM}
U(\mf a) \vac \otimes M  \\
= \mf a \on \left(U(\mf a)\vac\otimes M\right) \dotplus \vac \otimes M\ee
which is \eqref{c1}. It remains to show \eqref{c2}. 
Let $\{0\} = \F_{-1} \subset \C = \F_0 \subset \mf a = \F_1 \subset \F_2 \subset \dots$ be the natural ascending filtration of $U(\mf a)$. 
Note that in our setting $\mf b \on \vac = 0$ (elements of $\mf b$ have purely Taylor expansions at $u$ and hence annihilate $\vac$). 
Hence $\mf b\on (\vac\otimes M) = \vac\otimes \mf b \on M\subset \vac\otimes M$, which has zero intersection with $\mf a\on V$ by \eqref{c1}. 
So to show \eqref{c2} it is enough to show that 
\be \mf a\on V + \mf b \on V \subset \mf a \on V + \mf b \on (\vac\otimes M)\label{summ}\ee
(the reverse containment being obvious). Now, for each $k\in \Z_{\geq 0}$,
\begin{align*} \mf b \on (\F_{k+1}\vac \otimes M) = \mf b \on (\mf a \F_{k}\vac \otimes M)
                         &\subset \mf b \on \mf a \on (\F_k \vac\otimes M) + \mf b \on (\F_k \vac\otimes \mf a\on M)\\
           &\subset \mf b \on \mf a \on (\F_k \vac\otimes M) + \mf b \on (\F_k\vac \otimes M)\\
           &\subset \mf a \on \mf b \on (\F_k \vac\otimes M) + \mf a \on (\F_k\vac\otimes M) + \mf b \on (\F_k\vac \otimes M)\\
           &\subset \mf a \on V + \mf b \on (\F_k \vac\otimes M)
\end{align*}
where we used the fact that $[\mf b, \mf a] \subset \mf a + \mf b$. Hence by induction  $\mf b\on (\F_k \vac\otimes M) \subset \mf a \on V + \mf b \on (\vac \otimes M)$ for all $k$. Thus $\mf b\on V  \subset \mf a \on V + \mf b \on (\vac \otimes M)$ and we have \eqref{summ} as required.  
\end{proof}

\subsection{Smooth modules, and rational behaviour of coinvariants}\label{sec: SM}
A module $\M$ over the Lie algebra $\U^+(\vla)$ of \eqref{polar decomp} is called \emph{smooth} if for all $v\in\M$ and all $a\in \vla$,
\be a(n) \on v = 0 \qquad\text{for}\quad n \gg 0.\nn\ee

All modules over $\U(\vla)$ are modules over $\U^+(\vla)\hookrightarrow \U(\vla)$, so the definition of smooth makes sense in particular for $\U(\vla)$-modules. 

By construction, $\ueva$ is a smooth module over $\U(\vla)$. 

Now, for each $u\in \Cx \setminus \Gamma\bm z$ we have a linear map
\be F_u : \ueva \otimes \bigotimes_{i = 1}^N \M_{z_i}  \twoheadrightarrow 
   \left.\bigotimes_{i = 1}^N \M_{z_i}  \right/ \U^\Gamma_{\bm z}(\vla)\nn\ee
given by first taking coinvariants with respect to $\U^\Gamma_{u,\bm z}(\vla)$ and then using the linear isomorphism of Proposition \ref{prop: VL remove}. We can consider varying the point $u$ while holding fixed the points $\{z_i\}_{i=1}^N$.
That is, we can think of $F_u$ as a function of $u$.

\begin{prop}\label{prop: VL rat}
Suppose the modules $\M_{z_i}$ are all smooth. 
Then $F_u$ is a \emph{rational} function of $u$ with poles at most at the points in $\Gamma \bm z \cup \{ 0 \}$.
\end{prop}
\begin{proof}
In the proof of Proposition \ref{prop: VL remove}, we established existence of the following commutative diagram, where the notation $V, \mf a, \mf b, M$ is as in that proof, $\theta$ is the  isomorphism \eqref{vli} and $\psi$ is the isomorphism defined through the decomposition \eqref{c1}.
\be
\begin{tikzpicture}    
\matrix (m) [matrix of math nodes, row sep=2em, column sep=3em,text height=1.5ex, text depth=0.25ex]    
{
\ueva\otimes M & V & V/\mf a\on V& M \\
 & V\big/(\mf a \on V + \mf b \on V) & V/\mf a \on V \Big/ (\mf a \on V + \mf b \on V)/\mf a \on V & M/\mf b \on M \\
};
\path[->>,shorten <= 0,shorten >= 0]  (m-1-2) edge node[below] {$\pi$} (m-1-3);
\path[->] (m-1-1) edge node[above] {$\sim$} node[below] {$\theta\otimes \id$}(m-1-2);
\path[->] (m-1-3) edge node[above] {$\sim$} node[below] {$\psi$} (m-1-4);
\path[->] (m-2-3) edge node[above] {$\sim$} (m-2-4);
\path[->] (m-2-2) edge node[above] {$\sim$} (m-2-3);
\path[->>] (m-1-2) edge (m-2-2);
\path[->>] (m-1-3) edge (m-2-3);
\path[->>] (m-1-4) edge (m-2-4);
\end{tikzpicture}
\nn\ee
The map $F_u$ is the resulting map $\ueva\otimes M  \to M/\mf b\on M$. Consider the route right-right-right-down through the diagram. The final projection involves objects that do not depend on $u$ at all. So we are done if we can show that the map $\theta$ is a rational function of $u$ with pole at most at $0$ and that $\psi\circ \pi$ is a rational function of $u$ with poles at most in $\Gamma\bm z$. 

Consider $\theta$.
It is enough to consider states of the form $a_1(-1)a_2(-1)\dots a_K(-1)\vac$, $a_k\in \vla$, since these span $\ueva$. We are to re-express these as elements of $U(\U^\Gamma_u(\vla))\vac$ \ie as linear combinations of states of the form $(\iota_{t-u} f_1)(\iota_{t-u} f_2)\dots (\iota_{t-u} f_{K'})\vac$ with each $f_k$ of the form $\sum_{\alpha \in \Gamma} \frac{\alpha^{-1} R_{\alpha} a}{\alpha^{-1} t - u}$, $a\in \vla$. To do so we make recursive use of the identity 
\be a(-1) = \iota_{t-u} \left(\sum_{\alpha \in \Gamma} \frac{\alpha^{-1} R_{\alpha} a}{\alpha^{-1} t - u}\right)
+ \sum_{n\geq 0} \sum_{\alpha\in \Gamma\setminus \{1\}} \frac{(R_\alpha a)(n)}{\left((\alpha-1)u\right)^{n+1}},\nn\ee working from the outside inwards. Thus the coefficients are indeed rational functions of $u$ with poles at most at $u=0$.

Next, consider $\psi\circ \pi$. What \eqref{c1} asserts is that an $A\otimes\bm m \in U(\mf a) \vac \otimes M$ decomposes uniquely into a term of the form $\vac\otimes \tilde{\bm m}$ plus a term in $\mf a \on V$. Concretely, this is achieved by recursively applying the identity $((\iota_{t - u} f)\on A) \otimes \bm m = f\on (A \otimes \bm m) - \sum_{i=1}^P A \otimes (\iota_{t - z_i} f) \on \bm m$ where $f$ is one the $f_k$ above.
At each step in this recursive procedure, $\sum_{i=1}^P A \otimes (\iota_{t - z_i} f) \bm m$ is a rational function of $u$ with poles at most at points in $\Gamma \bm z$. This follows from smoothness of the modules $\M_{z_i}$ (which ensures that only finitely many terms in  each $(\iota_{t - z_i} f) \on \bm m$ are non-zero).
\end{proof}

The same statement holds if we include also a module at the origin.

Our real interest is in smooth modules. However, there is a technical difficulty: in what follows we shall need Proposition \ref{prop: wiiz} below, which says that if some equality holds in every space of coinvariants, then it is just true full stop. The authors do not know how, or whether, this proposition could be proved if one allowed only smooth modules. To get a straightforward proof, we have to allow also inverse limits of inverse systems of smooth modules, as follows. 

Suppose $\left((\M^k)_{k=1}^\8, (\pi^k_l)_{k\geq l}\right)$ is an \emph{inverse system} (see e.g. \cite[\S5.2]{Rotman}) of smooth modules over $\U(\vla)$. That is, suppose each $M^k$ is a smooth module over $\U(\vla)$ and $\pi^k_l: M^k \to M^l$ are module maps such that $\pi^k_k = \id_{M_k}$ for all $k$ and $\pi^k_l \circ \pi^m_k = \pi^m_l$ whenever $m\geq k\geq l$. The \emph{inverse limit} $\M^\8 := \varprojlim \M^k$ is the set of all sequences $(x_1,x_2,x_3,\dots)$ such that $x_k\in \M^k$ for all $k$ and $\pi^k_l(x_k) = x_l$ whenever $k\geq l$. $\M^\8$ is a module over $\U(\vla)$ and so we can define spaces of coinvariants including it as a tensor factor. 

By a rational function $g(u)$ of $u$ valued in $\M^\8$, we mean a sequence $(g_1(u),g_2(u),\dots)$ with each $g_k(u)$ a rational function of $u$ valued in $\M^k$, such that $\pi^k_l(g_k(u)) = g_l(u)$ whenever $k\geq l\geq 1$.\footnote{This is slightly subtle because the strength of the poles of these rational functions $g_k(u)$ is allowed to increase without bound as $k$ increases.} 
With this understanding, the proof of Proposition \ref{prop: VL rat} goes through also for inverse limits of inverse systems of smooth modules and so we have the following.
\begin{prop} Proposition \ref{prop: VL rat} holds also if the modules $\M_{z_i}$ are inverse limits of inverse systems of smooth modules, or submodules thereof.  \qed
\end{prop}
 
What allowing such modules buys us is a simple proof of the following important statement. 
\begin{prop}$ $\label{prop: wiiz} 
\begin{enumerate}
\item
Let $x\in \Cx$. Let $\M_x$ be any smooth module over $\U(\vla)$, and let $m \in \M_x$.\\ If $[ \dots \otimes \atp x m ] = 0$ then $m = 0$. 
\item
Let $\M_0$ be any smooth module over $\U(\vla)^\Gamma$ and let $m_0\in \M_0$. \\ If $[ \dots \otimes \atp 0 {m_0} ] = 0$
then $m_0 = 0$. 
\end{enumerate}
Here, and in what follows, the use of `$\,\cdots$' in the context of equivalence classes in spaces of coinvariants  means that the equality holds for all allowed choices of the remaining marked points (including none at all) and all allowed choices of modules assigned to them (smooth modules, inverse limits of smooth modules, or submodules thereof).
\end{prop}
\begin{proof}
Consider the universal enveloping algebra 
\be U := U(\U(\vla))\big/ \langle c(-1) - 1\rangle \cong 
U(\rho(\vla^o \otimes \C((t))))\nn\ee
cf. \eqref{L(vla) decomp}.
We let $\U(\vla)$ act on $U$ by letting the central element $c(-1)$ act as $1$ and letting $\rho(\vla^o\otimes \C((t)))$ act by left multiplication. This turns $U$ into a module over $\U(\vla)$, but not a smooth one. Let $F_k$, $k\geq 1$, be the two-sided ideal in $U$ generated by $\{a(n): a\in \vla^o, n > k\}$. Each quotient $U/ F_k$, $k\geq 1$, \emph{is} a smooth module over $\U(\vla)$. These quotients, together with the natural projections $\pi^k_l: U/F_k \twoheadrightarrow U/F_l$ for all $k\geq l$, form an inverse system  of smooth modules over $\U(\vla)$. The inverse limit $\varprojlim U/F_k$ has $U$ as a submodule, since any element $x\in U$ is uniquely specified by the sequence $(\pi^1(x),\pi^2(x),\pi^3(x),\dots)\in\varprojlim U/F_k $ where $\pi^k:U\twoheadrightarrow U/F_k$, $k\geq 1$, are the projection maps from $U$ itself.\footnote{$U$ is a proper submodule: for example $(a(1), a(1) + a(1) b(2), a(1) + a(1) b(2) + a(1) b(2) c(3), \dots)\in \varprojlim U/F_k$ is not the image of any element of $U$.}

Now let $\M$ be any smooth module over $\U(\vla)$, and let $m \in \M$. Suppose $\M$ is assigned to any $x\in \Cx$ and $U$ to any  $y\in \Cx\setminus\Gamma x$. We claim that if
\begin{equation*}
[ \atp y 1 \otimes \atp x m ] = 0
\end{equation*}
in the space of coinvariants $\left.U\otimes \M\right/ \U^\Gamma_{y,x}(\vla)$, then $m = 0$. Proving this claim is sufficient to prove part (1). For part (2) the argument is very similar but with $\U^\Gamma_{y,x}(\vla)$ replaced by $\U^\Gamma_{y,0}(\vla)$. 

Let $L= \rho(\vla^o\otimes \C((t-y)))$. 
The linear map
\be \iota_{t- y} : \U^\Gamma_{y, x}(\vla) \longrightarrow L \nn\ee
is an injection. (No two distinct rational functions have the same Laurent series at $y$).  Let $L^\perp\subset L$ be any complementary vector subspace to the image of $\U^\Gamma_{y,x}(\vla)$ in  $L$, so that
\be L = \U^\Gamma_{y,x}(\vla) \dotplus L^\perp.\nn\ee
Then 
\be U \cong U(L) = U(\U^\Gamma_{y,x}(\vla)) \dotplus U^\perp \quad\text{where}\quad U^\perp:=U(L) L^\perp U(L) .\nn\ee 
This is an isomorphism not only of vector spaces but also of $\U^\Gamma_{y,x}(\vla)$-modules. 
Hence we also have a decomposition of $U\otimes \M$ as an $\U^\Gamma_{y,x}(\vla)$-module:
\begin{equation} \label{UM decomp}
U \otimes M \,\,\cong_{\U^\Gamma_{y,x}(\vla)}\,\,  U(\U^\Gamma_{y,x}(\vla)) \otimes M \dotplus U^\perp \otimes M .
\end{equation}
To prove the claim we must show that if $m \neq 0$ then $1 \otimes m \not \in \U^\Gamma_{y,x}(\vla) \on ( U \otimes M )$. However, since $1\otimes m \in U(\U^\Gamma_{y,x}(\vla)) \otimes M$, by \eqref{UM decomp} it is enough to show that $1\otimes m \not \in 
\U^\Gamma_{y,x}(\vla) \on ( U(\U^\Gamma_{y,x}(\vla)) \otimes M )$.
But this is clear because $U(\U^\Gamma_{y,x}(\vla)) \otimes M$ is free as an $\U^\Gamma_{y,x}(\vla)$-module.
\end{proof}

\begin{rem}
Note how the very last step in this argument fails if we attempt to replace $U$ by the smooth module $U/F_k$ for any given finite $k$: the image of $U(\U^\Gamma_{y,x}(\vla)) \subset U$ under the projection $\pi^k:U\twoheadrightarrow U/F_k$ is not free as an $\U^\Gamma_{y,x}(\vla)$-module. Any rational function in $\U^\Gamma_{y,x}(\vla)$ with a zero of order $k$ at $y$ acts as zero.
\end{rem}

\subsection{Definitions of $Y$, $Y_M$ and $Y_W$}
In the space of coinvariants
\begin{equation} \label{tc}
\left.\ueva \otimes \bigotimes_{i = 1}^N \M_{z_i} \right/ \U^\Gamma_{u,\bm z}(\vla)
\end{equation}
consider the class
\be  [    \atp u A 
 \otimes \atp{z_1}{m_1} \otimes \atp{z_2}{m_2} \otimes \dots \otimes \atp{z_N}{m_N} 
 ]\label{class}
\ee
with $A\in \ueva$, $m_i\in \M_{z_i}$, $i=1,\dots,N$.
By Proposition \ref{prop: VL rat}, this is a rational function of $u$ valued in $\left.\bigotimes_{i = 1}^N \M_{z_i}  \right/ \U^\Gamma_{\bm z}(\vla)$. As with any rational function, it is possible to take its Laurent expansion as $u$ approaches, for example, the point $z_1$. The result is a Laurent series in $(u-z_1)$ with coefficients in $\left.\bigotimes_{i = 1}^N \M_{z_i} \right/ \U^\Gamma_{\bm z}(\vla)$. 
Proposition \ref{prop: YMYW} below then says, in particular, that this Laurent series can be written in the form
\be \sum_{n=-K}^\8 (u-z_1)^n [ 
  A^M_{(n)} \atp{z_1}{m_1} \otimes  \atp{z_2}{m_2} \otimes \dots \otimes \atp{z_N}{m_N}  ]\nn
\ee
for some linear maps $A^M_{(n)}\in \End(\M_{z_1})$, $n\in \Z$ which depend on the module $\M_{z_1}$ and (linearly) on $A$, and some integer $K\in \Z$ which may depend on $m_1$; the important point is that neither $K$ nor the maps $A^M_{(n)}$ depend at all on the modules assigned to other marked points, or on the positions of the marked points.

\begin{prop}[Definition of $Y_M$ and $Y_W$]$ $\label{prop: YMYW}
\begin{enumerate}[(a)]
\item 
For each smooth $\U(\vla)$-module $\M$ there exists a unique linear map 
\be \ueva\longrightarrow \Hom\left(\M,\M((v))\right),\quad A\longmapsto Y_M(A,v) ,\nn\ee 
such that for all $m\in \M$ and all $x\in \Cx$,
\be \iota_{u-x} \big[ \atp u A\otimes \atp x m \otimes \cdots \big] = 
    \big[Y_M(A,u-x) \atp x m \otimes \cdots \big]. \nn\ee
\item \label{YWdef}
For each smooth $\U(\vla)^\Gamma$-module $\M_0$ there exists a unique linear map 
\be \ueva\longrightarrow \Hom\left(\M_0,\M_0((v))\right),\quad A\longmapsto Y_W(A,v), \nn\ee 
such that for all $m_0\in \M_0$,
\be \iota_u \big[ \atp u A \otimes \cdots \otimes \atp 0 {m_0}\big] = 
    \big[ \cdots \otimes Y_W(A,u) \atp 0 {m_0}\big]. \nn\ee
\end{enumerate}
\end{prop}

\begin{proof} 
The proof is given in \S\ref{proof: YMYW} below. 
\end{proof}

By taking the special case $\M=\ueva$ in Proposition \ref{prop: YMYW} one has the following important corollary.

\begin{cor}[Definition of the state-field correspondence $Y$]\label{Ydef}
There exists a unique linear map
\be \ueva\to \Hom\left(\ueva,\ueva((v))\right),\quad A\longmapsto Y(A,v), \nn\ee 
such that for all $B\in \ueva$,
\begin{equation*}
\iota_{u-x} \big[ \atp uA\otimes \atp xB \otimes \cdots \big] = 
    \big[ Y(A,u-x) \atp x B \otimes \cdots\big].
\end{equation*}
\qed
\end{cor}

\begin{rem} 
Let $Y(A,x)_-$ denote the part of $Y(A,x)$ in $\Hom(\ueva,x^{-1}\ueva[x^{-1}])$. Then for all $a,b\in \vla \hookrightarrow \ueva$, $Y(a,x)_-b \in x^{-1}\vla[x^{-1}]$, and this is nothing but the map \eqref{Y map VLA} which encodes the vertex Lie algebra structure on $\vla$. See \eqref{Y def VA} below.
\end{rem}

\begin{prop}[Borcherds and quasi-Borcherds identities]\label{prop: borcherds}$ $
\begin{enumerate}
\item Let $\M$ be a smooth $\U(\vla)$-module.
For any $A, B \in \ueva$ and any $m\in \M$, and for all rational functions $f(x,y) \in \C[x^{\pm 1}, y^{\pm 1}, (x-y)^{-1}]$, we have
\begin{multline*} 
\res_{x - y}  \iota_{y, x - y} f(x,y) Y_M\big( Y(A, x - y) B, y \big) m  \notag\\
= \res_x \iota_{x, y} f(x,y)  Y_M(A, x) Y_M(B, y) m  - \res_x \iota_{y, x} f(x,y) Y_M(B, y) Y_M(A, x) m .
\end{multline*}
\item Let $\M_0$ be a smooth $\U(\vla)^\Gamma$-module
and let
\be p(x,y) := \prod_{\alpha \in \Gamma \setminus \{ 1 \}} (x - \alpha y) = (x^T - y^T)/(x - y)\in \C[x,y].\nn\ee 
For any $A, B \in \ueva$ and any $m_0\in \M_0$ there exists $k\in \Z_{\geq 0}$ such that  for all rational functions $f(x,y) \in \C[x^{\pm 1}, y^{\pm 1}, (x-y)^{-1}]$, we have
\begin{multline*} 
\res_{x - y}  p(x,y)^k \iota_{y, x - y} f(x,y) Y_W\big( Y(A, x - y) B, y \big) m_0  \notag\\
=\res_x p(x,y)^k \iota_{x, y} f(x,y) Y_W(A, x) Y_W(B, y) m_0  - \res_x p(x,y)^k \iota_{y, x} f(x,y) Y_W(B, y) Y_W(A, x) m_0  .
\end{multline*}
\end{enumerate}
\end{prop}
\begin{proof}
We give a proof of part (2) in \S\ref{sec: borcherdsproof}. Part (1) is very similar.
\end{proof}

The usual way to express the content of these propositions is to say that 
\begin{enumerate}[(1)]
\item $\ueva$ has the structure of a \emph{vertex algebra}, 
\item smooth $\U(\vla)$-modules have the structure of \emph{modules} over this vertex algebra $\ueva$, and 
\item smooth $\U(\vla)^\Gamma$-modules have the structure of \emph{quasi-modules} over $\ueva$. 
\end{enumerate}
To explain what these statements mean, let us now digress from $\ueva$ to recall some facts about vertex algebras in general. 

We should stress, however, that for us Proposition \ref{prop: YMYW} and Corollary \ref{Ydef} are what define the maps $Y_M$, $Y_W$ and $Y$. Their properties follow from these `global' definitions, and the proofs in \S\ref{sec: proofs} will involve applying the residue theorem to `global' objects like \eqref{class}, rather than using the axioms of vertex algebras and their (quasi-)modules.

\section{Vertex algebras}\label{sec: VAs}
A \emph{vertex algebra} is a vector space $\va$ over $\C$ with a distinguished vector $\vac \in \va$ called the \emph{vacuum} and equipped with a linear map, referred to as the \emph{state-field correspondence} or \emph{vertex operator map},
\begin{equation} \label{Y map}
\begin{split}
Y(\cdot, x) : \va &\to \Hom(\va, \va((x)) ) \subset \End \va [[x, x^{-1}]],\\
A &\mapsto Y(A, x) = \sum_{n \in \mathbb{Z}} A_{(n)} x^{-n-1}, \qquad A_{(n)} \in \End \va,
\end{split}
\end{equation}
satisfing the following axioms: 
\begin{enumerate}[($i$)]
  \item \emph{Vacuum axiom} \--- $Y(\vac, x) = \text{id}_\va$.
  \item \emph{Creation axiom} \--- For any $A \in \va$ we have $Y(A, x) \vac \in A + x \va[[x]]$.
  \item \emph{Borcherds identity} \--- For any $A, B, C \in \va$ and any $f(x,y) \in \C[x^{\pm 1}, y^{\pm 1}, (x-y)^{-1}]$, we have
\begin{multline} \label{Borcherds V}
\res_{x - y} \iota_{y, x - y} f(x,y) Y\big( Y(A, x - y) B, y \big)C\\
= \res_x \iota_{x, y} f(x,y) Y(A, x) Y(B, y)C - \res_x \iota_{y, x} f(x,y) Y(B, y) Y(A, x)C.
\end{multline}
\end{enumerate}

Elements of $\va$ are called \emph{states} and we say that $Y(A, x)$ is the \emph{vertex operator} or \emph{field} associated to a given state $A \in \va$. More generally, an $(\End \va)$-valued formal distribution $F(x) \in \End \va [[x, x^{-1}]]$ with the property that $F(x) A \in \va((x))$ for all $A \in \va$, \emph{i.e.} $F(x) \in \Hom(\va, \va((x)) )$, is called a \emph{field}.

One may equally regard the state-field correspondence \eqref{Y map} as defining a linear map $\va \otimes \va \to \va((x))$. This, in turn, can be thought of as describing a collection of products $\va \otimes \va \to \va$, $(A, B) \mapsto A_{(n)} B$ labelled by $n \in \mathbb{Z}$ and with the property that for any $A, B \in \va$ we have $A_{(n)} B = 0$ for sufficiently large $n$. In particular, the $n^{\rm th}$-product may be extracted as
\begin{equation} \label{nth prod}
A_{(n)} B = \res_x \, x^n Y(A, x) B.
\end{equation}

It is also convenient to introduce an endomorphism $\D : \va \to \va$, called the \emph{translation operator}, defined as $\D A = A_{(-2)} \vac$ on any $A \in \va$. It has the property that $\D \vac = 0$. We shall require (somewhat non-standardly, cf. Remark \ref{kerrem}), that in fact
\be \ker \D = \C\vac. \nn\ee
One also has \cite[Prop 3.1.21]{LLbook}
\begin{equation} \label{T rel}
Y(\D A, x) = [\D, Y(A, x)] = \partial_x Y(A, x),
\end{equation}
which, written in terms of modes, says that for any $n \in \mathbb{Z}$
\begin{equation} \label{TA V}
(\D A)_{(n)} = - n A_{(n-1)}.
\end{equation}
It also follows from \eqref{T rel} that $\D$ acts as a derivation of all $n^{\rm th}$-products. Indeed, for any $A, B \in \va$ we have
$\D \big( Y(A, x) B \big) = Y(\D A, x) B + Y(A, x) \D B$.
Multiplying this relation by $x^n$ for $n \in \Z$ and taking the residue in $x$ to extract the $n^{\rm th}$-product as in \eqref{nth prod} we obtain
\begin{equation*}
\D \big( A_{(n)} B \big) = (\D A)_{(n)} B + A_{(n)} (\D B).
\end{equation*}

A vertex algebra $\va$ is called $\mathbb{Z}$-graded (resp. $\mathbb{Z}_{\geq 0}$-graded) if $\va$ is $\mathbb{Z}$-graded (resp. $\mathbb{Z}_{\geq 0}$-graded) as a vector space, $\vac$ is of degree $0$ and for any two homogeneous states $A, B \in \va$ and $n \in \mathbb{Z}$ we have
\begin{equation} \label{degree rel}
\deg\big( A_{(n)} B \big) = \deg A + \deg B - n - 1.
\end{equation}
We denote the subspace of states $A \in \va$ with $\deg A = n$ for $n \in \mathbb{Z}_{\geq 0}$ as $\va_{(n)}$.
It follows from its definition that $\D$ is a linear operator of degree $1$, using \eqref{degree rel}. Introduce also the degree operator, denoted $L(0)$, which acts on any homogeneous $A \in \va$ as
\begin{equation} \label{L0 def}
L(0) A = \deg (A) \, A.
\end{equation}

\begin{exmp} \label{com VA}
Let $\mathcal{A}$ be a commutative associative unital algebra with a derivation $\delta \in \text{Der} \mathcal{A}$. We can endow $\mathcal{A}$ with a vertex algebra structure by taking the identity element $1 \in \mathcal{A}$ as vacuum vector and defining the state-field correspondence for any $f, g \in \mathcal{A}$ as
\begin{equation*}
Y(1, x) g = g, \qquad
Y(f, x) g = \sum_{n \geq 0} \frac{x^n}{n!} (\delta^n f) g = (e^{x \delta} f) g.
\end{equation*}
The translation operator is $\delta$ since the coefficient of $x$ in $Y(f, x) 1$ is precisely $\delta f$.
\end{exmp}

\subsection{Compatible $\Gamma$-action} \label{sec: Gamma VA}
We are concerned with vertex algebras equipped with a compatible action of the group $\Gamma$. We start by recalling the definition of a homomorphism of vertex algebras \cite{LLbook}.

Let $\va$ and $\va'$ be vertex algebras with state-field correspondences $Y$ and $Y'$ and vacuum vectors $\vac$ and $\vac'$, respectively. A homomorphism $\varphi : \va \to \va'$ is by definition a linear map such that $\varphi \big( \vac \big) = \vac'$ and
\begin{equation} \label{VA hom}
\varphi \big( Y(A, x) B \big) = Y' \big( \varphi(A), x \big) \varphi(B),
\end{equation}
for any $A, B \in \va$, where on the left hand side the map $\varphi$ has been canonically extended to a linear map $\va[[x, x^{-1}]] \to \va'[[x, x^{-1}]]$. The condition \eqref{VA hom} expresses the fact that $\varphi$ is a homomorphism of all $n^{\rm th}$-products \eqref{nth prod}. It follows from the definition of the translation operator $\D$ that a homomorphism $\varphi : \va \to \va'$ commutes with the action of $\D$ on $\va$ and $\va'$. Indeed, for any $A \in \va$ we have
\begin{equation*}
\D \big( \varphi(A) \big) = \big( \varphi(A) \big)_{(-2)} \vac' = \big( \varphi(A) \big)_{(-2)} \varphi(\vac) = \varphi\big( A_{(-2)} \vac \big) = \varphi(\D A).
\end{equation*}
An automorphism of a vertex algebra $\va$ is an isomorphism $\varphi : \va \to \va$. The group of automorphisms of $\va$ will be denoted $\Aut \va$.

Let $\va$ be a vertex algebra and $\sigma : \va \to \va$ an automorphism whose order divides $T$. We define an action of $\Gamma$ on $\va$ through automorphisms by letting $\omega$ act as $\sigma$. Specifically, we introduce the group homomorphism
\begin{equation*}
\check{R} : \Gamma \longrightarrow \Aut \va; \quad \alpha \longmapsto \check{R}_{\alpha}
\end{equation*}
defined by $\check{R}_{\omega} = \sigma$. We shall, however, be interested in a slight variant of this action under which homogeneous states of different degrees transform differently. In analogy with the vertex Lie algebra setup, \emph{cf.} \eqref{Gamma act VLA}, we therefore introduce the group homomorphism
\begin{equation} \label{Gamma act VA}
R : \Gamma \longrightarrow GL(\va); \quad \alpha \longmapsto R_{\alpha} := \alpha^{L(0)} \check{R}_{\alpha}.
\end{equation}
The map $R_{\alpha}$ for $\alpha \in \Gamma$ is not an automorphism of $\va$. Instead of satisfying a relation of the type \eqref{VA hom}, it obeys
\begin{equation} \label{Gamma eq VA}
R_{\alpha} \big( Y(A, x) B \big) = Y(R_{\alpha} A, \alpha x) R_{\alpha} B,
\end{equation}
for any $A, B \in \va$. Vertex algebras equipped with an action of a group $\Gamma$ satisfying such a relation were studied in \cite{Li4, Li3} and were referred to as \emph{$\Gamma$-vertex algebras}. In the present context such an action is completely specified by an automorphism $\sigma \in \Aut \va$ whose order divides $T$.

\subsection{Consequences of the Borcherds identity}\label{sec: cofb}
There are a number of useful consequences of the Borcherds identity \eqref{Borcherds V} obtained by making different choices for the rational function $f(x,y)$. Taking $f = (x - y)^{-1}$ yields the vertex operator associated to the $(-1)^{\rm st}$-product of any two states $A, B \in \va$, namely
\begin{equation} \label{norm prod}
Y\big( A_{(-1)} B, x \big) = \; \nord{Y(A, x) Y(B, x)},
\end{equation}
where we make use the following standard notation. For any $\End \va$-valued formal distribution $F(x) = \sum_{n \in \mathbb{Z}} F_n x^{-n-1}$ with $F_n \in \End \va$ we define
\begin{equation} \label{subpm}
F(x)_+ = \sum_{n < 0} F_n x^{-n-1}, \qquad
F(x)_- = \sum_{n \geq 0} F_n x^{-n-1}.
\end{equation}
Then given any two fields $F(x)$ and $G(x)$ we may defined their \emph{normal ordered product} as
\begin{equation*}
\nord{F(x) G(x)} \; = F(x)_+ G(x) + G(x) F(x)_- = \sum_{n \in \mathbb{Z}} \left( \sum_{k < 0} F_k G_{n - k - 1} + \sum_{k \geq 0} G_{n - k - 1} F_k \right) x^{-n-1}.
\end{equation*}
Note first of all that this a well defined element of $\End \va [[x, x^{-1}]]$ by virtue of the fact that $F(x)$ and $G(x)$ are both fields, since the expression in brackets is a finite sum when acting on any $A \in \va$. Moreover, it also clearly defines a field since $\nord{F(x) G(x)} \! A \in \va((x))$ using again the fact that $G(x)$ is a field.
More generally, the vertex operator associated to a negative product $A_{(-n-1)} B$, $n \geq 0$ between two states $A, B \in \va$ reads
\begin{align} \label{norm prod der}
Y\big( A_{(-n-1)} B, x \big) = \frac{1}{n!} \nord{\big( \partial_x^n Y(A, x) \big) Y(B, x)},
\end{align}
which is obtained by applying the Borcherds identity with $f = (x - y)^{-n-1}$ for $n \in \mathbb{Z}_+$.

The non-negative products $A_{(n)} B$, $n \geq 0$ are related instead to the commutator of the vertex operators associated to the states $A, B \in \va$. 
Specifically, by applying the Borcherds identity with $f = x^m$, $m \in \mathbb{Z}$ we find
\begin{equation} \label{com Am YB}
\big[ A_{(m)}, Y(B, y) \big] = \sum_{k \geq 0} \binom{m}{k} Y\big( A_{(k)} B, y \big) y^{m - k}.
\end{equation}
Using the fact that the $k^{\rm th}$-product $A_{(k)} B$ vanishes for sufficiently large $k$ we see that the above sum is finite.
Multiplying this identity by $y^n$, $n \in \mathbb{Z}$ and taking the residue in $y$ we obtain the commutation relation between modes of the fields associated to $A$ and $B$:
\begin{equation} \label{com Am Bn}
\big[ A_{(m)}, B_{(n)} \big] = \sum_{k \geq 0} \binom{m}{k} \big( A_{(k)} B \big)_{(m + n - k)}.
\end{equation}
This important relation shows that the modes of all vertex operators $Y(A, x)$ for $A \in \va$ form a Lie subalgebra of $\End \va$. Alternatively, multiplying \eqref{com Am YB} by $x^{-m-1}$ and summing over $m \in \mathbb{Z}$ yields
\begin{equation} \label{commutator formula}
\big[ Y(A, x), Y(B, y) \big] = \sum_{k \geq 0} \frac{1}{k!} Y\big( A_{(k)} B, y \big) \partial_y^k \delta(x, y).
\end{equation}
This immediately implies the crucial property of locality for vertex algebras. That is, for any two states $A, B \in \va$, there exists an $r \in \mathbb{Z}_+$ such that
\begin{equation} \label{locality}
(x - y)^r \big[ Y(A, x), Y(B, y) \big] = 0.
\end{equation}

Finally, another consequence of the Borcherds identity and the definition of the translation operator is the following skew-symmetry property \cite[Proposition 3.1.19]{LLbook}
\begin{equation*}
Y(A, x) B = e^{x \D} Y(B, -x) A,
\end{equation*}
for any $A, B \in \va$. Multiplying both sides by $x^n$ for any $n \in \Z_{\geq 0}$ and taking the residue in $x$, we have
\begin{equation} \label{vaskew}
A_{(n)} B = - \sum_{k \geq 0} \frac{(-1)^{n + k}}{k!} \D^k \big( B_{(n + k)} A \big).
\end{equation}

\subsection{Modules and quasi-modules over vertex algebras}\label{sec: vmods}
The notion of a quasi-module was introduced by Li in \cite{Li3,Li4}. By definition, a \emph{quasi-module} over a vertex algebra $\va$, or \emph{$\va$-quasi-module} for short, is a vector space $W$ equipped with a linear map 
\begin{align*}
Y_W(\cdot, x) : \va &\to \Hom(W, W((x))) \subset \End W [[x, x^{-1}]],\\
A &\mapsto Y_W(A, x) = \sum_{n \in \mathbb{Z}} A^W_{(n)} x^{-n-1}, \qquad A^W_{(n)} \in \End W,
\end{align*}
which must satisfy axioms similar to those of the state-field correspondence of a vertex algebra, namely
\begin{enumerate}[($i$)]
  \item \emph{Vacuum axiom} -- $Y_W(\vac, x) = \text{id}_W$.
  \item \emph{quasi-Borcherds identity} -- For any $A, B \in \va$ and any $w\in W$ there exists a nonzero polynomial $p(x,y)\in \C[x,y]$ such that for all $f(x, y) \in \C[x^{\pm 1}, y^{\pm 1}, (x-y)^{-1}]$, we have
\begin{align} \label{Borcherds W}
&\res_{x - y} p(x,y) \iota_{y, x - y} f(x,y) Y_W\big( Y(A, x - y) B, y \big) w \notag\\
&\quad = \res_x \big( p(x,y) \iota_{x, y} f(x,y) Y_W(A, x) Y_W(B, y) w - p(x,y) \iota_{y, x} f(x,y) Y_W(B, y) Y_W(A, x) w \big),
\end{align}
\end{enumerate}
Formal distributions $F(x) \in \Hom(W, W((x)))$ are sometimes referred to as \emph{$(\End W)$-valued fields}, or simply \emph{fields} when the context is clear.

A \emph{module} over a vertex algebra $\va$, or \emph{$\va$-module} for short, is a quasi-module $M$ such that the identity \eqref{Borcherds W} holds with $p(x,y) = 1$; that is, one in which we have  the Borcherds identity
\begin{align} \label{Borcherds M}
\res_{x - y} & \iota_{y, x - y} f(x,y) Y_M\big( Y(A, x - y) B, y \big) m \notag\\
&= \res_x \big( \iota_{x, y} f(x,y) Y_M(A, x) Y_M(B, y) m - \iota_{y, x} f(x,y) Y_M(B, y) Y_M(A, x) m \big),
\end{align}
for all $A, B \in \va$, all $m \in M$ and all rational functions $f(x,y)$ with poles at most at $x=0$, $y=0$ and $x=y$. 
(Every vertex algebra is a module over itself, with $Y_M=Y$.)

\subsection{From vertex algebras to Lie algebras}\label{sec: vaLA}
Comparing the definition of a vertex Lie algebra, \S\ref{sec: VLAs}, with \eqref{TA V}, \eqref{vaskew} and \eqref{com Am Bn}, we see that if one starts with a vertex algebra and forgets about the negative products $\cdot_{(-n)}\cdot: \va\otimes \va \to \va$, $n\in \Z_{\geq 1}$, then the resulting object is nothing but a vertex Lie algebra. (From one perspective this requirement is essentially what fixes the definition of a vertex Lie algebra.) That is, we have the following. 
\begin{lem} Every ($\Gamma$-)vertex algebra has the structure of a ($\Gamma$-)vertex Lie algebra.
\label{forgetful} \qed
\end{lem}
By virtue of Lemma \ref{forgetful}, all the statements from \S\ref{sec: LtoLie} onwards about going from vertex Lie algebras to (genuine) Lie algebras carry over and apply to vertex algebras. 
First, associated to any vertex algebra $\va$ one has the Lie algebra $\U(\va)$, the local copies of this Lie algebra, $\U(\va)_{z_i}$, associated to the points $\{z_i\}_{i=1}^N$, and 
\be \U(\va)_{\bm z} := \label{UVN def}
 \left.\bigoplus_{i = 1}^N \U(\va)_{z_i} \right/ J_N
\ee
where $J_N$ is the ideal spanned by states of the form
\begin{equation} \label{UNV ideal J}
\vac(-1)_{z_i} - \vac(-1)_{z_j}, \qquad 1 \leq i < j \leq N.
\end{equation}
Second, we have the global Lie algebra $\U_{\bm z}^\Gamma(\va)$. Its definition is entirely parallel to that of $\U_{\bm z}^\Gamma(\vla)$, but for clarity let us go through the details. Given an automorphism $\sigma\in \Aut\va$ whose order divides $T=|\Gamma|$, let $R : \Gamma \to \text{GL}(\va)$ be the group homomorphism defined in \eqref{Gamma act VA}. We consider the action of $\Gamma$ on $\va \otimes \C_{\Gamma \bm z}^{\infty}(t)$ defined for any $\alpha \in \Gamma$ by
\begin{equation} \label{Gamma act LVz}
\alpha . (A \otimes f(t)) := \alpha^{-1} R_{\alpha} A \otimes f(\alpha^{-1} t),
\end{equation}
and then, applying Lemma \ref{lem: VLA} with $\mathcal{A} = \C_{\Gamma {\bm z}}^{\infty}(t)$ we have the Lie algebra
\begin{equation} \label{quotient VAp}
\Lie_{\C_{\Gamma {\bm z}}^{\infty}(t)} \va = \big( \va \otimes \C_{\Gamma {\bm z}}^{\infty}(t) \big) \big/ \im \partial.
\end{equation}
By Lemma \ref{lem: twist action} the action \eqref{Gamma act LVz} descends to an action of $\Gamma$ by automorphisms on $\Lie_{\C_{\Gamma {\bm z}}^{\infty}(t)} \va$. 
We define the subspace $J_{\Gamma {\bm z}}$ of $\Lie_{\C_{\Gamma {\bm z}}^{\infty}(t)} \va$ as
\begin{equation} \label{UGzV ideal J}
J_{\Gamma {\bm z}} = \textup{span} \left\{ \sum_{\alpha \in \Gamma} \frac{\vac}{t - \alpha z_i} - \sum_{\alpha \in \Gamma} \frac{\vac}{t - \alpha z_j} \; \bigg| \; 1 \leq i < j \leq N \right\},
\end{equation}
The elements of \eqref{UGzV ideal J} are all central in $\Lie_{\C_{\Gamma {\bm z}}^{\infty}(t)} \va$. Thus $J_{\Gamma {\bm z}}$ is an ideal and we have the quotient map $\Lie_{\C_{\Gamma {\bm z}}^{\infty}(t)} \va \to \big( \Lie_{\C_{\Gamma {\bm z}}^{\infty}(t)} \va \big) \big/ J_{\Gamma {\bm z}}$. Let $\U_{\Gamma {\bm z}}(\va)$ denote the image of $\rho\big( \va^o \otimes \C^{\infty}_{\Gamma {\bm z}}(t) \big)$ under this map, \ie
\begin{equation} \label{quotient VA}
\U_{\Gamma {\bm z}}(\va) := \rho\big( \va^o \otimes \C^{\infty}_{\Gamma {\bm z}}(t) \big) \big/ J_{\Gamma {\bm z}},
\end{equation}
where $\va^o$ is a choice of complement of $\im \D \oplus \C \vac$ in $\va$. 
Since the ideal $J_{\Gamma {\bm z}}$ is also invariant under the action of $\Gamma$, we have an induced action of $\Gamma$ on the quotient \eqref{quotient VA} and we may consider the subspace of $\Gamma$-invariants
\begin{equation} \label{UGdef}
\U^{\Gamma}_{\bm z}(\va) := \big( \U_{\Gamma {\bm z}}(\va) \big)^{\Gamma}.
\end{equation}
Similarly, we have vertex algebra analogs of $\U(\vla)_{\bm z,0}$ and $\U^\Gamma_{\bm z, 0}(\vla)$, namely $\U(\va)_{\bm z,0}$ and $\U^\Gamma_{\bm z, 0}(\va)$.

\section{Universal enveloping vertex algebras and `big' Lie algebras} \label{sec: uevaandbigla}

\subsection{$\ueva$ as a vertex algebra}
Recall the definition of $\ueva$ from \S\ref{sec: ueva}.
\begin{prop}\label{va struc}
$\ueva$ is a $\Z_{\geq 0}$-graded vertex algebra, with vacuum $\vac$ and vertex operator map $Y: \ueva \to \Hom(\ueva,\ueva((x)))$ as defined in Corollary \ref{Ydef}. 
\end{prop}
\begin{proof}
First we must check that $Y$ satisfies axioms $(i)$--$(iii)$. Axiom $(iii)$, namely Borcherds identity, follows from part $(1)$ of Proposition \ref{prop: borcherds} in the case $\M=\ueva$. Axiom $(i)$ is clear. Now consider axiom $(ii)$. By Proposition \ref{prop: VL remove}, we have
\begin{equation*}
[ \atp u A \otimes \atp x \vac \otimes \atp z m] = [ \atp u A \otimes \atp z m]
\end{equation*}
under the natural isomorphism $(\ueva\otimes \ueva \otimes \M)/\U^\Gamma_{u,x,z}(\vla) \cong_\C (\ueva \otimes \M)/\U^\Gamma_{u,z}(\vla)$, which is independent of $x$ and in particular has no pole at $x$ when viewed as a function of $u$. Thus
\begin{align*} \iota_{u-x} [\atp u A\otimes \atp x \vac \otimes \atp z m] = \iota_{u-x} [\atp u A \otimes \atp z m] &= [\atp x A\otimes \atp z m] + \mc O(u-x)
\end{align*}
which, according to the definition of $Y$ in Corollary \ref{Ydef}, says that
$Y(A,x) \vac = A + x \ueva [[x]]$.

The map $Y$ defined in Corollary \ref{Ydef} therefore endows $\ueva$ with the structure of a vertex algebra. By equation \eqref{Y def VA} and remark \ref{rem: tostandard} below, the latter agrees with the standard vertex algebra structure on $\ueva$ and is therefore $\Z_{\geq 0}$-graded.
Finally, the fact that $\ker \D = \C \vac$ (which, recall is part of the definition of a vertex algebra for us)  follows from the explicit form of the action of $\D$ on $\ueva$ given below in \eqref{D act V}.
\end{proof}

Following Primc, \cite{Primc}, we sometimes refer to the vacuum Verma module $\ueva$  as the \emph{universal enveloping vertex algebra} of $\vla$.

For states of the form $a(-1) \vac$ with $a \in \vla$, we have
\begin{equation} \label{Y def VA}
Y(a(-1) \vac, x) = \sum_{n \in \mathbb{Z}} a(n) x^{-n-1},
\end{equation}
and in particular $(a(-1) \vac)_ {(n)} = a(n)$. (Here the action of $a(n) \in \U(\vla)$ on $\ueva$ is given by the $\U(\vla)$-module structure on $\vla$.) 
Indeed, let $\M_{z_i}$, $i = 1, \ldots, K$ be any collection of $\U(\vla)$-modules attached to the set of points $\bm z = \{ z_i \}_{i=1}^K \subset \Cx \setminus \Gamma\{ u, x \}$ with $\Gamma z_i \cap \Gamma z_j = \emptyset$ for $i \neq j$, and consider the class of $a(-1)\vac \otimes B \otimes \bm m$ in the space of coinvariants $\VV(\vla) \otimes \VV(\vla) \otimes \bigotimes_{i = 1}^K \M_{z_i} \big/ \U^{\Gamma}_{u, x, \bm z}(\vla)$. Using the $\Gamma$-equivariant function
\begin{equation*}
\sum_{\alpha\in \Gamma} \alpha\on \left( \frac{ a}{t-u} \right)
= \sum_{\alpha \in \Gamma} \frac{\alpha^{-1} R_{\alpha} a}{\alpha^{-1} t - u} \in  \U^{\Gamma}_{u, x, \bm z}(\vla)
\end{equation*}
we may `swap' $a(-1)$ to obtain
\begin{equation*}
\big[ a(-1) \atp u {\vac} \otimes \atp x B \otimes \atp {\bm z}{\bm m} \big] = \sum_{n \in \Z_{\geq 0}} \sum_{\alpha \in \Gamma} \bigg[ \atp u {\vac} \otimes \frac{(R_\alpha a)(n)}{(\alpha u - x)^{n+1}} \atp x B \otimes \atp {\bm z}{\bm m} \bigg] + \sum_{i = 1}^K \sum_{n \in \Z_{\geq 0}} \sum_{\alpha \in \Gamma} \bigg[ \atp u {\vac} \otimes \atp x B \otimes \frac{(R_\alpha a)(n)}{(\alpha u - z_i)^{n+1}} \atp {\bm z}{\bm m} \bigg].
\end{equation*}
Next, applying the map $\iota_{u - x}$ to both sides yields
\begin{multline*}
\iota_{u - x} \big[ a(-1) \atp u {\vac} \otimes \atp x B \otimes \atp {\bm z}{\bm m} \big] = \sum_{n \in \Z_{\geq 0}} \bigg[ \atp u {\vac} \otimes (u - x)^{-n-1} a(n) \atp x B \otimes \atp {\bm z}{\bm m} \bigg]\\
+ \sum_{n \in \Z_{\geq 0}} \sum_{\alpha \in \Gamma \setminus \{ 1 \}} \iota_{u - x} \bigg[ \atp u {\vac} \otimes \frac{(R_\alpha a)(n)}{(\alpha u - x)^{n+1}} \atp x B \otimes \atp {\bm y}{\bm m} \bigg]
+ \sum_{i = 1}^K \sum_{n \in \Z_{\geq 0}} \sum_{\alpha \in \Gamma} \iota_{u - x} \bigg[ \atp u {\vac} \otimes \atp x B \otimes \frac{(R_\alpha a)(n)}{(\alpha u - y_i)^{n+1}} \atp {\bm y}{\bm m} \bigg].
\end{multline*}
The two terms on the second line are Taylor series in $(u - x)$ and can be shown to be equal to
\begin{equation*}
\sum_{n \in \Z_{\geq 0}} \bigg[ \atp u {\vac} \otimes (u - x)^n a(-n-1) \atp x B \otimes \atp {\bm z}{\bm m} \bigg]
\end{equation*}
upon swapping the mode $a(-n-1)$. It follows that
\begin{equation*}
\iota_{u - x} \big[ a(-1) \atp u {\vac} \otimes \atp x B \otimes \atp {\bm y}{\bm m} \big] = \sum_{n \in \Z} \bigg[ \atp u {\vac} \otimes (u - x)^{-n-1} a(n) \atp x B \otimes \atp {\bm y}{\bm m} \bigg],
\end{equation*}
which by Corollary \ref{Ydef} implies equation \eqref{Y def VA}.

\begin{rem}\label{rem: tostandard}
For us, Corollary \ref{Ydef} is what defines the state-field correspondence $Y$ on $\ueva$. 
The more standard approach is to define $Y$ by first requiring that \eqref{Y def VA} hold and then proving a `reconstruction theorem' to show that $Y$ extends uniquely to a well-defined state-field correspondence on all of $\ueva$. In this form, the result that a unique such vertex algebra structure on $\ueva$ exists for the general vertex Lie algebra $\vla$ is due to Primc, \cite{Primc}.  See also \cite[\S 2.4]{FB} and \cite[\S4.5]{KacVertex}.
To see why equation \eqref{Y def VA} is enough to define the vertex operator map on the whole of $\ueva$, note that by \eqref{Vvla vect} a general state in $\ueva$ can be written as a finite linear combination of states of the form $a_1(- n_1) \ldots a_j(- n_j) \vac$ for some $a_1, \ldots, a_j \in \vla$ and $f_{n_1, \ldots, n_j} \in \C$. The vertex operator associated to the latter state can be obtained by writing it as
\begin{equation*}
(a_1(-1) \vac)_{(- n_1)} \ldots (a_j(-1) \vac)_{(- n_j)} \vac.
\end{equation*}
Then, applying the relation \eqref{norm prod der} recursively, one finds
\begin{align}\label{iterate}
Y(a_1(- n_1) \ldots a_j(- n_j) &\vac, x)\\
&= \frac{1}{(n_1 - 1)!} \ldots \frac{1}{(n_j - 1)!} \nord{\partial^{n_1 - 1}_x Y(a_1(-1) \vac, x) \ldots \partial^{n_j - 1}_x Y(a_j(-1) \vac, x)},\nn
\end{align}
where the normal ordered product of more than two operators is defined to be right associative. For example $\nord{F(x) G(x) H(x)}$ is shorthand for $\nord{F(x) (\nord{G(x) H(x)})}$, which in general will be different from $\nord{( \nord{F(x) G(x)}) H(x)}$. The formula \eqref{iterate} is sometimes called the \emph{iterate formula}. 
Such an iterate formula also holds for modules over vertex algebras, but \emph{not} in general for quasi-modules. This is where the `global'/`geometric' definitions of $Y_M$ and in particular $Y_W$ of Proposition \ref{prop: YMYW} will come into their own; see \S\ref{sec: quasiiterate}, below.
\end{rem}

The action of the operator $\D$ on $\ueva$ is given by
\begin{equation} \label{D act V}
\D \big( a_1(- n_1) \ldots a_j(- n_j) \vac \big) = \sum_{i = 1}^j a_1(- n_1) \ldots (\D a_i)(- n_i) \ldots a_j(- n_j) \vac
\end{equation}
for any $a_1, \ldots, a_j \in \vla$ and $n_1, \ldots, n_j \in \Z_{> 0}$. Indeed, we first note that for $a \in \vla$ we have
\begin{equation} \label{D act V 1}
\D( a(-1) \vac ) = ( a(-1) \vac )_{(-2)} \vac = a(-2) \vac = (\D a)(-1)\vac.
\end{equation}
The relation \eqref{D act V} then follows by writing the state as $(a_1(-1) \vac)_{(- n_1)} \ldots (a_j(-1) \vac)_{(- n_j)} \vac$ and using the fact that $\D$ is a derivation with respect to all $n^{\rm th}$-products.

\subsection{$\ueva$ as a $\Gamma$-vertex algebra}

\begin{prop} \label{prop: Aut V}
Any $\sigma\in \Aut \vla$ extends to a degree-preserving automorphism of the vertex algebra $\ueva$.
\end{prop}
\begin{proof}
First, $\sigma$ extends to an automorphism of $\U(\vla)$ by letting $\sigma(a(n)) := (\sigma a)(n)$ for any $a \in \vla$ and $n \in \Z$. 
This assignment, $\sigma(a(n)) := (\sigma a)(n)$, is compatible with the relation \eqref{D rel VLA} in $\U(\vla)$ since $\sigma \in \Aut \vla$ commutes with the operator $\D$ by definition.
Then $\sigma$ extends to a well-defined linear map $\ueva\to \ueva$ acting as
\begin{equation*}
\sigma\big( a_1(- n_1) \ldots a_j(- n_j) \vac \big) := (\sigma a_1)(- n_1) \ldots (\sigma a_j)(- n_j) \vac
\end{equation*}
(so in particular $\sigma\vac = \vac$) which is manifestly degree-preserving. Moreover, it clearly defines an automorphism of $\ueva$ as a vertex algebra.
\end{proof}

Given an automorphism $\sigma$ of $\ueva$, let $R_\alpha$ be the map of \eqref{Gamma act VA}. 
\begin{lem}\label{lem: alpha u}
Let $\alpha \in \Gamma$ and $A \in \ueva$. Then for each $u \in \Cx \setminus \Gamma \bm z$ we have
\begin{equation} \label{insertion point}
[\atp u A \otimes \atp {\bm z}{\bm m} ] 
= [R_{\alpha} \atp{\alpha u} A \otimes \atp {\bm z}{\bm m} ] ,
\end{equation}
where by equality we mean equality with respect to the canonical vector-space isomorphisms
\begin{equation*}
\left(\ueva \otimes \bigotimes_{i=1}^N M_{z_i} \right)\bigg/\U_{u,\bm z}^\Gamma(\vla) \cong_\C \bigotimes_{i=1}^N M_{z_i} \bigg/ \U_{\bm z}^\Gamma(\vla) \cong_\C \left(\ueva \otimes \bigotimes_{i=1}^N M_{z_i} \right)\bigg/\U_{\alpha u,\bm z}^\Gamma(\vla)
\end{equation*}
following from Proposition \ref{prop: VL remove} (together with the obvious analogs with a module at the origin).
\end{lem}
\begin{proof} 
The proof is given in \S\ref{proof: alpha u}.
\end{proof}

\begin{cor} For all $\alpha \in \Gamma$ and all $A\in \ueva$ and $u\in \Cx$,
$Y_W(R_{\alpha} A, \alpha u) = Y_W(A,u)$. 
\end{cor}
\begin{proof} This follows immediately, in view of the definition of $Y_W$ in part $(b)$ of Proposition \ref{prop: YMYW}. \end{proof}

\subsection{`Big' and `little' Lie algebras}
Recall from \S\ref{sec: vaLA} that vertex algebras are in particular vertex Lie algebras. It follows that
the universal enveloping vertex algebra $\ueva$ of a given vertex Lie algebra $\vla$ is \emph{itself} a vertex Lie algebra.\footnote{Alternatively, this follows just by checking that the map $Y(\cdot,x)_-$ obeys the axioms of a vertex Lie algebra.}
Given any commutative associative algebra $\mc A$, we thus have two Lie algebras associated to $\vla$, namely 
\be \Lie_{\mathcal{A}} \vla \quad\text{and}\quad \Lie_{\mathcal{A}} \ueva.\nn\ee 
We think of these as respectively the `little' and `big' Lie algebras associated to the pair $(\vla,\mc A)$. The following lemma supports this intuition. 

\begin{lem} \label{lem: gN UN}
The inclusion map of \eqref{vla in VV} is an embedding of vertex Lie algebras. Hence there is a natural embedding of Lie algebras
\be\label{Lie VLA VA}
\Lie_{\mathcal{A}} \vla \hookrightarrow \Lie_{\mathcal{A}} \ueva; 
\qquad \rho(a \otimes f) \mapsto \rho(a(-1) \vac \otimes f).
\ee
Under this embedding, the central element $\rho(\cent \otimes f)$, $f \in \mathcal{A}$ in $\Lie_{\mathcal{A}} \vla$ is sent to $\rho(\vac \otimes f)$.
\begin{proof}
We must check that $(a(-1)\vac)_{(n)}(b(-1)\vac) = (a_{(n)}b)(-1)\vac$ for all $n\in \Z_{\geq 0}$. And indeed $(a(-1)\vac)_{(n)}(b(-1)\vac)$ is equal to
\be  a(n) b(-1) \vac = [a(n),b(-1)]\vac = 
 \sum_{k=0}^\8 \binom n k (a_{(k)}b)(n-1-k)\vac = (a_{(n)}b)(-1) \vac \nn\ee
as required, where in the final step we used the fact that all non-negative modes of $a_{(k)}b$ annihilate the vacuum, while $\binom n k=0$ for $k>n$, so that $k=n$ is the only term in the sum that can contribute. 
The second part follows.
\end{proof}
\end{lem}

The constructions of \S\ref{sec: vaLA} all hold in particular when we take $\va$ to be the universal enveloping vertex algebra $\ueva$ of a vertex Lie algebra $\vla$, and further take the automorphism $\sigma \in \Aut \va$ to be the unique extension of an underlying automorphism $\sigma$ of $\vla$, as in Proposition \ref{prop: Aut V}. That is, all the `little' Lie algebras introduced in \S\ref{sec: LtoLie}--\ref{sec: tgla} have their `big' analogs. The next proposition makes precise the commutative diagram \eqref{se} from the introduction.
\begin{prop} \label{prop: gz Uz}
We have the following commutative diagrams of embeddings of Lie algebras:
\be
\begin{tikzpicture}    
\matrix (m) [matrix of math nodes, row sep=2em, column sep=3em]    
{
\U(\vla)_{\bm z}  & \U(\ueva)_{\bm z} \\
\U_{\bm z}^\Gamma(\vla) & \U^{\Gamma}_{\bm z}(\ueva) \\
};
\path[right hook->]
(m-1-1) edge (m-1-2)
(m-2-1) edge (m-2-2)
(m-2-1) edge (m-1-1)
(m-2-2) edge (m-1-2);
\end{tikzpicture}\qquad 
\begin{tikzpicture}    
\matrix (m) [matrix of math nodes, row sep=2em, column sep=3em]    
{
\U(\vla)_{\bm z,0}  & \U(\ueva)_{\bm z,0} \\
\U_{\bm z,0}^\Gamma(\vla) & \U^{\Gamma}_{\bm z,0}(\ueva) \\
};
\path[right hook->]
(m-1-1) edge (m-1-2)
(m-2-1) edge (m-2-2)
(m-2-1) edge (m-1-1)
(m-2-2) edge (m-1-2);
\end{tikzpicture}
\nn\ee
\end{prop}

\begin{proof}
We consider the first diagram. The second is similar.

The vertical embeddings follow from Proposition \ref{prop: LG embed}, applied to the vertex Lie algebras $\vla$ and $\ueva$ respectively. (For the latter, one replaces $\cent$ by $\vac$ and $I_{\Gamma\bm z}$ by $J_{\Gamma\bm z}$ as appropriate).

Turning to the horizontal embeddings, if we apply Lemma \ref{lem: gN UN} with $\mathcal{A} = \C((t))$ then we have an embedding of the local Lie algebras
\begin{equation*}
\U(\vla) \hookrightarrow \U(\ueva)\label{lge}
\end{equation*}
and hence an embedding
\begin{equation} \label{VLAi VAi}
j : \bigoplus_{i = 1}^N \U(\vla)_{z_i}  \hookrightarrow
 \bigoplus_{i = 1}^N \U(\ueva)_{z_i} .
\end{equation}
Recall the definition \eqref{LN def} of $\U(\vla)_{\bm z}$ as the quotient of $ \bigoplus_{i = 1}^N \U(\vla)_{z_i}$ by the ideal $I_N$ defined in \eqref{LN ideal I}. We have $j(I_N) = J_N$, where $J_N$ is the ideal defined in \eqref{UNV ideal J}.
We therefore deduce from Lemma \ref{simple lem} that there is an embedding of Lie algebras
\begin{equation} \label{LN in UVN}
\U(\vla)_{\bm z} \hookrightarrow \U(\ueva)_{\bm z}.
\end{equation}
Finally, applying Lemma \ref{lem: gN UN} to the commutative algebra $\mathcal{A} = \C^{\infty}_{\Gamma {\bm z}}(t)$ we obtain an embedding of Lie algebras $\text{Lie}_{\C^{\infty}_{\Gamma {\bm z}}(t)} \vla \hookrightarrow \text{Lie}_{\C^{\infty}_{\Gamma {\bm z}}(t)} \ueva$. Under this embedding, the image of the ideal $I_{\Gamma {\bm z}}$ in $\text{Lie}_{\C^{\infty}_{\Gamma {\bm z}}(t)} \vla$ defined by \eqref{LGz ideal I} coincides with the ideal $J_{\Gamma {\bm z}}$ in $\text{Lie}_{\C^{\infty}_{\Gamma {\bm z}}(t)} \ueva$ defined by \eqref{UGzV ideal J}. It follows from Lemma \ref{simple lem} that we have an embedding
\begin{equation} \label{emb1}
\big( \text{Lie}_{\C^{\infty}_{\Gamma {\bm z}}(t)} \vla \big) \big/ I_{\Gamma {\bm z}} \hookrightarrow \big( \text{Lie}_{\C^{\infty}_{\Gamma {\bm z}}(t)} \ueva \big) \big/ J_{\Gamma {\bm z}}.
\end{equation}

Let $i : \vla \hookrightarrow \ueva$ denote the embedding of vertex Lie algebras defined in \eqref{vla in VV}. 
By virtue of the relation \eqref{D act V 1}, the image under $i$ of the subspace $\im \D$ of $\vla$ is the intersection of $i(\vla)$ with the subspace $\im \D$ of $\ueva$. We may therefore choose our complement $(\ueva)^o$ of $\im \D \oplus \C \vac$ in $\ueva$ to contain $\vla^o$.
Hence, by restricting \eqref{emb1} to the subspace $\rho \big( \vla^o \otimes \C^{\infty}_{\Gamma {\bm z}} \big) \big/ I_{\Gamma {\bm z}}$ we obtain an embedding $\U_{\Gamma\bm z}(\vla) \hookrightarrow U_{\Gamma {\bm z}}(\ueva)$. Since this map commutes with the action of $\Gamma$ on both Lie algebras, its restriction to the subalgebra $\U_{\bm z}^\Gamma(\vla)$ of $\Gamma$-invariants yields the desired embedding.

By construction the diagram commutes.
\end{proof}

\subsection{Modules over the `big' Lie algebras $\U(\ueva)$ and $\U(\ueva)^\Gamma$} \label{sec: bigmod}
With the embedding of Lie algebras
\begin{equation*}
\U(\vla)\longhookrightarrow\U(\ueva),
\end{equation*}
from \eqref{lge} in hand, we certainly have that all modules over the big Lie algebra $\U(\ueva)$ pull back to modules over the little Lie algebra $\U(\vla)$. 
But, less obviously, it turns out that one can always go in the other direction too: every smooth module over $\U(\vla)$ has the structure of a smooth module over $\U(\ueva)$. 
Similarly, restricting \eqref{lge} to $\Gamma$-invariants yields an embedding
\begin{equation*}
\U(\vla)^\Gamma \longhookrightarrow\U(\ueva)^\Gamma,
\end{equation*}
so that modules over $\U(\ueva)^\Gamma$ pull back to modules over $\U(\vla)^\Gamma$. And again it turns out that in fact every smooth module over $\U(\vla)^\Gamma$ has the structure of a smooth module over $\U(\ueva)^\Gamma$.

Let us show that this is true. We consider the case of $\U(\vla)^\Gamma$.
Suppose, indeed, that $\M_0$ is any smooth module over $\U(\vla)^\Gamma$. To turn it into a module over $\U(\ueva)^\Gamma$ we need to specify how the latter acts. Now, $\U(\ueva)^\Gamma$ is spanned by formal sums $\sum_{n \geq N} f_n \, A^{\Gamma}(n)$ with $f_n \in \C$, $N \in \Z$ and where\footnote{ 
Recall that the action of $\alpha\in \Gamma$ on $\U(\ueva)$ is defined by $\alpha \on A(n) \equiv \alpha\on \rho(A\otimes t^n) := \rho(\alpha^{-1} R_\alpha A\otimes (\alpha^{-1}t)^n) = (R_\alpha A)(n)\alpha^{-n-1}$. See \S\ref{sec: tgla} and set `$\vla = \ueva$'.
} 
\begin{equation*}
A^{\Gamma}(n) := \sum_{\alpha\in \Gamma} \alpha\on(A(n)) = \sum_{\alpha\in \Gamma} \alpha^{-n-1} (R_\alpha A)(n)
\end{equation*}
for any $A\in \ueva$ and $n\in \Z$.
In particular, we have a natural surjective map
\begin{equation*}
\U(\ueva) \longtwoheadrightarrow \U(\ueva)^\Gamma.
\end{equation*}
(Of course, this is not a homomorphism of Lie algebras in general.)

The Lie bracket of $\U(\ueva)^\Gamma$ is therefore completely specified by the Lie brackets of such projected modes. 
Let us compute the Lie bracket of $A^{\Gamma}(m)$ with $B^{\Gamma}(n)$.
We have  
\begin{multline} \label{commutationrels}
\left[ A^{\Gamma}(m), B^{\Gamma}(n) \right] = \sum_{\beta \in \Gamma} \beta . \left[ A(m), B^{\Gamma}(n) \right] 
= \sum_{\alpha , \beta \in \Gamma} \beta . \left[ A(m), (R_\alpha B)(n)\right] \alpha^{-n-1}\\
= \sum_{\alpha\in \Gamma} \sum_{k\geq 0} \binom m k \alpha^{-n-1} \left( A_{(k)} (R_\alpha B)\right)^{\Gamma}(m+n-k).
\end{multline}

Now we have, as claimed, the following proposition. Recall the definition of $Y_W(A,x)$ from  Proposition \ref{prop: YMYW}. 

\begin{prop}\label{prop: com}
Let $\M_0$ be a smooth module over $\U(\vla)^\Gamma$. Then there is a well-defined smooth $\U(\ueva)^\Gamma$-module structure on $\M_0$ given by
\begin{equation} \label{action}
A^{\Gamma}(n) \on m_0 = A^W_{(n)} m_0,
\end{equation}
for all $m_0\in \M_0$, where $A^W_{(n)}$ are the elements of $\End\M_0$ defined by 
\be Y_W(A,x) =: \sum_{n\in \Z} A^W_{(n)} x^{-n-1} .\nn\ee
\end{prop}
\begin{proof}
We have $Y_W(A,x)\in \Hom(\M_0, \M_0((x)))$. That means for all $m_0\in \M_0$, $Y(A,x)m_0$ is a  Laurent series with coefficients in  $\M_0$. Hence $A^W_{(n)} m_0= 0$ for $n\gg 0$. So if this is an action at all it is certainly smooth.   
What we have to show is that
\be \left[ A^W_{(m)}, B^W_{(n)} \right]
= \sum_{k=0}^\8 \binom m k   \sum_{\alpha\in \Gamma} \alpha^{-n-1}    \left( A_{(k)} (R_\alpha B) \right)^W_{(m+n-k)},
\label{Wcom}\ee
in agreement with \eqref{commutationrels}. We do so in \S\ref{sec: proofofcommutator}. 
\end{proof}

This construction of modules for $\U(\ueva)^\Gamma$ appears not to be in the literature for the general vertex Lie algebra $\vla$, but cf. \cite[Proposition 7.4]{Li4} and \cite[Theorem 3]{Szc}.
For clarity let us also state explicitly the following, which is standard, see e.g. \cite{FB} (and which can be recovered by setting $\Gamma=\{1\}$ in the preceding result). 

\begin{prop}\label{prop: com2}
Let $\M$ be a smooth module over $\U(\vla)$. Then there is a well-defined smooth $\U(\ueva)$- module structure on $\M$ given by
\be A(n)\on m =  A^M_{(n)} m, \label{action}\ee
for all $m \in \M$, where $A^M_{(n)}$ are the elements of $\End\M$ defined by 
\be Y_M(A,x) =: \sum_{n\in \Z} A^M_{(n)} x^{-n-1}.\nn\ee
\qed\end{prop}

\begin{rem} 
It is important to avoid one possible confusion here. Let $\M$ be a smooth module over $\U(\vla)$.
Then $\M$ is in particular an $\U(\vla)^\Gamma$-module, by pulling back the action by the natural embedding $\U(\vla)^\Gamma\hookrightarrow \U(\vla)$ of Lie algebras. By Proposition \ref{prop: com2}, $\M$ carries an action of $\U(\ueva)$. Let us use (here only) $\triangleright$ to denote this action: $A(n) \triangleright m = A_{(n)}^M m$.
By the natural embedding $\U(\ueva)^\Gamma\hookrightarrow \U(\ueva)$ of the big Lie algebras, $\triangleright$ pulls back to give an action of $\U(\ueva)^\Gamma$ on $\M$. The $\triangleright$-action of the element $A^{\Gamma}(n)\in \U(\ueva)^\Gamma$ is, by definition,
\begin{equation*}
A^{\Gamma}(n) \triangleright m = \sum_{\alpha\in \Gamma} \alpha^{-n-1} (R_\alpha A)(n) \triangleright m
= \sum_{\alpha\in \Gamma} \alpha^{-n-1} (R_\alpha A)^M_{(n)} m.
\end{equation*}
\emph{However} this action of $\U(\ueva)^\Gamma$ on $\M$ does \emph{not} coincide with that of Proposition \ref{prop: com} in general. That is, 
\be A^W_{(n)} \quad\text{is not in general equal to}\quad  \sum_{\alpha\in \Gamma} \alpha^{-n-1} (R_\alpha A)^M_{(n)}.
\nn  \ee
(They are equal in the special case $A=a(-1)\vac$, $a\in \vla$.)

Of course, in general a module $\M_0$ over $\U(\vla)^\Gamma$ need not be a pullback of any module over $\U(\vla)$, and if it is not then $\sum_{\alpha\in \Gamma} \alpha^{-n-1} (R_\alpha A)^M_{(n)}$ is not even defined and there is no scope for confusion.
\end{rem}

\subsection{The quasi-iterate formula for $Y_W$}\label{sec: quasiiterate}
At this stage, the module and quasi-module maps $Y_M$ and $Y_W$ are defined by Proposition \ref{prop: YMYW}, which will be very helpful in proofs (see \S\ref{sec: proofs}) but is less useful in concrete computations. As mentioned in Remark \ref{rem: tostandard}, the module map $Y_M$ can be defined more explicitly by an iterate formula which is structurally the same as that for the state-field correspondence $Y$ itself, \eqref{iterate}. That is, for all $a\in \vla$,
\begin{subequations}
\label{ymdeff}
\be Y_M(a(-1)\vac,u) = \sum_{n\in \Z} a^{M}_{(n)} u^{-n-1},\ee
while for all $A,B\in \ueva$  -- and in particular, when $A= a(-1)\vac$ --
\be Y_M( A(-1) B, u) = \; \nord{Y_M( A, u) Y_M( B, u)}. \label{ynord}\ee
\end{subequations}
The following result gives a similarly explicit recursive definition of the quasi-module map $Y_W$. 

\begin{thm}[Quasi-iterate formula] \label{thm: quasi rec}
Given any smooth $\U(\vla)^\Gamma$-module $\M_0$, let $Y_W(\cdot,x):\ueva\to \Hom(\M_0,\M_0((x)))$ be the map defined in part (\ref{YWdef}) of Proposition \ref{prop: YMYW}. Then $Y_W(\vac,u)=\id$, while for all $a\in \vla$, 
\begin{equation*}
Y_W(a(-1)\vac, u) =
\sum_{\alpha \in \Gamma} \sum_{n \in \mathbb{Z}} (R_{\alpha} a)(n) (\alpha u)^{-n-1} 
\end{equation*}
and for all $a\in \vla$ and all $B\in \ueva$,
\begin{equation} \label{YW rec}
Y_W( a(-1) B, u) = \; \nord{Y_W( a(-1) \vac, u) Y_W( B, u)} + \sum_{\substack{\alpha\in \Gamma\\\alpha \neq 1}} \sum_{n \geq 0} \frac{1}{( (\alpha - 1) u)^{n+1}} Y_W\big( (R_{\alpha} a)(n) B, u \big).
\end{equation}
\end{thm}
\begin{proof}
See \S\ref{proof: YMYW} below. 
\end{proof}
See also Corollary \ref{cor: quasi rec} below. 

\begin{rem}
From one perspective, the reason that $Y_M$ obeys the usual iterate formula is that $Y_M$ obeys the same Borcherds identity as $Y$. This fact also suffices to ensure that the modes of $Y_M$ obey the correct commutation relations, \ie the same as those of the modes of $Y$; see \eqref{com Am Bn} and \eqref{VLA bracket loop}. 

In constrast, $Y_W$ obeys only the quasi-Borcherds identity. It is not clear, at least to the authors, how  one could derive the quasi-iterate formula \eqref{YW rec}, or the commutator formula \eqref{Wcom} for the modes of $Y_W$, from the quasi-Borcherds identity alone. The proofs in \S\ref{sec: proofs} below use rather the global definition, Proposition \ref{prop: YMYW}, of $Y_W$, together with the residue theorem.
\end{rem}

\begin{rem}
A somewhat subtle effect of the pole term in $u$ on the right of \eqref{YW rec} is that, while every smooth $\U^+(\vla)$-module becomes an $\U^+(\ueva)$-module, it is \emph{not} true that every smooth $\U^+(\vla)^\Gamma$-module becomes an $\U^+(\ueva)^\Gamma$-module.

To explain this statement, suppose that $\M$ is any smooth module over $\U^+(\vla)$. Then we have the induced $\U(\vla)$-module $\MM:= \Ind_{\U^+(\vla)}^{\U(\vla)} \M := U(\U(\vla))\otimes_{U(\U^+(\vla))} M$. By Proposition \ref{prop: com2}, $\MM$ becomes an $\U(\ueva)$-module and hence, pulling back by the embedding $\U^+(\ueva)\hookrightarrow \U(\ueva)$,  an $\U^+(\ueva)$-module. But more than that, this $\U^+(\ueva)$-module $\MM$ has $\M$ as a submodule. To see this, note that, in the notation of \S\ref{sec: cofb}, 
\be(:F(x)G(x):)_- = F(x)_- G(x)_-.\nn\ee 
That implies that the action on $\MM$ of any non-negative mode $A(n)$ of any state $A\in \ueva$ can be expressed in terms of non-negative modes of states in $\vla\hookrightarrow \ueva$. And $\M$ closes under the latter, by definition. 

On the other hand, suppose $W$ is a smooth module over $\U^+(\vla)^\Gamma$. Then the induced module $\WW := \Ind_{\U^+(\vla)^\Gamma}^{\U(\vla)^\Gamma} W$ becomes a module over $\U^+(\ueva)^\Gamma$ by the same argument. \emph{But} the pole term in $u$ on the right of \eqref{YW rec} means that the action of $A^\Gamma(n)$, $n\geq 0$, on $\WW$ cannot in general be expressed in terms only of non-negative modes of states in $\vla$ (it may also involve negative modes). So in general $A^\Gamma(n) \on W \not\subset W$. 
\end{rem}

\section{Main results} \label{sec: mr}
\subsection{Big and little cyclotomic coinvariants}\label{sec: big=little}
We now return to our main objects of interest, cyclotomic coinvariants. 

As in \S\ref{sec: indg}, let us once again take $\M_{z_i}$ to be a smooth module over $\U(\vla)_{z_i}$ for each $i = 1, \ldots, N$. 
By Proposition \ref{prop: com2}, these are also smooth modules over the big Lie algebras $\U(\ueva)_{z_i}$.
Let $\M_0$ be a smooth module over $\U(\vla)^\Gamma$. By Proposition \ref{prop: com}, $\M_0$ is also a smooth module over the big Lie algebra $\U(\ueva)^\Gamma$.
In view of the `vertical' embeddings of big Lie algebras from Proposition \ref{prop: gz Uz}, we can form the spaces of coinvariants with respect to the big Lie algebras, simply replacing $\vla$ by $\ueva$ in \eqref{tc1} and \eqref{tc2}. For example in place of \eqref{tc2} we can form  the space of coinvariants
\be  \left.\bigotimes_{i=1}^N\M_{z_i} \otimes \M_0\right/ \U^{\Gamma}_{\bm z,0}(\ueva).\nn\ee
The natural question is then how such spaces of coinvariants with respect to big Lie algebras are related to the spaces of coinvariants with respect to the little Lie algebras. In view of the `horizontal' embeddings of Lie  
algebras in Proposition \ref{prop: gz Uz}, we certainly have surjective linear maps
\begin{alignat}{3} &\left.\bigotimes_{i=1}^N\M_{z_i}  \right/ \U_{\bm z}^\Gamma(\vla) 
 &&\longtwoheadrightarrow 
&&\left.\bigotimes_{i=1}^N\M_{z_i}  \right/ \U_{\bm z}^\Gamma(\ueva) \label{btc1}\\
 &\left.\bigotimes_{i=1}^N\M_{z_i}  \otimes \M_{0}\right/ \U^\Gamma_{\bm z,0}(\vla) 
 &&\longtwoheadrightarrow 
&&\left.\bigotimes_{i=1}^N\M_{z_i} \otimes \M_{0}\right/ \U^\Gamma_{\bm z,0}(\ueva) 
\label{btc2}\end{alignat}
as a direct application of the following elementary lemma.
\begin{lem}\label{lem: ab} 
For any Lie algebras $\mf a$ and $\mf b$, any homomorphism $\phi:\mf a\to \mf b$ and any $\mf b$-module $V$, there is a surjective linear map $V/\mf a \twoheadrightarrow V/\mf b$ sending $v+\phi(\mf a)\on V$ to $v+ \mf b \on V$.
\end{lem}
\begin{proof} This is obvious, though it is worth emphasizing that the map is only well-defined by virtue of the fact that the $\mf a$-module structure of $V$ is the pull-back via $\phi$, \ie  $V/\mf a:= V/\mf a\on V := V/\phi(\mf a)\on V$, and $\phi(\mf a)\subseteq \mf b$.
\end{proof}

The remarkable fact is that the second of these surjections is actually a bijection (though the first is \emph{not}, in general: see Example \ref{exmp: need0}). That is, we have the following.

\begin{thm} \label{thm: big=little}
There is a linear isomorphism
\begin{align}\label{big swapping}
          \left.\bigotimes_{i=1}^N\M_{z_i}  \otimes \M_0\right/ \U^{\Gamma}_{\bm z, 0}(\ueva) 
&\cong_{\C} \left.\bigotimes_{i=1}^N\M_{z_i}  \otimes \M_0\right/ \U^{\Gamma}_{\bm z, 0}(\vla). 
\end{align}
\end{thm}
\begin{proof} 
We must show that \eqref{btc2} is injective. This amounts to showing that
\begin{equation*}
\U^\Gamma_{\bm z, 0}(\ueva) \on \left(\bigotimes_{i=1}^N \M_{z_i} \otimes \M_0 \right)
\subseteq \U^{\Gamma}_{{\bm z}, 0}(\vla) \on \left(\bigotimes_{i=1}^N \M_{z_i} \otimes \M_0 \right).
\end{equation*}
(The inclusion the other way is the obvious one, as in Lemma \ref{lem: ab}.)
Now, $\U^\Gamma_{\bm z,0}(\ueva)$ is spanned by elements of the form $\sum_{\alpha\in \Gamma} \alpha\on(A f)$ with $A\in \ueva$ and $f\in \C^\8_{\bm z,0}(t)$. 
So what we must show is that the class of $\big( \sum_{\alpha \in \Gamma} \alpha . (A f) \big). ({\bm m} \otimes m_0)$
in $\left.\bigotimes_{i=1}^N \M_{z_i} \otimes \M_0 \right/ \U^\Gamma_{\bm z,0}(\vla)$ is the zero class,
\begin{equation} \label{big swapping 2}
\bigg[ \bigg( \sum_{\alpha \in \Gamma} \alpha\on (A f) \bigg)\on 
(\atp{\bm z}{\bm m} \otimes \atp{0}{m_0}) \bigg] = 0.
\end{equation}
We have 
\be \sum_{\alpha\in\Gamma}\alpha\on(A f) = \sum_{\alpha \in \Gamma} \alpha^{-1} R_{\alpha}A \otimes f(\alpha^{-1} t)\nn\ee 
as in \S\ref{sec: tgla}. Recall how this element acts. According to Proposition \ref{prop: LG embed prime} (applied to the vertex Lie algebra $\ueva$ rather than $\vla$) we are first to take the Laurent  expansions about the points $z_i$, $i=1,\dots,N$, and about $0$, to obtain an element of $\U(\ueva)_{\bm z,0}$. The latter then acts, summand by summand, according to Propositions \ref{prop: com} and \ref{prop: com2}. This amounts to
\begin{multline} \label{big Lie action}
\left( \sum_{\alpha \in \Gamma} \alpha^{-1} R_\alpha A \otimes f(\alpha^{-1} t) \right) . (\bm m \otimes m_0)\\
= \sum_{i=1}^N \sum_{\alpha \in \Gamma} \res_{t - z_i} \big( \iota_{t - z_i} \alpha^{-1} f(\alpha^{-1} t) \big) (Y_M(R_{\alpha} A, t - z_i)_{z_i} \bm m
\otimes m_0)\\
 + \res_t \big(\iota_t f(t)\big) (\bm m \otimes Y_W(A,t) m_0). 
\end{multline}

Now to establish \eqref{big swapping 2} consider the space of coinvariants $\ueva \otimes  \bigotimes_{i=1}^N\M_{z_i} \otimes \M_0 \big/ \U^\Gamma_{u,\bm z,0}(\vla)$. By Proposition \ref{prop: VL remove}  we have the linear isomorphism 
\be 
\left.\ueva \otimes  \bigotimes_{i=1}^N\M_{z_i} \otimes \M_0\right/ \U^\Gamma_{u,\bm z,0}(\vla)
\cong_\C 
\left.\bigotimes_{i=1}^N\M_{z_i} \otimes \M_0\right/ \U^\Gamma_{\bm z,0}(\vla).
\nn\ee
An element, say 
\be [\atp uA \otimes \atp{\bm z}{\bm m}\otimes \atp{0}{m_0}],\nn\ee 
of this space of coinvariants depends rationally on $u$, with poles at most at the points $\Gamma \bm z$ and $0$, and at least a simple zero at infinity. Given any $f(u) \in \C^{\infty}_{\Gamma {\bm z}, 0}(u)$, 
\be f(u)[\atp uA \otimes \atp{\bm z}{\bm m}\otimes \atp{0}{m_0}],\nn\ee 
has at least a double zero at infinity. Therefore, by the residue theorem we have
\be
0 =  \sum_{i=1}^N \sum_{\alpha \in \Gamma}\res_{u - \alpha^{-1} z_i} \iota_{u - \alpha^{-1} z_i} f(u)  [\atp uA \otimes \atp{\bm z}{\bm m}\otimes \atp{0}{m_0}]
 + \res_u \iota_u f(u) 
[\atp uA \otimes \atp{\bm z}{\bm m} \otimes \atp{0}{m_0}].\nn
\ee
Introducing the new variable $t = \alpha u$ in the sum and renaming the variable $u$ to $t$ in the final term, we may rewrite this as
\be
0 =  \sum_{i=1}^N \sum_{\alpha \in \Gamma}\res_{t - z_i} \iota_{t - z_i} \alpha^{-1} f(\alpha^{-1}t)  [\atp{\alpha^{-1}t}A \otimes \atp{\bm z}{\bm m}\otimes \atp{0}{m_0}] 
 + \res_t \iota_t f(t) 
[\atp tA \otimes \atp{\bm z}{\bm m}\otimes \atp{0}{m_0}].\nn
\ee
Applying Lemma \ref{lem: alpha u} and then Proposition \ref{prop: YMYW}, we obtain
\begin{equation*}
0 =  \sum_{i=1}^N \sum_{\alpha \in \Gamma}\res_{t - z_i} \iota_{t - z_i} \alpha^{-1}f(\alpha^{-1}t)  [Y_M(R_\alpha A,t-z_i)_{z_i} \atp{\bm z}{\bm m}\otimes \atp{0}{m_0}] + \res_t \iota_t f(t) 
[\atp{\bm z}{\bm m}\otimes Y_W(A,t) \atp{0}{m_0}].
\end{equation*}
which, given \eqref{big Lie action}, is \eqref{big swapping 2}, as required.
\end{proof}

Intuitively, it is helpful to think about Theorem \ref{thm: big=little} in terms of `swapping' -- see \cite[\S2.7]{VY} and \cite[\S3]{FFR} for the meaning of `swapping', which is also illustrated in Example \ref{exmp: need0} below. The content of the theorem is that `little swapping implies big swapping', i.e. whenever one is allowed to `swap' using elements of the global little Lie algebra $\U_{\bm z,0}^\Gamma(\vla)$ then `for free' one is also allowed to `swap' using elements of the global big Lie algebra $\U_{\bm z,0}^\Gamma(\ueva)$.

The presence of the module $\M_0$ at the origin is crucial. That is, generically it will be the case that  
\be \left.\bigotimes_{i=1}^N\M_{z_i}  \right/ \U^{\Gamma}_{\bm z}(\ueva) 
\not\cong_{\C} \left.\bigotimes_{i=1}^N\M_{z_i}  \right/ \U^{\Gamma}_{\bm z}(\vla).\nn\ee 
The following example illustrates this.

\begin{exmp}[On the need for the module at the origin in Theorem \ref{thm: big=little}]\label{exmp: need0}
Take $N=1$, $z_1=u\in \Cx$ and $\M_{z_1} = \ueva$. 
Consider the class
\be [ a(-1)b(-1) \atp u \vac ] \quad\text{with}\quad a,b\in \vla\nn\ee
in the space of coinvariants $\left.\ueva\right/\U_{u}^\Gamma(\vla)$.  We have
\begin{equation*}
\iota_{t-u} \sum_{\alpha \in \Gamma} \frac{\alpha^{-1} R_{\alpha} a}{\alpha^{-1} t - u}
= a(-1) 
- \sum_{n\geq 0} \sum_{\alpha\in \Gamma\setminus \{1\}} \frac{(R_\alpha a)(n)}{\left((\alpha-1)u\right)^{n+1}}
\end{equation*}
and hence by swapping using the $\Gamma$-equivariant rational function $\sum_{\alpha \in \Gamma} \frac{\alpha^{-1} R_{\alpha} a}{\alpha^{-1} t - u}$ we find
\be [ a(-1)b(-1) \atp u \vac ] = 
 \sum_{n\geq 0} \sum_{\alpha\in \Gamma\setminus \{1\}} \frac{1}{\left((\alpha-1)u\right)^{n+1}}
  [ (R_\alpha a)(n)b(-1) \atp u \vac ]. \nn\ee 
In general, this need not be the zero class. 
To be concrete, let us take $\vla$ to be the vertex Lie algebra of Example \ref{ex: Heis VLA}, and let $a\in \mf n$, $b\in \mf n^*$ be such that $(a,b) =1 $. Suppose for simplicity that $\sigma = \id$ (but take $T> 1$).
Then $a(0) b(-1) \vac = \vac $, $a(n) b(-1) \vac = 0$ for all $n\geq 1$, and  
\be  [ a(-1)b(-1) \atp u \vac ] = [ \atp u \vac]  \sum_{\alpha\in \Gamma\setminus \{1\}} \frac{1}{(\alpha-1)u},\label{ssr}\ee
which is not zero in $\ueva\big/ \U^{\Gamma}_{u}(\vla)$ by Proposition \ref{prop: VL remove} in the case $N=0$.

On the other hand, in $\ueva\big/ \U^{\Gamma}_{u}(\ueva)$ we are allowed to `swap' using 
\be  \sum_{\alpha \in \Gamma} \alpha \on \left( \frac{a(-1)b(-1) \vac} {t-u} \right) 
   = \sum_{\alpha \in \Gamma} \frac{\alpha^{-1} R_{\alpha} a(-1)b(-1)\vac}{\alpha^{-1} t - u}
\nn,\ee
whose Laurent expansion is
\be \iota_{t-u} \sum_{\alpha \in \Gamma} \alpha \on \left( \frac{a(-1)b(-1) \vac} {t-u} \right)
= \left(a(-1)b(-1)\vac\right)(-1)
- \sum_{n\geq 0} \sum_{\alpha\in \Gamma\setminus \{1\}} 
\frac{(R_\alpha a(-1)b(-1)\vac)(n)}{\left((\alpha-1)u\right)^{n+1}}. \nn\ee
By the creation axiom we have $(R_\alpha a(-1)b(-1)\vac)(n) \vac = 0$ for all $n \geq 0$ and hence 
\be \left( \sum_{\alpha \in \Gamma} \alpha \on \left( \frac{a(-1)b(-1) \vac} {t-u} \right) \right)\on \vac  = \big(a(-1)b(-1)\vac\big)(-1) \vac = a(-1)b(-1)\vac , \nn\ee
which says that in $\ueva\big/ \U^{\Gamma}_{u}(\ueva)$, the class of $a(-1)b(-1)\vac$ vanishes. Therefore -- since the relation \eqref{ssr} still holds in $\ueva\big/ \U^{\Gamma}_{u}(\ueva)$ -- we have here that
\be [ \atp u \vac ] = 0 \quad \text{in} \quad \ueva\big/ \U^{\Gamma}_{u}(\ueva),\nn\ee
demonstrating that in this case the surjection 
\be\ueva\big/ \U^{\Gamma}_{u}(\vla)\longtwoheadrightarrow\ueva\big/ \U^{\Gamma}_{u}(\ueva)\nn\ee
has a non-trivial kernel. 
\end{exmp}

Recall the quasi-iterate formula, Theorem \ref{thm: quasi rec}, for computing $Y_W$ for composite states. We now have the following minor addition.
\begin{cor}\label{cor: quasi rec}
Given any $\U(\vla)^\Gamma$-module $\M_0$, let $Y_W(\cdot,x):\ueva\to \Hom(\M_0,\M_0((x)))$ be the map defined in Proposition \ref{prop: YMYW}. Then for all $A,B\in \ueva$,
\begin{equation} \label{YW rec2}
Y_W( A(-1) B, u) = \; \nord{Y_W( A(-1) \vac, u) Y_W( B, u)} + \sum_{\substack{\alpha\in \Gamma\\\alpha \neq 1}} \sum_{n \geq 0} \frac{1}{( (\alpha - 1) u)^{n+1}} Y_W\big( (R_{\alpha} A)(n) B, u \big).
\end{equation}
\end{cor}
\begin{proof}
This follows by the same argument as in the proof of Theorem \ref{thm: quasi rec} -- see \S\ref{proof: YMYW} -- now that we know that Theorem \ref{thm: big=little} holds. (That is, informally speaking, now that we know  we are allowed to `swap big states'.)
\end{proof}

\subsection{Functoriality of cyclotomic coinvariants}
Let $\vla$ and $\vma$ be $\Gamma$-vertex Lie algebras and suppose that we have a homomorphism of their universal enveloping vertex algebras,
\begin{equation} \label{hom VA}
\rho : \ueva \rightarrow \ueval,
\end{equation}
which commutes with the action of $\Gamma$ (\ie a homomorphism of $\Gamma$-vertex algebras).

Note that this homomorphism need not restrict to a homomorphism of vertex Lie algebras $\vla\to\vma$, because the image of $\vla\subset \ueva$ need not lie in $\vma\subset \ueval$.

Checking the definitions, one straightforwardly verifies that $\rho$ induces homomorphisms of the various big Lie algebras: schematically
\be \U^\bullet_\bullet(\ueva)^\bullet \to \U^\bullet_\bullet(\ueval)^\bullet \label{bighoms}\ee
(for example $\U(\ueva) \to \U(\ueval)$, $\U^{\Gamma}_{\bm z}(\ueva) \to \U^{\Gamma}_{\bm z}(\ueval)$, $\U^+(\ueva)^\Gamma \to \U^+(\ueval)^\Gamma$).

Now let $\M_{z_i}$ and $\M_0$ be as before, but now with respect to $\vma$ rather than $\vla$. That is, suppose $\M_{z_i}$ is a module over $\U(\vma)_{z_i}$ for each $i = 1, \ldots, N$ and suppose $M_0$ is a module over $\U(\vma)^\Gamma$. They become modules over the corresponding big local Lie algebras, $\U(\ueval)_{z_i}$ and $\U(\ueval)^\Gamma$, just as discussed in \S\ref{sec: bigmod}. Pulling back by the homomorphisms \eqref{bighoms}, they become modules over $\U(\ueva)_{z_i}$ and $\U(\ueva)^\Gamma$. The following is then immediate from Lemma \ref{lem: ab}.
\begin{prop} \label{prop: functoriality}
There are surjective linear maps
\begin{alignat}{3} &\left.\bigotimes_{i=1}^N\M_{z_i}  \right/ \U_{\bm z}^\Gamma(\ueva) 
 &&\longtwoheadrightarrow 
&&\left.\bigotimes_{i=1}^N\M_{z_i}  \right/ \U_{\bm z}^\Gamma(\ueval) \label{f1}\\
 &\left.\bigotimes_{i=1}^N\M_{z_i}  \otimes \M_0\right/ \U^\Gamma_{\bm z,0}(\ueva) 
 &&\longtwoheadrightarrow 
&&\left.\bigotimes_{i=1}^N\M_{z_i} \otimes \M_0\right/ \U^\Gamma_{\bm z,0}(\ueval) 
\label{f2}\end{alignat}
\qed\end{prop}

Combining Theorem \ref{thm: big=little} with Proposition \ref{prop: functoriality}, one has an important corollary.

\begin{cor} \label{cor: functoriality}
There is a surjective linear map
\begin{alignat}{3} 
 &\left.\bigotimes_{i=1}^N\M_{z_i}  \otimes \M_0 \right/ \U^\Gamma_{\bm z,0}(\vla) 
 &&\longtwoheadrightarrow 
&&\left.\bigotimes_{i=1}^N\M_{z_i} \otimes \M_0 \right/ \U^\Gamma_{\bm z,0}(\vma) 
\label{f2}\end{alignat}
\qed\end{cor}
Again, it should be stressed that this is not in general true if the module $\M_0$ is omitted. However, it may happen that one is really interested in a space of coinvariants of the form $\left.\bigotimes_{i=1}^N\M_{z_i}  \right/ \U^{\Gamma}_{\bm z}(\vla)$, not $\left.\bigotimes_{i=1}^N\M_{z_i}  \otimes \M_0 \right/ \U^{\Gamma}_{\bm z, 0}(\vla)$. This is so in the case of the cyclotomic Gaudin models in \cite{VY}, for example. To deal with this, we now prove Theorem \ref{thm: Gaudin}, below, which says that under certain conditions we can improve on Corollary \ref{cor: functoriality}.

First, define 
\be \uevaG := \Ind_{\U^+(\vla)^\Gamma}^{\U(\vla)^\Gamma} \C\vac, \label{uevaGdef}\ee
cf. \eqref{Vvla def}. The proof of the following is very similar to that of Proposition \ref{prop: VL remove}. 
\begin{prop}\label{prop: VL remove zero}
There is a linear isomorphism
\begin{equation*}
\left.\bigotimes_{i = 1}^N \M_{z_i} \otimes \uevaG \right/ \U^\Gamma_{\bm z,0}(\vla) \xrightarrow{\sim_\C} \bigotimes_{i = 1}^N \M_{z_i} \bigg/ \U^\Gamma_{\bm z}(\vla)
\end{equation*}
sending the class of $\bm m \otimes \vac$ to the class of $\bm m$.
\qed\end{prop}
We also need the following trivial fact.
\begin{lem}\label{tlem} Let $V$ and $W$ be modules over a Lie algebra $\g$ and $\phi:V\to W$ a map of $\g$-modules. Then there is a well-defined map $V/\g\to W/\g$ sending $[v]\mapsto [\phi(v)]$. 
\end{lem}
\begin{proof} If $[v'] = [v]$ then $v'=v+ \sum_{i} A_i \on x_i$ for some finite collection of $A_i\in \g$ and $x_i\in V$. Then $\phi(v') = \phi(v) + \sum_i\phi(A_i\on x_i) = \phi(v) +\sum_i A_i\on \phi(x_i)$ so $[\phi(v')] = [\phi(v)]$ as required.
\end{proof}

\begin{thm}\label{thm: Gaudin}
Suppose that there exists a non-zero vector $m_0\in \M_0$ with the property  that, 
for all $a\in \vla$,
$\iota_u [ \atp u {\rho(a(-1) \vac)} \otimes \dots \otimes \atp 0 {m_0} ]$
is a Taylor series in $u$.
Then there is a well-defined linear map
\be \left.\bigotimes_{i=1}^N\M_{z_i}  \right/ \U^{\Gamma}_{\bm z}(\vla) \longrightarrow \left.\bigotimes_{i=1}^N\M_{z_i}  \otimes M_0 \right/ \U^{\Gamma}_{\bm z,0}(\vma) \nn\ee
which sends the class of any $\bm m \in \bigotimes_{i=1}^N\M_{z_i} $ in $\left.\bigotimes_{i=1}^N\M_{z_i}  \right/ \U^{\Gamma}_{\bm z}(\vla)$ to the class of $\bm m\otimes m_0$ in $\left.\bigotimes_{i=1}^N\M_{z_i}  \otimes \M_0 \right/ \U^{\Gamma}_{\bm z,0}(\vma)$. 
\end{thm}
\begin{proof}
In view of Proposition \ref{prop: YMYW} part (b) and 
Proposition \ref{prop: wiiz}, the given property of $m_0$ implies that $Y_W(\rho(a(-1) \vac), u)_- m_0 = 0$. That is, $\rho(a(-1)\vac)^W_{(n)} m_0 =0$ for all $n\geq 0$, in the notation of Proposition \ref{prop: com2}, or in other words $\U^+(\vla)^\Gamma\on m_0 = 0$. It follows that the $\U(\vla)^\Gamma$-submodule of $M_0$ through $m_0$ is isomorphic to a quotient of the induced module $\uevaG$,
\ie there are $\U(\vla)^\Gamma$-module maps
\be \uevaG \twoheadrightarrow \U(\vla)^\Gamma\on m_0  \hookrightarrow M_0.\label{gmm}\ee
Thus we have the following linear map which indeed sends the class of $\bm m$ to the class of $\bm m\otimes m_0$, as required:
\begin{multline*} \left.\bigotimes_{i=1}^N\M_{z_i}  \right/ \U^{\Gamma}_{\bm z}(\vla) 
\to \left.\bigotimes_{i = 1}^N \M_{z_i} \otimes \uevaG \right/ \U^\Gamma_{\bm z,0}(\vla)\\
\to \left.\bigotimes_{i = 1}^N \M_{z_i} \otimes M_0 \right/ \U^\Gamma_{\bm z,0}(\vla) 
\to \left.\bigotimes_{i=1}^N\M_{z_i} \otimes M_0 \right/ \U^{\Gamma}_{\bm z,0}(\vma).\end{multline*}
Here the first map is the inverse of the linear isomorphism of Proposition \ref{prop: VL remove zero}, the second is by Lemma \ref{tlem} and \eqref{gmm} and the final map is from Corollary \ref{cor: functoriality}.
\end{proof}

\section{Proofs}\label{sec: proofs}

\subsection{Proof of Lemma \ref{lem: VLA central}}\label{sec: proofVLAcentral}
Using \eqref{D lie alg}, any element in $\Lie_{\mathcal{A}} \vla$ can always be written as a finite sum of terms of the form $\rho(a \otimes f)$ with $a \in \vla^o \oplus \ker \D$ and $f \in \mathcal{A}$. Moreover, if $a \in \ker \D$ then the relation \eqref{D lie alg} also implies that $\rho(a \otimes f) = 0$ for any $f \in \im\, \delta$. It follows that the restriction of the quotient map $\rho$ to the subspace $\vla^o \otimes \mathcal{A} \oplus \ker \D \otimes \mathcal{A}^o$ of $\vla \otimes \mathcal{A}$ is a surjection onto $\Lie_{\mathcal{A}} \vla$.

By definition, the linear map $\rho$ has kernel $\im \partial$. Showing its restriction to $\vla^o \otimes \mathcal{A} \oplus \ker \D \otimes \mathcal{A}^o$ is injective thus amounts to showing that
\begin{equation*}
\im \partial \cap \big( \vla^o \otimes \mathcal{A} \oplus \ker \D \otimes \mathcal{A}^o \big) = 0.
\end{equation*}
To show this we shall suppose there is a non-zero vector in this intersection and obtain a contradiction. So consider $\sum_{i \in I} \big( \D a_i \otimes f_i + a_i \otimes \delta f_i \big) \in \im \partial$ with $I$ a non-empy finite set, and non-zero $a_i \in \genvla'$, $f_i \in \mathcal{A}$. 
By reorganising terms in this sum if necessary, we may assume that the $a_i \in \genvla'$ are homogeneous vectors. The sum may then be broken up further as
\begin{equation} \label{dsum af}
\sum_n \sum_{i \in I_n} \big( \D a_i \otimes f_i + a_i \otimes \delta f_i \big),
\end{equation}
where $\deg a_i = n$ for each $i \in I_n$.
Furthermore, we may assume without loss of generality that the sets $\{ f_i \}_{i \in I_n}$ are linearly independent for each $n$ (if there is a linear relation among the $f_i$ we may use it to eliminate one of them. This process may be repeated until the remaining $f_i$'s are linearly independent).
Suppose now that \eqref{dsum af} also belongs to $\vla^o \otimes \mathcal{A} \oplus \ker \D \otimes \mathcal{A}^o$, that is
\begin{equation} \label{sum af}
\sum_n \sum_{i \in I_n} \big( \D a_i \otimes f_i + a_i \otimes \delta f_i \big) \in \vla^o \otimes \im \delta \oplus \vla^o \otimes \mathcal{A}^o \oplus \ker \D \otimes \mathcal{A}^o,
\end{equation}
where we have used the decomposition $\mathcal{A} = \mathcal{A}^o \oplus \im \delta$.
Let us also decompose
$f_i = g_i + \delta h_i$ for some unique $g_i \in \mathcal{A}^o$ and $\delta h_i \in \im \delta$. We then have
\begin{equation} \label{sum daf}
\sum_n \sum_{i \in I_n} \big( \D a_i \otimes f_i + a_i \otimes \delta f_i \big) =
\sum_n \sum_{i \in I_n} \D a_i \otimes g_i +
\sum_n \sum_{i \in I_n} \big( \D a_i \otimes \delta h_i + a_i \otimes \delta f_i \big).
\end{equation}
The first double sum on the right belongs to $\genvla \otimes \mathcal{A}^o$ and the second to $\genvla \otimes \im \delta$. More precisely, the first sum belongs to $\im \D \otimes \mathcal{A}^o$. But according to \eqref{sum af} it also belongs to $\vla^o \otimes \mathcal{A}^o \oplus \ker \D \otimes \mathcal{A}^o$, which has vanishing intersection with $\im \D \otimes \mathcal{A}^o$. It follows that this first double sum vanishes. In fact, each term in the sum over $n$ must separately vanish by comparing degrees of the first tensor factor, so for each $n$ we have
\begin{equation} \label{sum Dag}
\sum_{i \in I_n} \D a_i \otimes g_i = 0.
\end{equation}

Consider now the second sum in \eqref{sum daf}, namely
\begin{equation} \label{sum Imd}
\sum_n \sum_{i \in I_n} \big( \D a_i \otimes \delta h_i + a_i \otimes \delta f_i \big),
\end{equation}
which according to \eqref{sum af} belongs to $\vla^o \otimes \im \delta$. Now let $n_{\rm max}$ denote the upper bound in the sum over $n$. Since the operator $\D$ is of degree $1$, we have $\deg (\D a_i) = n_{\rm max} + 1$ for any $i \in I_{n_{\rm max}}$. We conclude that $\sum_{i \in I_{n_{\rm max}}} \D a_i \otimes \delta h_i$ must separately belong to $\vla^o \otimes \im \delta$. It now follows that this sum vanishes since it also belongs to $\im \D \otimes \im \delta$ which has vanishing intersection with $\vla^o \otimes \im \delta$. On the other hand, from \eqref{sum Dag} we also know that $\sum_{i \in I_{n_{\rm max}}} \D a_i \otimes g_i = 0$. Therefore, by definition of $g_i$ and $\delta h_i$ we obtain
\begin{equation*}
\sum_{i \in I_{n_{\rm max}}} \D a_i \otimes f_i = \sum_{i \in I_{n_{\rm max}}} \D a_i \otimes g_i + \sum_{i \in I_{n_{\rm max}}} \D a_i \otimes \delta h_i = 0.
\end{equation*}
Since the $f_i$'s are linearly independent, we conclude that $\D a_i = 0$, or in other words $a_i \in \ker \D$, for each $i \in I_{n_{\rm max}}$. Finally, consider the sum of terms in \eqref{sum Imd} for which the first tensor factor is of degree $n_{\rm max}$, namely
\begin{equation*}
\sum_{i \in I_{n_{\rm max} - 1}} \D a_i \otimes \delta h_i + \sum_{i \in I_{n_{\rm max}}}  a_i \otimes \delta f_i.
\end{equation*}
These two sums respectively belong to $\im \D \otimes \im \delta$ and $\ker \D \otimes \im \delta$, contradicting the fact that \eqref{sum Imd} belongs to $\vla^o \otimes \im \delta$, which has vanishing intersection with $\im \D \otimes \im \delta \oplus \ker \D \otimes \im \delta$.

\subsection{Proof of Proposition \ref{prop: YMYW}}\label{proof: YMYW}
We consider first the map $Y_W$. We shall prove together Theorem \ref{thm: quasi rec} and part (\ref{YWdef}) of Proposition \ref{prop: YMYW}. So consider the map $Y_W$ defined recursively by the conditions in Theorem \ref{thm: quasi rec}. We shall show that it does obey the stated conditions in Proposition \ref{prop: YMYW}. Uniqueness then follows from Proposition \ref{prop: wiiz}. 

We want to show that for all $A\in \ueva$ and all $m_0\in \M_0$,
\begin{equation} \label{Y-map swap 0}
\iota_u [\atp u A \otimes \atp{\bm z}{\bm m} \otimes \atp 0 {m_0}] 
= \big[ \atp{\bm z}{\bm m} \otimes Y_W(A, u) \atp 0 {m_0} \big].
\end{equation}

We proceed by induction on the depth of the state $A$. When $A = \vac$ the result follows from Proposition \ref{prop: VL remove}. For the inductive step, we  assume that \eqref{Y-map swap 0} holds for states of depth strictly less than that of $A$. Without loss of generality we can take the state $A$ to be of the form $A = a(-1) B$ for some $a \in \vla$ and $B \in \ueva$. Indeed, for any $n \in \mathbb{Z}_{\geq 0}$ we have $a(-n-1) = \frac{1}{n!} (\D^n a)(-1)$. By definition of coinvariants we have
\be \big[ f(t) \on (\atp u B \otimes \atp{\bm z}{\bm m} \otimes \atp 0 {m_0}) \big] =0, \quad\text{where}\quad 
f(t) = \sum_{\alpha \in \Gamma} \frac{\alpha^{-1} R_{\alpha} a}{\alpha^{-1} t - u}.
\nn\ee
Thus the left hand side of \eqref{Y-map swap 0} may be written as
\begin{align*}
\iota_u &[\atp u{a(-1) B} \otimes\atp{\bm z}{\bm m} \otimes \atp 0{m_0}]
= \iota_u \bigg[  {\sum_{\alpha \neq 1} \sum_{n \geq 0} \frac{(R_{\alpha} a)(n)}{((\alpha - 1) u)^{n+1}} \atp uB} \otimes \atp{\bm z}{\bm m} \otimes \atp 0 {m_0} \bigg]\\
&+ \iota_u \bigg[ \atp uB \otimes \sum_{i=1}^N \sum_{\alpha \in \Gamma} \sum_{n \geq 0} \frac{(R_{\alpha} a)(n)_{z_i}}{(\alpha u - z_i)^{n+1}} \atp{\bm z}{\bm m} \otimes \atp 0 {m_0} \bigg]
+ \iota_u \bigg[ \atp uB \otimes \atp{\bm z}{\bm m} \otimes \sum_{\alpha \in \Gamma} \sum_{n \geq 0} \frac{(R_{\alpha} a)(n)}{(\alpha u)^{n+1}} \atp 0 {m_0} \bigg].
\end{align*}
By the inductive hypothesis we may `swap' the remaining states at $u$. That is, we have
\begin{align*}
&\iota_u [ a(-1) \atp u B \otimes \atp{\bm z}{\bm m} \otimes \atp 0 {m_0}] = \bigg[ \atp{\bm z}{\bm m} \otimes \sum_{\alpha \neq 1} \sum_{n \geq 0} \frac{1}{((\alpha - 1) u)^{n+1}} Y_W\big( (R_{\alpha} a)(n) B, u \big) \atp 0 {m_0} \bigg]\\
&+ \iota_u \bigg[ \sum_{i=1}^N \sum_{\alpha \in \Gamma} \sum_{n \geq 0} \frac{(R_{\alpha} a)(n)_{z_i}}{(\alpha u - z_i)^{n+1}} \atp{\bm z}{\bm m} \otimes Y_W(B, u) \atp 0 {m_0} \bigg]
+ \bigg[ \atp{\bm z}{\bm m} \otimes Y_W(B, u) Y_W(a(-1)\vac, u)_- \atp 0 {m_0} \bigg].
\end{align*}
The first and last term on the right hand side are already in the desired form. Consider the second term. Taking the $\iota_u$ map explicitly, we may rewrite it as follows
\begin{align} \label{swap 0 xi}
\iota_u &\bigg[ \sum_{i=1}^N \sum_{\alpha \in \Gamma} \sum_{n \geq 0} \frac{(R_{\alpha} a)(n)_{z_i} }{(\alpha u - z_i)^{n+1}}\atp{\bm z}{\bm m} \otimes Y_W(B, u) \atp 0 {m_0} \bigg] \notag\\
&\qquad = - \bigg[ \sum_{i=1}^N \sum_{\alpha \in \Gamma} \sum_{n \geq 0} \sum_{m \geq 0} \bigg( \!\!\! \begin{array}{c} m+n\\ n \end{array} \!\!\! \bigg) \frac{(-1)^n (\alpha u)^m}{z_i^{m+n+1}} (R_{\alpha} a)(n)_{z_i} \atp{\bm z}{\bm m} \otimes Y_W(B, u) \atp 0 {m_0} \bigg].
\end{align}
Now, we have 
\be
\sum_{m \geq 0} u^m \big[ g_m(t) \on (\atp{\bm z}{\bm m} \otimes Y_W(B, u) \atp 0 {m_0}) \big] = 0, \quad\text{where}\quad
g_m(t) = \sum_{\alpha \in \Gamma} \frac{\alpha^{-1} R_{\alpha} a}{(\alpha^{-1} t)^{m+1}}.\nn
\ee
Using this we may rewrite \eqref{swap 0 xi} simply as
\begin{equation*}
\iota_u \bigg[ \sum_{i=1}^N \sum_{\alpha \in \Gamma} \sum_{n \geq 0} \frac{(R_{\alpha} a)(n)_{z_i} }{(\alpha u - z_i)^{n+1}}\atp{\bm z}{\bm m} \otimes Y_W(B, u) \atp 0 {m_0} \bigg] \notag\\
= \bigg[ \atp{\bm z}{\bm m} \otimes Y_W(a(-1)\vac, u)_+ Y_W(B, u) \atp 0 {m_0} \bigg].
\end{equation*}
Putting the above together and using the recursive definition \eqref{YW rec} of the $Y_W$ map we obtain
\begin{equation*}
\iota_u [ a(-1) \atp u B \otimes \atp{\bm z}{\bm m} \otimes \atp 0 {m_0}] = [\atp{\bm z}{\bm m} \otimes Y_W(a(-1) B, u) \atp 0 {m_0}],
\end{equation*}
as required.

\medskip

We now turn to the map $Y_M$. We want to show that the map $Y_M$ defined recursively by \eqref{ymdeff} obeys the stated conditions in Proposition \ref{prop: YMYW}; again, uniqueness follows by Proposition \ref{prop: wiiz}. Indeed, let us show that
\begin{equation} \label{Y-map swap}
\iota_{u - z_i} [\atp u A \otimes \atp{\bm z}{\bm m} \otimes \atp 0 {m_0}]
= \big[ Y_M(A, u - z_i)_{z_i} \atp{\bm z}{\bm m} \otimes \atp 0 {m_0} \big].
\end{equation}
Since the role of the origin in the following proof will be completely analogous to that of the points $z_j$ with $j \neq i$, to alleviate the notational clutter we shall omit this point and simply regard it as one of the points $z_j$, $j \neq i$.

We shall prove \eqref{Y-map swap} by induction on the depth of the state $A$. The result is obvious in the case where $A = \vac$. So let $A \in \ueva$ be a state of strictly positive depth and suppose that \eqref{Y-map swap} holds for any states of depth strictly less than that of $A$.

We can write $A$ in the form $A = a(-1) B$ for some $a \in \vla$ and $B \in \ueva$. Using the $\Gamma$-equivariant rational function
\begin{equation*}
f(t) = \sum_{\alpha \in \Gamma} \frac{\alpha^{-1} R_{\alpha} a}{\alpha^{-1} t - u}
\end{equation*}
to swap the operator $a(-1)$ we obtain
\begin{align} \label{swap2}
&\iota_{u - z_i} [ a(-1) \atp u B \otimes \atp {\bm z}{\bm m}] = \iota_{u - z_i} \bigg[ \atp u B \otimes \sum_{\alpha \in \Gamma} \sum_{n \geq 0} \frac{(R_{\alpha} a)(n)_{z_i}}{(\alpha u - z_i)^{n+1}} \atp {\bm z}{\bm m} \bigg] \notag\\
&+ \iota_{u - z_i} \sum_{j \neq i} \bigg[ \atp u B \otimes \sum_{\alpha \in \Gamma} \sum_{n \geq 0} \frac{(R_{\alpha} a)(n)_{z_j}}{(\alpha u - z_j)^{n+1}} \atp {\bm z}{\bm m} \bigg]
+ \iota_{u - z_i} \bigg[ \sum_{\alpha \neq 1} \sum_{n \geq 0} \frac{(R_{\alpha} a)(n)}{\big( (\alpha - 1) u\big)^{n+1}} \atp u B \otimes \atp {\bm z}{\bm m} \bigg].
\end{align}
To each term on the right we may now apply the inductive hypothesis.
Consider to begin with the first term on the right of \eqref{swap2}. Using the inductive hypothesis for $B$ and separating the $\alpha = 1$ term in the sum over $\alpha$ we may write this term as
\begin{equation*}
\big[ Y_M(B, u - z_i)_{z_i} Y_M(a(-1)\vac, u - z_i)_{z_i}{}_- \atp {\bm z}{\bm m} \big]
+ \iota_{u - z_i} \bigg[ Y_M(B, u - z_i)_{z_i} \sum_{\alpha \neq 1} \sum_{n \geq 0} \frac{(R_{\alpha} a)(n)_{z_i}}{(\alpha u - z_i)^{n+1}} \atp {\bm z}{\bm m} \bigg].
\end{equation*}
Taking the $\iota$-map explicitly in the second term we may rewrite the latter more explicitly as
\begin{equation} \label{Ba i}
- \bigg[ \sum_k \sum_{\alpha \neq 1} \sum_{n \geq 0} \sum_{m \geq 0} \bigg( \!\!\! \begin{array}{c} m+n\\ n \end{array} \!\!\! \bigg) \frac{(-1)^n \alpha^{-n-1} (u - z_i)^{m-k-1}}{(\alpha^{-1} - 1)^{m+n+1} z_i^{m+n+1}} B(k)_{z_i} (R_{\alpha} a)(n)_{z_i} \atp {\bm z}{\bm m} \bigg].
\end{equation}
To reverse the order of the two operators acting on ${\bm m}$ we make use of the commutator relation
\begin{equation*}
\big[ (R_{\alpha} a)(n), B(k) \big] = \sum_{p \geq 0} \bigg( \!\! \begin{array}{c} n\\ p \end{array} \!\! \bigg) \big( (R_{\alpha} a)(p) B \big)(n+k-p).
\end{equation*}
The expression \eqref{Ba i} now takes the form
\begin{align} \label{RaB}
&- \bigg[ \sum_{\alpha \neq 1} \sum_{n \geq 0} \sum_{m \geq 0} \bigg( \!\!\! \begin{array}{c} m+n\\ n \end{array} \!\!\! \bigg) \frac{(-1)^n \alpha^{-n-1} (u - z_i)^m}{(\alpha^{-1} - 1)^{m+n+1} z_i^{m+n+1}} (R_{\alpha} a)(n)_{z_i} Y(B, u - z_i)_{z_i} \atp {\bm z}{\bm m} \bigg]\\
&+ \bigg[ \sum_k \sum_{\alpha \neq 1} \sum_{n \geq 0} \sum_{m \geq 0} \sum_{p=0}^n \bigg( \!\! \begin{array}{c} n\\ p \end{array} \!\! \bigg) \bigg( \!\!\! \begin{array}{c} m+n\\ n \end{array} \!\!\! \bigg) \frac{(-1)^n \alpha^{-n-1} (u - z_i)^{m-k-1}}{(\alpha^{-1} - 1)^{m+n+1} z_i^{m+n+1}} \big( (R_{\alpha} a)(p) B \big)(n+k-p)_{z_i} \atp {\bm z}{\bm m} \bigg]. \notag
\end{align}
For later convenience we shall rewrite the second term in \eqref{RaB} more simply by evaluating one of the sums. Explicitly, we start by replacing the sum over $n \geq 0$ and $0 \leq p \leq n$, with a sum over $p \geq 0$ and the new variable $q = n - p \geq 0$. This yields
\begin{equation*}
\bigg[ \sum_k \sum_{\alpha \neq 1} \sum_{q \geq 0} \sum_{m \geq 0} \sum_{p \geq 0} \frac{(m+p+q)!}{m! \, p! \, q!} \frac{(-1)^{m+1} \alpha^m (u - z_i)^{m-k-1}}{(\alpha - 1)^{m+q+p+1} z_i^{m+q+p+1}} \big( (R_{\alpha} a)(p) B \big)(q+k)_{z_i} \atp {\bm z}{\bm m} \bigg].
\end{equation*}
Performing the change of variable $k \to k - q$ for each term in the sum over $q$ and then replacing the sum over $q, m \geq 0$, by a sum over $r = m + q \geq 0$ and $0 \leq m \leq r$ we find
\begin{equation*}
\bigg[ \sum_k \sum_{\alpha \neq 1} \sum_{r \geq 0} \sum_{m=0}^r \sum_{p \geq 0} \bigg( \!\! \begin{array}{c} r\\ m \end{array} \!\! \bigg) \bigg( \!\!\! \begin{array}{c} r+p\\ r \end{array} \!\!\! \bigg) \frac{(-1)^{m+1} \alpha^m (u - z_i)^{r-k-1}}{(\alpha - 1)^{r+p+1} z_i^{r+p+1}} \big( (R_{\alpha} a)(p) B \big)(k)_{z_i} \atp {\bm z}{\bm m} \bigg].
\end{equation*}
Finally, taking the sum over $m$ explicitly, the second term in \eqref{RaB} can be rewritten more simply as
\begin{equation} \label{RaBsimp}
\bigg[ \sum_k \sum_{\alpha \neq 1} \sum_{r \geq 0} \sum_{p \geq 0} \bigg( \!\!\! \begin{array}{c} r+p\\ p \end{array} \!\!\! \bigg) \frac{(-1)^{r+1} (u - z_i)^{r-k-1}}{(\alpha - 1)^{p+1} z_i^{r+p+1}} \big( (R_{\alpha} a)(p) B \big)(k)_{z_i} \atp {\bm z}{\bm m} \bigg].
\end{equation}

Next, consider the second term on the right hand side of \eqref{swap2}. Using the inductive hypothesis for the state $B$ it reads
\begin{align*}
\iota_{u - z_i} &\sum_{j \neq i} \bigg[ \sum_{\alpha \in \Gamma} \sum_{n \geq 0} \frac{(R_{\alpha} a)(n)_{z_j}}{(\alpha u - z_j)^{n+1}} Y_M(B, u - z_i)_{z_i} \atp {\bm z}{\bm m} \bigg]\\
&= - \sum_{j \neq i} \bigg[ \sum_{\alpha \in \Gamma} \sum_{n \geq 0} \sum_{m \geq 0} \bigg( \!\!\! \begin{array}{c} m+n\\ n \end{array} \!\!\! \bigg) \frac{(-1)^n \alpha^{-n-1} (u - z_i)^m}{(\alpha^{-1} z_j - z_i)^{m+n+1}} (R_{\alpha} a)(n)_{z_j} Y_M(B, u - z_i)_{z_i} \atp {\bm z}{\bm m} \bigg].
\end{align*}
Here we have used the fact that operators acting at sites $i$ and $j$, with $i \neq j$, commute.

Finally, consider the last term on the right hand side of \eqref{swap2}. Since $n \geq 0$, the depth of the state $(R_{\alpha} a)(n) B$ in $U(\U^-(\vla)) \vac$ is at most equal to that of $B$. Therefore, we can apply the inductive hypothesis to the last term in \eqref{swap2} as well, which therefore reads
\begin{align*}
\iota_{u - z_i} &\bigg[ \sum_{\alpha \neq 1} \sum_{p \geq 0} \frac{1}{\big( (\alpha - 1) u\big)^{p+1}} Y_M\big( (R_{\alpha} a)(p) B, u - z_i \big)_{z_i} \atp {\bm z}{\bm m} \bigg]\\
&= - \bigg[ \sum_k \sum_{\alpha \neq 1} \sum_{p \geq 0} \sum_{m \geq 0} \bigg( \!\!\! \begin{array}{c} m+p\\ p \end{array} \!\!\! \bigg) \frac{(-1)^{m+1} (u - z_i)^{m-k-1}}{(\alpha - 1)^{p+1} z_i^{m+p+1}} \big( (R_{\alpha} a)(p) B \big)(k)_{z_i} \atp {\bm z}{\bm m} \bigg].
\end{align*}
However, this is exactly the opposite of the expression \eqref{RaBsimp} which corresponds to the second term in \eqref{RaB}. After cancelling these terms and putting all of the above together, we obtain
\begin{align} \label{swap3}
&\iota_{u - z_i} [ a(-1) \atp u B \otimes \atp{\bm z}{\bm m}] = \big[ Y_M(B, u - z_i)_{z_i} Y_M(a(-1)\vac, u - z_i)_{z_i}{}_- \atp {\bm z}{\bm m} \big] \notag\\
&- \bigg[ \sum_{\alpha \neq 1} \sum_{n \geq 0} \sum_{m \geq 0} \bigg( \!\!\! \begin{array}{c} m+n\\ n \end{array} \!\!\! \bigg) \frac{(-1)^n \alpha^{-n-1} (u - z_i)^m}{(\alpha^{-1} - 1)^{m+n+1} z_i^{m+n+1}} (R_{\alpha} a)(n)_{z_i} Y_M(B, u - z_i)_{z_i} \atp {\bm z}{\bm m} \bigg] \notag\\
&- \sum_{j \neq i} \bigg[ \sum_{\alpha \in \Gamma} \sum_{n \geq 0} \sum_{m \geq 0} \bigg( \!\!\! \begin{array}{c} m+n\\ n \end{array} \!\!\! \bigg) \frac{(-1)^n \alpha^{-n-1} (u - z_i)^m}{(\alpha^{-1} z_j - z_i)^{m+n+1}} (R_{\alpha} a)(n)_{z_j} Y_M(B, u - z_i)_{z_i} \atp {\bm z}{\bm m} \bigg].
\end{align}

Now consider the following rational function for $m \geq 0$,
\begin{equation*}
g_m(t) = \sum_{\alpha \in \Gamma} \frac{\alpha^{-1}  R_{\alpha} a}{(\alpha^{-1} t - z_i)^{m+1}}.
\end{equation*}
Its expansion in $t - z_i$ reads
\begin{align*}
\iota_{t - z_i} g_m(t)
&= a(-m-1)_{z_i} + \sum_{\alpha \neq 1} \sum_{n \geq 0} \bigg(\!\!\! \begin{array}{c} m+n\\ n \end{array} \!\!\!\bigg) \frac{(-1)^n \alpha^{-n-1} (R_{\alpha} a)(n)_{z_i}}{(\alpha^{-1} - 1)^{m+n+1} z_i^{m+n+1}}.
\end{align*}
Likewise, the expansion of $g_m(t)$ in $t - z_j$, for $j \neq i$, takes the form
\begin{align*}
\iota_{t - z_j} g_m(t)
&= \sum_{\alpha \in \Gamma} \sum_{n \geq 0} \bigg(\!\!\! \begin{array}{c} m+n\\ n \end{array} \!\!\!\bigg) \frac{(-1)^n \alpha^{-n-1} (R_{\alpha} a)(n)_{z_j}}{(\alpha^{-1} z_j - z_i)^{m+n+1}}.
\end{align*}
The invariance of the class of ${\bm m}$ in $\otimes_{i=1}^N M_{(i)} \big/ \U^{\Gamma}_{\bm z}(\vla)$ under the rational function $g_m(t)$ for each $m \geq 0$ implies
\begin{equation*}
\sum_{m \geq 0} (u - z_i)^m \big[ g_m(t) \cdot \atp{\bm z}{\bm m} \big] = 0.
\end{equation*}
It follows that the last two terms in \eqref{swap3} can be obtained by swapping from
\begin{equation*}
\sum_{m \geq 0} (u - z_i)^m \big[ a(-m-1)_{z_i} Y_M(B, u - z_i)_{z_i} \atp{\bm z}{\bm m} \big].
\end{equation*}
Finally, using the notation $F(x)_+ = \sum_{n < 0} F_n \, x^{-n-1}$ for any field $F(x) = \sum_{n \in \mathbb{Z}} F_n \, x^{-n-1}$, we can now rewrite \eqref{swap3} simply as
\begin{align*}
\iota_{u - z_i} [ a(-1) \atp u B \otimes \atp{\bm z}{\bm m}] &= \big[ Y_M(B, u - z_i)_{z_i} Y_M(a(-1)\vac, u - z_i)_{z_i}{}_- \atp{\bm z}{\bm m} \big]\\ 
&\qquad\qquad + \big[ Y_M(a(-1)\vac, u - z_i)_{z_i}{}_+ Y_M(B, u - z_i)_{z_i} \atp{\bm z}{\bm m} \big].
\end{align*}
The result now follows from the expression \eqref{ynord} for $Y_M(a(-1)B, u - z_i)$ in terms of the normal ordered product of $Y_M(a(-1)\vac, u - z_i)$ and $Y_M(B, u - z_i)$.

\subsection{Proof of Lemma \ref{lem: alpha u}}\label{proof: alpha u}
To minimize notational clutter, we shall show that 
\begin{equation} \label{insertion point}
[\atp u A \otimes \atp z m ] = [ R_{\alpha} \atp {\alpha u}A \otimes \atp z m ]
\end{equation}
but it will be clear that the argument runs in the same way if one includes more marked points, including a marked point at the origin.

We use induction on the depth of the state $A$ in $\ueva \cong_{\C} U(\U^-(\vla))\vac$. The result is trivial when $A = \vac$, so consider a state $A \in \ueva$ not proportional to the vacuum and suppose the result \eqref{insertion point} holds for all states of depth strictly less than that of $A$.

It is enough to consider states of the form  $A= a(-1) B$ for some $a \in \vla$ and $B \in \ueva$. We compute both sides of \eqref{insertion point}. On the one hand we have
\begin{equation} \label{swap u}
[ a(-1) \atp u B \otimes \atp z m] = \bigg[ \atp u B \otimes \sum_{\beta \in \Gamma} \sum_{n \geq 0} \frac{(R_{\beta} a)(n)}{(\beta u - z)^{n+1}} \atp z m \bigg]
+ \bigg[ \sum_{\beta \neq 1} \sum_{n \geq 0} \frac{(R_{\beta} a)(n)}{((\beta - 1) u)^{n+1}} \atp u B \otimes \atp z m \bigg],
\end{equation}
where to `swap' $a(-1)$ we used the following $\Gamma$-equivariant rational function
\begin{equation*}
f(t) = \sum_{\beta \in \Gamma} \frac{\beta^{-1} R_{\beta} a}{\beta^{-1} t - u}.
\end{equation*}
On the other hand, the singular term in the expansion of this rational function in $t - \alpha u$ takes the form $f(t) \sim \frac{R_{\alpha} a}{t - \alpha u}$, which represents the operator $(R_{\alpha} a)(-1)$ acting on the module at $\alpha u$. Therefore, using the \emph{same} rational function we also have
\begin{align} \label{swap au}
[(R_{\alpha} a)(-1) (R_{\alpha} \atp{\alpha u}B) \otimes \atp z m] = \bigg[ R_{\alpha} \atp {\alpha u}B \otimes \sum_{\beta \in \Gamma} \sum_{n \geq 0} \frac{(R_{\beta} a)(n)}{(\beta u - z)^{n+1}} \atp z m \bigg] + \bigg[ \sum_{\beta \neq \alpha} \sum_{n \geq 0} \frac{(R_{\beta} a)(n)}{((\beta - \alpha) u)^{n+1}} R_{\alpha} \atp{\alpha u}B \otimes \atp z m \bigg],
\end{align}
Now for any two states $A, B \in \ueva$ the map $R_{\alpha}$ has the property \eqref{Gamma eq VA}.
In particular, it follows from this that
\begin{align} \label{rel R}
R_{\alpha} \big( a(n) B \big) &= R_{\alpha} \big( (a(-1)\vac)_{(n)} B \big) = \alpha^{-n-1} \big(R_{\alpha}(a(-1)\vac) \big)_{(n)} (R_{\alpha} B) \notag\\
&= \alpha^{-n-1} \big( (R_{\alpha} a)(-1)\vac \big)_{(n)} (R_{\alpha} B) = \alpha^{-n-1} (R_{\alpha} a)(n) (R_{\alpha} B),
\end{align}
where in the second last equality we have used the fact that $\deg (a(-1)\vac) = \deg a$.
Replacing $a$ by $R_{\beta \alpha^{-1}} a$ in this relation we may use it to rewrite \eqref{swap au} as
\begin{align*}
[(R_{\alpha} a)(-1) (R_{\alpha} \atp{\alpha u}B) \otimes \atp z m] 
&= \bigg[ R_{\alpha} \atp {\alpha u}B \otimes \sum_{\beta \in \Gamma} \sum_{n \geq 0} \frac{(R_{\beta} a)(n)}{(\beta u - z)^{n+1}} \atp z m \bigg] \notag\\
&\qquad\qquad\qquad + \bigg[ \sum_{\gamma \neq 1} \sum_{n \geq 0} \frac{1}{((\gamma - 1) u)^{n+1}}R_{\alpha} \big( (R_{\gamma} a)(n) \atp{\alpha u}B \big) \otimes \atp z m \bigg],
\end{align*}
where $\gamma = \beta \alpha^{-1}$. Now, using the inductive hypothesis, the right hand side of the above is clearly equal to the right hand side of \eqref{swap u}. The result thus follows using the relation \eqref{rel R} with $n = -1$ which implies that $(R_{\alpha} a)(-1) (R_{\alpha} B) = R_{\alpha} \big( a(-1) B \big)$.

\subsection{On formal variables}
Let $x$ and $y$ be formal variables and consider the ring of polynomials $\C[x, y]$. In formulating Proposition \ref{prop: borcherds} we make use of the localisation of $\C[x,y]$  by the multiplicative subset consisting of elements of the form $x^m (x-y)^n y^k$ for $m, n, k \in \Z_{\geq 0}$. We denote the resulting ring by $\C[x^{\pm 1}, y^{\pm 1}, (x - y)^{-1}]$.

Furthermore, in the proofs of Propositions \ref{prop: borcherds} and \ref{prop: com} we also make use of the following lemma concerning the localisation $S^{-1} \C[x, y]$ of $\C[x, y]$ by the multiplicative subset $S = \C[x] \C[y] \setminus \{ 0 \}$.
\begin{lem}\label{lem: iotacom}
For any $F \in S^{-1} \C[x, y]$ we have $\iota_{x, y} F = \iota_{y, x} F$.
\end{lem}
\begin{proof}
Any element of $S^{-1} \C[x, y]$ takes the form $r^{-1} s^{-1} q$ with $r \in \C[x]$, $s \in \C[y]$ and $q \in \C[x, y]$. By writing $q$ as a finite linear combination of monomials $x^m y^n$, $m, n \in \Z_{\geq 0}$, we can express $r^{-1} s^{-1} q$ as a finite sum of products $f g$ for some $f \in \C(x)$ and $g \in \C(y)$, where $\C(x)$ and $\C(y)$ denote the field of fractions of $\C[x]$ and $\C[y]$ respectively. Therefore, by linearity of the maps $\iota_x$ and $\iota_y$, it suffices to prove the statement for such products $F = f g$. But in this case the result is obvious.
\end{proof}

\subsection{Proof of Proposition \ref{prop: borcherds}}\label{sec: borcherdsproof}
We prove part (2), making use of Proposition \ref{prop: wiiz}. The proof of part (1) is similar, and more standard since it essentially does not involve the twisting by $\Gamma$. 

We begin by considering the class
\begin{equation*}
[\atp x A \otimes \atp y B \otimes \atp z m \otimes \atp 0{m_0}] \in \ueva \otimes \ueva \otimes  \M_z\otimes \M_0\Big/ \U^\Gamma_{x,y,z,0}(\vla).
\end{equation*}
Let $f$ be any rational function of $x$ and $y$ with poles at most at  $x = 0$, $y = 0$ and $x = y$. By virtue of Proposition \ref{prop: VL rat} we can regard 
\begin{equation*}
f(x, y) [\atp x A \otimes \atp y B \otimes \atp z m \otimes \atp 0 {m_0}]
\end{equation*}
as a rational function in $x$ valued in the space of coinvariants $\ueva \otimes  \M_z\otimes \M_0\big/ \U^\Gamma_{y,z,0}(\vla)$. It has poles at most at the points $0$, $\alpha y$, $\alpha\in \Gamma$, and $\alpha z$, $\alpha \in \Gamma$. 
If we let
\be p(x,y) := \prod_{\alpha \in \Gamma \setminus \{ 1 \}} (x - \alpha y) = (x^T - y^T)/(x - y) \nn\ee 
then for some sufficiently large $k \in \Z_{\geq 0}$, the rational function $p(x, y)^k f(x, y) [A \otimes B \otimes m \otimes m_0]$ in $x$ will have poles only at $0$, $\8$, $y$ and $\alpha z$, $\alpha\in \Gamma$. 
So by the residue theorem  we obtain
\begin{multline*}
0 = \res_{x - y} p(x,y)^k \iota_{x - y} f(x,y) [ \atp x A \otimes \atp y B \otimes \atp z m \otimes \atp 0 {m_0}]\\
+ \res_x p(x,y)^k \iota_x f(x,y) [ \atp x A \otimes \atp y B \otimes \atp z m \otimes \atp 0 {m_0}]\\
+ \left(\sum_{\alpha\in \Gamma}\res_{x-\alpha z} \iota_{x-\alpha z}-\res_{x^{-1}}x^2\iota_{x^{-1}}\right) p(x,y)^k  f(x,y) [ \atp x A \otimes \atp y B \otimes \atp z m \otimes \atp 0 {m_0}].
\end{multline*}
Hence, by Proposition \ref{prop: YMYW} and Corollary \ref{Ydef}, we have the equality
\begin{multline*}
0= \big[ \big(\res_{x - y} p(x,y)^k \iota_{x - y} f(x,y) Y(A, x - y) \atp y B\big) \otimes \atp z m \otimes \atp 0 {m_0}\big]\\
+ \big[\atp y B\otimes \atp z m \otimes \big(\res_x p(x,y)^k \iota_x f(x,y) \otimes Y_W(A, x) \atp 0 {m_0}\big)\big]\\
+  \left(\sum_{\alpha\in \Gamma}\res_{x-\alpha z} \iota_{x-\alpha z}-\res_{x^{-1}}x^2\iota_{x^{-1}}\right) p(x,y)^k  f(x,y) [\atp x A \otimes \atp y B \otimes \atp z m \otimes \atp 0 {m_0}].
\end{multline*} 
Next we may apply $\iota_y$ and then use Proposition \ref{prop: YMYW} and Corollary \ref{Ydef} once more, to find the equality
\begin{multline*}
0 = \big[ \atp z m \otimes \big( \res_{x - y} p(x,y)^k \iota_{y, x - y} f(x,y) Y_W(Y(A, x - y) B, y) \atp 0 {m_0}\big)\big]\\
+ \big[ \atp z m\otimes \big(\res_x p(x,y)^k \iota_{y, x} f(x,y) Y_W(B,y) Y_W(A, x) \atp 0 {m_0}\big)\big]\\
+  \left(\sum_{\alpha\in \Gamma}\res_{x-\alpha z} \iota_{y,x-\alpha z}-\res_{x^{-1}}x^2\iota_{y,x^{-1}}\right) p(x,y)^k  f(x,y) [ \atp x A \otimes \atp y B \otimes \atp z m \otimes \atp 0 {m_0}].
\end{multline*}
of formal Laurent series in $y$. Consider the final line here, and note the following fact:  whenever $F(u,v)$ is a rational function with no pole at $u=v$ then $\iota_{u,v}F(u,v) = \iota_{v,u}F(u,v)$. (See Lemma \ref{lem: iotacom} above.)  Now the function $p(x,y)^k f(x,y) [A \otimes B \otimes m \otimes m_0]$ has no pole at $y=x-\alpha z$ for any $\alpha\in \Gamma$ and no pole at $y=x^{-1}$. So the final line above is equal to
\begin{multline}\label{ltt}
\left(\sum_{\alpha\in \Gamma}\res_{x-\alpha z} \iota_{x-\alpha z}-\res_{x^{-1}}x^2\iota_{x^{-1}}\right)\iota_y p(x,y)^k  f(x,y) [\atp x A \otimes \atp y B \otimes \atp z m \otimes \atp 0 {m_0}]\\
= \left(\sum_{\alpha\in \Gamma}\res_{x-\alpha z} \iota_{x-\alpha z}-\res_{x^{-1}}x^2\iota_{x^{-1}}\right)\iota_y p(x,y)^k  f(x,y) [ \atp x A \otimes \atp z m \otimes Y_W(B,y) \atp 0 {m_0}]
\end{multline}
Now consider the following  formal Laurent series in $y$ with coefficients in $\ueva \otimes \M_z \otimes \M_0 \big/ \U^\Gamma_{x,z,0}(\vla)$:
\begin{equation*}
\iota_y  p(x,y)^k f(x,y) \big[ \atp x A \otimes \atp z m \otimes Y_W(B,y) \atp 0 {m_0} \big].
\end{equation*}
At each order in $y$ it is a rational function of $x$ with poles at most at $0$, $\8$, and $\alpha z$, $\alpha\in \Gamma$. On applying the residue theorem and Proposition \ref{prop: YMYW} order-by-order in $y$ one has
\begin{multline*}
0= \left(\sum_{\alpha\in \Gamma}\res_{x-\alpha z} \iota_{x-\alpha z}-\res_{x^{-1}}x^2\iota_{x^{-1}}\right)\iota_y p(x,y)^k  f(x,y) [\atp x A \otimes \atp z m \otimes Y_W(B,y) \atp y {m_0}]\\
+ \res_{x} p(x,y)^k \iota_{x,y} f(x,y) \big[\atp z m \otimes Y_W(A,x) Y_W(B,y) \atp 0 {m_0} \big]. 
\end{multline*}
Putting the above equations together, we have shown that
\begin{multline*} 
0 = \big[ \atp z m\otimes \big( \res_{x - y} p(x,y)^k \iota_{y, x - y} f(x,y) Y_W(Y(A, x - y) B, y) \atp 0 {m_0}\big)\big]\\
+ \big[ \atp z m\otimes \big(\res_x p(x,y)^k \iota_{y, x} f(x,y) Y_W(B,y) Y_W(A, x) \atp 0 {m_0}\big)\big]\\
- \big[ \atp z m\otimes \big(\res_x p(x,y)^k \iota_{x, y} f(x,y) Y_W(A,x) Y_W(B, y) \atp 0 {m_0}\big)\big].
\end{multline*}
By Proposition \ref{prop: wiiz} this is enough to establish the result.

\subsection{Proof of Proposition \ref{prop: com}}\label{sec: proofofcommutator}
In the space of coinvariants
\be \left(\ueva \otimes \ueva \otimes  \M_{(z)}\otimes \M_0\right)\bigg/ \U^\Gamma_{x,y,z,0}(\vla),\nn\ee
we consider the element
\begin{equation*}
x^m [ \atp x A \otimes \atp y B \otimes \atp z m \otimes \atp 0 {m_0}]
\end{equation*}
for any $m\in \Z$. Viewed as a function of $x$, this is rational and has poles at most at $\alpha y$, $\alpha\in \Gamma$, at $\alpha z$, $\alpha\in \Gamma$, and at $0$ and $\8$. 
Thus, by the residue theorem,
\be
0= \sum_{\alpha\in \Gamma} \res_{x-\alpha y} \iota_{x-\alpha y} x^m [ \atp x A \otimes \atp y B \otimes \atp z m \otimes \atp 0 {m_0}]
+ \res_{x}  x^m \iota_{x} [ \atp x A \otimes \atp y B \otimes \atp z m \otimes \atp 0 {m_0}] + R  \nn
\ee
where 
\begin{equation*}
R = \left(\sum_{\alpha\in \Gamma} \res_{x-\alpha z} \iota_{x-\alpha z}- \res_{x^{-1}} x^2\iota_{x^{-1}}\right) x^m [ \atp x A \otimes \atp y B \otimes \atp z m \otimes \atp 0 {m_0}]
\end{equation*}
We have, by Lemma \ref{lem: alpha u} and then Proposition \ref{prop: YMYW},
\begin{align*}
\iota_{x-\alpha y} [ \atp x A \otimes \atp y B \otimes \atp z m \otimes \atp 0 {m_0}]
  &= \iota_{x-\alpha y} [ \atp x A \otimes R_\alpha \atp {\alpha y}B \otimes \atp z m \otimes \atp 0 {m_0}]\\ 
  &= [ Y(A,x-\alpha y) R_\alpha \atp {\alpha y}B \otimes \atp z m \otimes \atp 0 {m_0}],
\end{align*}
and by definition $\iota_{x-\alpha y} x^m = \sum_{k=0}^\8 \binom m k (\alpha y)^{m-k} (x-\alpha y)^k$. 
Therefore
\begin{multline*}
0= \sum_{k=0}^\8 \binom m k \sum_{\alpha\in \Gamma}  (\alpha y)^{m-k} \res_{x-\alpha y} (x-\alpha y)^k [ Y(A,x-\alpha y) R_\alpha \atp {\alpha y}B \otimes \atp z m \otimes \atp 0 {m_0}] \\
  + \res_{x}  x^m [ \atp y B \otimes \atp z m \otimes Y_W(A,x) \atp 0 {m_0}] + R.
\end{multline*}
We use Proposition \ref{prop: YMYW} here, and again in the following, to find, on taking $\iota_y$,
\begin{multline*}
0= \sum_{k=0}^\8 \binom m k \sum_{\alpha\in \Gamma} (\alpha y)^{m-k}  \res_{x-\alpha y} (x-\alpha y)^k [ \atp z m \otimes Y_W(Y(A,x-\alpha y) R_\alpha B, \alpha y) \atp 0 {m_0}] \\
  + \res_{x}  x^m  [ \atp z m \otimes Y_W(B,y) Y_W(A,x) \atp 0 {m_0}]
  + \iota_yR.
\end{multline*}
Arguing as in the previous proof, cf. \eqref{ltt}, we have by the residue theorem that
\be
\iota_y R 
= -\res_x x^m [ \atp z m \otimes Y_W(A,x) Y_W(B,y) \atp 0 {m_0}]. \nn
\ee
Hence 
\begin{multline*}
0= \sum_{k=0}^\8 \binom m k  \sum_{\alpha\in \Gamma} (\alpha y)^{m-k} \res_{x-\alpha y} (x-\alpha y)^k [ \atp z m \otimes Y_W(Y(A,x-\alpha y) R_\alpha B, \alpha y) \atp 0 {m_0}] \\
  + \res_{x}  x^m  [ \atp z m \otimes Y_W(B,y) Y_W(A,x) \atp 0 {m_0}]   - \res_{x}  x^m [ \atp z m \otimes Y_W(A,x) Y_W(B,y) \atp 0 {m_0}]
\end{multline*}
and so, in view of Proposition \ref{prop: wiiz},
\be \res_x x^m \left[  Y_W(A,x) ,Y_W(B,y) \right] = 
\sum_{k=0}^\8 \binom m k (\alpha y)^{m-k} \sum_{\alpha\in \Gamma} \res_{x-\alpha y} (x-\alpha y)^k Y_W(Y(A,x-\alpha y) R_\alpha B, \alpha y).\nn
\ee
Taking $\res_y y^n$ of this equality, we have finally
\be \left[ A^W_{(m)}, B^W_{(n)} \right]
= \sum_{k=0}^\8 \binom m k    \sum_{\alpha\in \Gamma}\alpha^{-n-1}   \left( A_{(k)} (R_\alpha B) \right)^W_{(m+n-k) }.\nn
\ee

\def\cprime{$'$}
\providecommand{\bysame}{\leavevmode\hbox to3em{\hrulefill}\thinspace}
\providecommand{\MR}{\relax\ifhmode\unskip\space\fi MR }
\providecommand{\MRhref}[2]{%
  \href{http://www.ams.org/mathscinet-getitem?mr=#1}{#2}
}
\providecommand{\href}[2]{#2}


\begin{thebibliography}{MY12b}

\bibitem[Bor86]{Borcherds}
R.~Borcherds, 
\emph{Vertex operator algebras, Kac-Moody algebras and the monster}, 
Proc. Nat. Acad. Sci. \textbf{83} (1986) no. 10, 3068--71.

\bibitem[BPZ84]{BPZ}
  A.~A.~Belavin, A.~M.~Polyakov and A.~B.~Zamolodchikov,
  \emph{Infinite conformal symmetry in two-dimensional quantum field theory},
  Nucl. Phys. B \textbf{241} (1984) no. 2, 333--380.



\bibitem[DLM02]{Dong02vertexlie}
  C.~Dong, H.~Li and G.~Mason,
  \emph{Vertex Lie algebras, vertex Poisson algebras and vertex algebras},
   In \emph{Recent developments in infinite-dimensional Lie algebras and conformal field theory},
   of {Contemp. Math}, 2002.

\bibitem[EFK07]{EFKbook}
  P.~I.~Etingof, I.~B.~Frenkel, A.~A.~Kirillov, Jr.,
  \emph{Lectures on representation theory and Knizhnik-Zamolodchikov equations},
  Mathematical Surveys and Monographs \textbf{58} (1998).

\bibitem[FFR94]{FFR}
  B.~Feigin, E.~Frenkel and N.~Reshetikhin, \emph{
  Gaudin model, Bethe ansatz and critical level},
  Comm. Math. Phys. \textbf{166} no. 1 (1994), 27--62.

\bibitem[Fre07]{Fre07}
  E.~Frenkel,
  \emph{Langlands Correspondence for Loop Groups},
  (Cambridge studies in advanced mathematics 103)
  Cambridge University Press, 2007.

\bibitem[FB04]{FB}
  E.~Frenkel and D.~Ben-Zvi,
  \emph{Vertex algebras and algebraic curves},
  (Mathematical surveys and monographs. vol 88) 
  American Mathematical Society, 2004 (Second Edition).
  
\bibitem[FS04]{FS}
  E.~Frenkel and M.~Szczesny,
  \emph{Twisted modules over vertex algebras on algebraic curves},
Adv. Math. \textbf{187} (2004) 195--227.

\bibitem[FHL93]{FHL}
I.~B.~Frenkel, Y.~Z.~Huang, and J.~Lepowsky,
\emph{On axiomatic approaches to vertex operator algebras and modules}, 
Memoirs of the AMS \textbf{104} (1993), no. 494, viii+64.

\bibitem[FLM88]{FLM} 
I.~B.~Frenkel, J.~Lepowsky and A.~Meurman,
\emph{Vertex operator algebras and the Monster}, 
Pure and Applied Mathematics \textbf{134} (1988), Academic Press.




\bibitem[Kac98]{KacVertex}
  V.~G.~Kac,
  \emph{Vertex Algebras for Beginners},
  volume 10 of University Lecture Series. AMS, Providence, RI, second edition, 1998.

\bibitem[KZ84]{KZ}
  V.~G.~Knizhnik and A.~B.~Zamolodchikov,
  \emph{Current Algebra and Wess-Zumino Model in Two-Dimensions},
  Nucl. Phys. B \textbf{247} (1984), 83--103.

\bibitem[Li06a]{Li4}
  H.~Li,
 \emph{A new construction of vertex algebras and quasi modules for vertex algebras},
Adv. Math. \textbf{202} (2006), no 1, 232–-286. 

\bibitem[Li06]{Li3}
  H.~Li,
 \emph{On certain generalisations of twisted affine Lie algebras and quasi modules for $\Gamma$-vertex algebras},
J. Pure and Appl. Algebra \textbf{209} (2007), no 3., 853–-871.

\bibitem[Li06b]{Li5}
  H.~Li,
  \emph{Twisted modules and quasi-modules for vertex operator algebras},
in Proceedings of the International Conference in honor of Professors
James Lepowsky and Robert Wilson, Contemp. Math. \textbf{422} (2007),
389--400.

\bibitem[LL04]{LLbook}
  J.~Lepowsky and H.~Li,
  \emph{Introduction to Vertex Operator Algebras and Their Representations},
  Progress in mathematics,
  Birkh\"auser, 2004.

\bibitem[Pr99]{Primc}
  M.~Primc,
 \emph{Vertex algebras generated by Lie algebras},
  J. Pure Appl. Algebra, \textbf{135} (1999) no. 3, 253--293.

\bibitem[Rot09]{Rotman}
  J.~J.~Rotman,
 \emph{An Introduction to Homological Algebra} (Second Edition),
Springer, 2009

\bibitem[Sz02]{Szc}
M.~Szczesny,
\emph{Wakimoto modules for twisted affine Lie algebras}, Math. Res. Lett. \textbf 9 (2002), 433--448

\bibitem[TUY89]{TUY}
A.~Tsuchiya, K.~Ueno and Y.~Yamada,
\emph{Conformal field theory on universal family of stable curves with gauge symmetries},
Adv. Studies in Pure Math. 19, 459--566, Academic Press (1989).

\bibitem[VY14]{VY}
B.~Vicedo and C.~Young,
\emph{Cyclotomic Gaudin models: construction and Bethe ansatz},
preprint [arXiv:1409.6937]. 

\end{thebibliography}
\end{document}